\documentclass[a4paper]{amsart}
\usepackage{amssymb, bbm}

\usepackage[utf8]{inputenc} 
\usepackage[T1]{fontenc}

\usepackage[hidelinks]{hyperref} % no ugly link borders
   \hypersetup{final} % to have those links while draft option is activated

\usepackage{mathtools}
\mathtoolsset{centercolon} % makes ' := ' nice	

\usepackage{graphicx, tikz}

\newtheorem{lem}{Lemma}
\newtheorem{prop}[lem]{Proposition}
\newtheorem{cor}[lem]{Corollary}
\newtheorem{thm}[lem]{Theorem}
\newtheorem{conj}[lem]{Conjecture}

\theoremstyle{remark}
\newtheorem{rmk}[lem]{Remark}

\def\Var{\mathrm{Var}}
\def\er{{\mathbb R}}
\def\en{{\mathbb N}}

\def\Ex{{\mathbb E}}
\def\Pr{\mathbb P}

\def\zet{{\mathbb Z}}

\def\conv{\mathrm{conv}}

\def\ve{\varepsilon}
\def\ind{\mathbbm{1}}

\newcommand*{\Log}{\operatorname{Log}} 
\def\supp{\mathrm{supp}}

\makeatletter
\@namedef{subjclassname@2020}{%
  \textup{2020} Mathematics Subject Classification}
\makeatother


\begin{document}

\title[Operator norms of Gaussian matrices]{Operator $\ell_p\to\ell_q$ norms of Gaussian matrices}

\author[R. Lata{\l}a]{Rafa{\l} Lata{\l}a}
\address{University of Warsaw, Institute of Mathematics,  Banacha 2, 02--097 Warsaw, Poland.}
\email{rlatala@mimuw.edu.pl}

\author[M. Strzelecka]{Marta Strzelecka}
\address{University of Warsaw, Institute of Mathematics, Banacha 2, 02--097 Warsaw, Poland.}
\email{martast@mimuw.edu.pl}

\begin{abstract}
We confirm the conjecture posed by Gu{\'e}don, Hinrichs, Litvak, and Prochno in 2017 that 
$\Ex\|(a_{ij}g_{ij})_{i\le m, j\le n}\colon \ell_p^n \to \ell_q^m\|$ is comparable, 
up to constants depending only on $p$ and $q$, to
\[
\max_i \|(a_{ij})_j\|_{p^*} +\max_j \|(a_{ij})_i\|_{q} +\Ex \max_{i,j} |a_{ij}g_{ij}|
\]
provided that $1\le p \le 2\le q \le \infty$. This was known before only in the case $p=1$ or $q=\infty$, and in the spectral case $p=2=q$. We also reprove the conjecture in the case $p=2=q$ without using spectral theory (which was employed in the previously known proof).
\end{abstract}

\maketitle

%%%%%%%%%%%%%%%%%%%%%%%%%%%%
%%%%%%%%%%%%%%%%%%%%%%%%%%%%%%
\section{Introduction}

Let $A=(a_{ij})_{i\le m, j\le n}$ be a deterministic $m\times n$ matrix and let 
$p,q\in[1,\infty]$. In this paper we study $\ell_p^n\to  \ell_q^m$ norms of centered 
structured Gaussian random matrices $G_A=(a_{ij}g_{ij})_{i\le m,j\le n}$ with a variance  
profile $A\circ A=(a_{ij}^2)_{i\le m, j\le n}$, i.e., quantities of the form

\[
\|G_A\|_{p\to q}=\|G_A \colon \ell_p^n\to\ell_q^m  \|
=\sup\Bigl\{\sum_{i=1}^m\sum_{j=1}^na_{ij}g_{ij}s_it_j\colon s\in B_{q^*}^m, t\in B_{p}^n\Bigr\},
\]
where random variables $g_{ij}$ are iid standard Gaussians, $q^*$ denotes the H{\"o}lder 
conjugate of $q$, i.e., the unique number from $[1,\infty]$ satisfying 
$\frac 1q+\frac 1{q^*}=1$, and $B_p^n$ is the unit ball in the $\ell_p$-norm in $\er^n$.

Although the behaviour of random matrices with iid entries is quite well understood, it is not the case for  random matrices with a non-trivial variance profile, whose $\ell_p^n \to \ell_q^m$ norms  appear naturally in many problems in applied mathematics; see the introduction of \cite{APSS} and the references therein. 
However, much effort was made recently to understand $\ell_p^n \to \ell_q^m$ norms of structured random matrices (cf. \cite{BvH, vH,  LvHY, GHLP, vH2017_survey,  LS, Lbern, APSS, BH2024, RX}). 

In this paper we focus on two-sided estimates (i.e., lower and upper bounds matching 
up to a multiplicative constant) for the expectation of $\|G_A\|_{p\to q}$. 
Such bounds encode much more information than only the order of $\Ex \|G_A\|_{p\to q}$. 
They imply two-sided estimates on higher moments and tail bounds for 
$\|G_A\|_{p\to q}$ 
(see Corollary~\ref{cor:higher-moments-tailsGauss} below).
Moreover, they yield a condition for an infinite Gaussian matrix to be a bounded operator 
from $\ell_p$ to $\ell_q$ (see Corollary~\ref{cor:bd-operators}  below). We also discuss 
how to generalize the estimates for $\|G_A\|_{p\to q}$  to  more general classes of random 
matrices with independent, but not necessarily Gaussian entries.

Before we move further, let us introduce some more notation. 
For two nonnegative functions $f$ and $g$ we write $f\lesssim g$ (or $g \gtrsim f$) 
if there exists an absolute constant $C$ such that $f  \le  Cg$;
the notation $f \sim g$ means that $f \lesssim g  \lesssim f$. 
We write $\lesssim_{ \alpha}$, $\sim_{K,\gamma}$, etc.\ if the underlying constant depends 
on the parameters given in the subscripts.
Whenever we write $p\ge p_1$ or $p\le p_2$ we mean that $p\in[p_1, \infty]$ 
or $p\in[1,p_2]$, respectively.
By $[m]$ we denote the set $\{1,\ldots,m\}$ of the first $m$ positive integers.
Let us also denote
\[
\Log 0=1  \qquad \text{and } \qquad \Log x = 1\vee \ln x \quad  \text{for }x>0.
\]

If $p=2=q$, the $\ell_p^n\to \ell_q^m$ norm coincides with the spectral norm and it 
is known by \cite{LvHY} that
\begin{align*}
\Ex\|G_A\|_{2\to 2} 
& \sim \max_i \|(a_{ij})_j\|_{2} +\max_j \|(a_{ij})_i\|_{2} +\Ex\max_{i,j}|a_{ij}g_{ij}|
\\ 
&  \sim  \Ex \max_i \|(a_{ij}g_{ij})_j\|_{2}+  \Ex \max_i \|(a_{ij}g_{ij})_j\|_{2}.
\end{align*}
Moreover,  two-sided bounds are also known for extremal values of $(p,q)$, i.e., 
when $p\in\{1,\infty\}$ or $q\in\{1,\infty\}$ (see \cite[Remark~1.4]{GHLP} and 
\cite[Propositions~1.8 and 1.10]{APSS}). 
The question whether similar two-sided inequalities  hold  for other ranges of $p$ and $q$ 
with arbitrary $A$ was, up to now, entirely open; all  known bounds match only up to 
logarithmic terms in the dimension (see \cite{GHLP, APSS}) or are valid only in some very special cases 
(for the trivial structure, i.e., when $a_{ij}=1$ for all $i,j$, or, more generally, for tensor structures -- this follows from the Chevet inequality \cite{Ch}). 
We refer to  the introductions of \cite{APSS} and \cite{LSChevet} for more details and 
an overview of the history of the problem.

From now on, we will consider only the case $1\le p\le 2\le q\le\infty$. 
The following conjecture was formulated in \cite{GHLP} 
(see  \cite{APSS} for a discussion of other ranges of $p$ and~$q$).

\begin{conj}
\label{conj:lplq}
For every $p\le 2\le q$ and every deterministic $m\times n$ matrix $A=(a_{ij})_{i\le m,j\le n}$,
\[
\Ex\|G_A\|_{p\to q}\sim_{p,q} \max_i \|(a_{ij})_j\|_{p^*} +\max_j \|(a_{ij})_i\|_{q} 
+\Ex\max_{i,j}|a_{ij}g_{ij}|.
\]
\end{conj}

The main difficulty in obtaining Conjecture~\ref{conj:lplq} is to prove the upper estimate, 
since the lower bound is easy.
 As we mentioned above, \cite[Remark~1.4]{GHLP} implies that for every 
$p^*,q\ge 2$,
\begin{equation}
\label{eq:ineq-p=1}
\Ex \|G_A\|_{1\to q} =\Ex \max_j \|(a_{ij}g_{ij})_i\|_{q}
 \lesssim \sqrt{q}  \max_{j} \|(a_{ij})_{i}\|_q 
 + \Ex\max_{i, j}|a_{ij}g_{ij}|
\end{equation}
and
\begin{equation}
\label{eq:ineq-p=1-dual}
\Ex \|G_A\|_{p\to \infty} =\Ex \max_i \|(a_{ij}g_{ij})_j\|_{p^*} 
 \lesssim \sqrt{p^*}  \max_{i} \|(a_{ij})_{j}\|_{p^*} + \Ex\max_{i,j}|a_{ij}g_{ij}|,
\end{equation}
so Conjecture \ref{conj:lplq} holds if $p=1$ or $q=\infty$.
Moreover, for other ranges of $p$ and $q$ it was shown in \cite{GHLP} that 
the upper bound holds up to  multiplicative constants depending logarithmically on the 
dimensions. Our main result provides the upper bound without these logarithmic factors, 
so it confirms Conjecture~\ref{conj:lplq}.

\begin{thm}
\label{thm:main}
If $p^*, q\in [2,\infty)$, then for every deterministic matrix 
$A=(a_{ij})_{i\le m,j\le n}$ we have
\begin{align*}
\Ex\|G_A\|_{p\to q}
& \sim_{p,q}
\max_i \|(a_{ij})_j\|_{p^*} +\max_j \|(a_{ij})_i\|_{q} +\Ex\max_{i,j}|a_{ij}g_{ij}|
\\
&\sim 
\max_i \|(a_{ij})_j\|_{p^*} +\max_j \|(a_{ij})_i\|_{q} 
+\max_{k\ge 0}\inf_{|I|=k}\sqrt{\Log k}\max_{(i,j)\notin I} |a_{ij}|
\\
&  \sim_{p,q}
\max_i \|(a_{ij})_j\|_{p^*} +\max_j \|(a_{ij})_i\|_{q} 
+\max_{k\ge 0}\inf_{ |I|=|J|=k}\sqrt{\Log k}\max_{ i\notin I, j\notin J} |a_{ij}|
\\
&\sim_{p,q}  
\Ex \max_i \|(a_{ij}g_{ ij})_j\|_{p^*} + \Ex \max_j \|(a_{ij}g_{ ij})_i\|_{q}.
\end{align*}
\end{thm}

\begin{rmk} 
\label{rmk:constants}
The constant in the first lower bound of Theorem~\ref{thm:main} does not depend on $p$ and $q$. 
Moreover, it follows from Propositions~\ref{prop:dimdependent}, \ref{prop:weakerD3-allpq},
and \ref{prop:weaker-esimates-suffice} below, that  the constant in the first upper bound of 
Theorem~\ref{thm:main} is at most of order 
\[
\begin{dcases}
(p^*\vee q)^{13/2} 
& \text{ if }\,   1/p + 1/q^*<3/2,
\\
(p^*\vee q)^{5/2} 
& \text{ if }\,  1/p + 1/q^*\ge 3/2.
\end{dcases}
\]
\end{rmk}

\begin{rmk}
\label{rmk:infty}
Recall that the two-sided  bounds for $\Ex\|G_A\|_{p\to q}$ were known before 
 if $p=q=2$ -- in this case we provide a new proof, 
which does not use the spectral methods as the previously known proof did. 
Moreover, inequalities \eqref{eq:ineq-p=1} and \eqref{eq:ineq-p=1-dual} show that in order to prove Conjecture~\ref{conj:lplq} it suffices to restrict ourselves 
to the case $p^*, q\in [2,\infty)$ in Theorem~\ref{thm:main}.
\end{rmk}

Although we were able to confirm Conjecture~\ref{conj:lplq}, 
our methods do not allow us to retrieve the exact dependence of $\Ex\|G_A\|_{p,q}$ on 
$p^*$ and $q$ in Theorem~\ref{thm:main}. For example, if $a_{i,j}=1$ for all $i\le m$ 
and $j\le n$, then 
\[
\Ex\|G_A\|_{p,q} \sim  \sqrt{p^* \wedge \Log n}\max_i \|(a_{ij})_j\|_{p^*} 
+  \sqrt{q\wedge \Log m}\max_j \|(a_{ij})_i\|_{q}
\]
(the third term disappears since, in this case, it is upper bounded by the sum of the 
first two terms), 
whereas the constant in the first upper bound in Theorem~\ref{thm:main}
grows like $(p^*\vee q)^\gamma$ with 
$\gamma>1$.
However, we conjecture that the correct dependence of parameters $p$ and $q$ in the range 
$p\le 2\le q$ is the following.

\begin{conj}
For every $p\le 2\le q$ and every deterministic $m\times n$ matrix $A=(a_{ij})_{i\le m,j\le n}$,
\begin{align*}
\Ex\|G_A\|_{p,q} & \lesssim 
 \sqrt{p^* \wedge \Log n}\max_i \|(a_{ij})_j\|_{p^*} 
+ \sqrt{q\wedge \Log m}\max_j \|(a_{ij})_i\|_{q}  
\\ 
& \qquad+ \Ex\max_{i,j}|a_{ij}g_{ij}|.
\end{align*}

\end{conj}

%%%%%%%%%%%%%%%%%%%%%%%%%%%%
%%%%%%%%%%%%%%%%%%%%%%%%%%%%%%
\subsection{Consequences of the main result}

Let us now present a couple of consequences of Theorem~\ref{thm:main}. Some of them are immediate and the rest is proven in Section~\ref{sect:proof-main-thm}.

Theorem~\ref{thm:main} and inequalities~\eqref{eq:ineq-p=1} and \eqref{eq:ineq-p=1-dual} 
easily imply their  non-centered counterpart:
\[
\Ex\|(a_{ij}g_{ij}+m_{ij})_{i,j}\|_{p\to q}
\sim_{p,q} 
\Ex\max_i \|(a_{ij}g_{ ij})_j\|_{p^*} +\Ex\max_j \|(a_{ij}g_{ ij})_i\|_{q} 
+ \|(m_{ij})_{i,j}\|_{p\to q}
\]
for every $m_{ij},a_{ij}\in \er$, $i\le m$, $j\le n$, and every  $p^*, q\ge 2$.

Moreover, Theorem~\ref{thm:main}, inequalities~\eqref{eq:ineq-p=1} and \eqref{eq:ineq-p=1-dual}, and 
\cite[Proposition~1.2]{APSS} yield the following 
characterisation of the boundedness of Gaussian linear operators from $\ell_p$ to $\ell_q$ 
whenever $p^*, q\ge 2$.
We say that a matrix $B= (b_{ij})_{i,j\in \en}$ 
defines a bounded operator from $\ell_p$ to $\ell_q$ if for all $x \in \ell_p$ the product 
$B x$ is well defined, belongs to $\ell_q$, and the corresponding linear operator is bounded.

\begin{cor}
\label{cor:bd-operators}
Let  $p^*, q\in [2,\infty]$, and let $(a_{ij})_{i,j\in \en}$ 
be an infinite deterministic real matrix. The matrix $(a_{ij}g_{ij})_{i,j\in \en}$  
defines a bounded linear operator between $\ell_p$ and $\ell_q$ almost surely 
if and only if  
 $\sup_i\|(a_{ij})_j\|_{p^*}<\infty$, $\sup_{j}\|(a_{ij})_i\|_q<\infty$, and  
$\Ex \sup_{i,j\in\en} |a_{ij}g_{ij}|<\infty$.
\end{cor}

\begin{rmk}
The condition  $\Ex \sup_{i,j\in\en} |a_{ij}g_{ij}|<\infty$ in Corollary~\ref{cor:bd-operators} 
is equivalent to the deterministic bound
\[
\sup_{k\geq 0}\inf_{|I|=k}\sqrt{\Log k}\sup_{(i,j)\in ( \en\times\en)\setminus I} |a_{ij}|
<\infty
\]
(see estimate \eqref{eq:Emax-vs-max-subsets} below). 
\end{rmk}

Theorem~\ref{thm:main} easily implies two-sided bounds for norms of 
Gaussian mixtures. 
We say that a random  variable $X$ is a Gaussian mixture if there exists a 
nonnegative random variable  $R$ such that $X$ has the same distribution as $Rg$,  
where $g$ is a standard Gaussian random variable, independent of $R$  (cf. \cite{ENT}). 
The next corollary is an immediate consequence of Theorem~\ref{thm:main} and 
inequalities~\eqref{eq:ineq-p=1} and \eqref{eq:ineq-p=1-dual}.
In all the corollaries of this subsection the constants depending on $p$ and $q$ 
are  at most of order $q_0^7$,  provided that $p^*,q\in[2,q_0]$.

\begin{cor}
\label{cor:mixtures}
Assume that $p^*, q\in [2,\infty]$ and let $X_{ij}$, $i\le m,j\le n$, be independent 
Gaussian mixtures. Then
\begin{align*}
\Ex\|(X_{ij})_{i\le m, j\le n}\|_{p\to q}
&\sim_{p,q} 
\Ex \max_i \|(X_{ij})_j\|_{p^*} + \Ex \max_j \|(X_{ij})_i\|_{q}.
\end{align*}
\end{cor}

We say that  $X$ is a symmetric Weibull random variable with (shape) parameter $r\in~(0,\infty]$ 
if  $X$ is symmetric and for every $t\ge 0$,
\[
\Pr(|X|\ge t)=e^{-t^r}.
\]

\begin{cor}	
\label{cor:Weibulls-below4}
Let $X_{ij}$, $i\le m$, $j\le n$, be independent symmetric Weibull variables with parameter 
$r\in(0,2]$. Then for every  $p^*, q\in [2,\infty)$,  
and every deterministic matrix $A=(a_{ij})_{i\le m,j\le n}$ we have
\begin{align*}
\Ex\| (a_{ij}X_{ij} & )_{i\le m, j\le n}\|_{p\to q}
\\ &\sim_{p,q,r} 
\max_i \|(a_{ij})_j\|_{p^*} +\max_j \|(a_{ij})_i\|_{q} +\Ex\max_{i,j}|a_{ij}X_{ij}|
\\ & \sim_{  r} 
\max_i \|(a_{ij})_j\|_{p^*} +\max_j \|(a_{ij})_i\|_{q} 
+\max_{k\ge 0}\inf_{|I|=k}\Log^{1/r} k\max_{(i,j)\notin I} |a_{ij}|
\\
& \sim_{ p,q, r} 
\max_i \|(a_{ij})_j\|_{p^*} +\max_j \|(a_{ij})_i\|_{q} 
+\max_{k\ge 0}\inf_{ |I|=|J|=k}\Log^{1/r}k\max_{ i\notin I, j\notin J} |a_{ij}|
\\
&\sim_{p,q,r} 
\Ex \max_i \|(a_{ij}X_{ij})_j\|_{p^*} + \Ex \max_j \|(a_{ij}X_{ij})_i\|_{q}.
\end{align*}
Moreover, if  $p^*, q\in [2,\infty]$, then 
\begin{align*}
\Ex\| (a_{ij}X_{ij} & )_{i\le m, j\le n}\|_{p\to q}
\\ &\sim_{p,q,r} 
\max_i \|(a_{ij})_j\|_{p^*} +\max_j \|(a_{ij})_i\|_{q} +\Ex\max_{i,j}|a_{ij}X_{ij}|
\\ & \sim_{  r} 
\max_i \|(a_{ij})_j\|_{p^*} +\max_j \|(a_{ij})_i\|_{q} 
+\max_{k\ge 0}\inf_{|I|=k}\Log^{1/r} k\max_{(i,j)\notin I} |a_{ij}|
\\
&\sim_{p,q,r} 
\Ex \max_i \|(a_{ij}X_{ij})_j\|_{p^*} + \Ex \max_j \|(a_{ij}X_{ij})_i\|_{q}.
\end{align*}
\end{cor}

We postpone the proof of Corollary~\ref{cor:Weibulls-below4} to Section~\ref{sect:proof-main-thm}.

\begin{rmk}
\label{rmk:rademisdiff}
One cannot omit the assumption $r\le 2$ in  Corollary~\ref{cor:Weibulls-below4}. 
Indeed, in the limit case $r=\infty$ the entries $X_{ij} =\ve_{ij}$ are independent symmetric 
Bernoulli random variables  and it is known that the behaviour of the expected operator 
norm is different than in the case $r\le 2$ (see  \cite{Seginer}). Moreover, 
it was conjectured in \cite{LS} and proven in \cite{Lbern}  up to a factor of 
order $\log\log\log(mn)$ that
\begin{align*}
\Ex \|(a_{ij}\ve_{ij})_{i,j}\|_{2\to 2} 
& \mathop{\sim}^? \max_i \|(a_{ij})_j\|_{2} +\max_j \|(a_{ij})_i\|_{2} 
\\ 
& \qquad\qquad + \max_{k\ge 0}\inf_{|I|=k} \sup_{\|s\|_2,\|t\|_2\le 1} 
\Bigl\| \sum_{i,j\notin I} a_{ij}\ve_{ij}s_it_j\Bigr\|_{\Log k}	
\\	
& = \Ex \max_i \|(a_{ij}\ve_{ij})_j\|_{2} + \Ex \max_j \|(a_{ij}\ve_{ij})_i\|_{2} 
\\ 
& \qquad \qquad+ \max_{k\ge 0}\inf_{|I|=k} \sup_{\|s\|_2,\|t\|_2\le 1} 
\Bigl\| \sum_{i,j\notin I} a_{ij}\ve_{ij}s_it_j\Bigr\|_{\Log k}	.
\end{align*}
We conjecture that for $p\le 2\le q$,
\begin{align}	
\label{eq:bern-conj-pq}
\Ex \|(a_{ij}\ve_{ij})_{i,j}\|_{p\to q} 
& \mathop{\sim_{p,q}}^? \max_i \|(a_{ij})_j\|_{p^*} +\max_j \|(a_{ij})_i\|_{q} 
\\ 
\notag
& \qquad\qquad + \max_{k\ge 0}\inf_{|I|=k} \sup_{\|s\|_{q^*},\|t\|_p\le 1} 
\Bigl\| \sum_{i,j\notin I} a_{ij}\ve_{ij}s_it_j\Bigr\|_{\Log k}	
\\	
\notag
&  = \Ex \max_i \|(a_{ij}\ve_{ij})_j\|_{p^*} 
+ \Ex \max_j \|(a_{ij}\ve_{ij})_i\|_{q} 
\\ 
\notag
& \qquad \qquad+ \max_{k\ge 0}\inf_{|I|=k} \sup_{\|s\|_{q^*},\|t\|_p\le 1} 
\Bigl\| \sum_{i,j\notin I} a_{ij}\ve_{ij}s_it_j\Bigr\|_{\Log k}.
\end{align}
We also believe that the methods of \cite{Lbern} could be adapted to the case $p^*,q\ge 2$ 
and that -- together with Theorem~\ref{thm:main} -- they would imply \eqref{eq:bern-conj-pq}
up to a polylog factor.
\end{rmk}

The bounds for the expectation of the norm of a random matrix with  independent entries 
satisfying  some mild regularity assumptions automatically imply  bounds for higher moments 
as well as for the tails of this norm.  
Let us state explicitly two such estimates for the structured Gaussian and Weibull 
random matrices.

Theorem~\ref{thm:main}, inequalities~\eqref{eq:ineq-p=1} and \eqref{eq:ineq-p=1-dual},
and the Gaussian concentration yield the following moment and tail bounds.
\begin{cor}	
\label{cor:higher-moments-tailsGauss}
If  $p^*, q\in [2,\infty]$, then for every deterministic matrix 
$A=(a_{ij})_{i\le m,j\le n}$,  $\rho \ge 1$, and $t>0$ we have
\begin{align*}
(\Ex\|G_A\|_{p\to q}^\rho)^{1/\rho}
&\sim_{p,q} 
\max_i \|(a_{ij})_j\|_{p^*} +\max_j \|(a_{ij})_i\|_{q} 
+\max_{k\ge 0}\inf_{|I|=k}\sqrt{\Log k}\max_{(i,j)\notin I} |a_{ij}|
\\ 
& \qquad \quad+\sqrt \rho \max_{i,j}|a_{ij}|,
\end{align*}
and
\begin{align*}
\Pr\Bigl(\|G_A\|_{p\to q} 
\ge  C(p,q)\Bigl(&\max_i \|(a_{ij})_j\|_{p^*} +\max_j \|(a_{ij})_i\|_{q} 
\\ 
&+\max_{k\ge 0}\inf_{|I|=k}\sqrt{\Log k}\max_{(i,j)\notin I} |a_{ij}|\Bigr) + t\Bigr)
  \le e^{-t^2/( 2\max_{i,j}a_{ij}^2)}.
\end{align*}
\end{cor}

In the Weibull case,  Corollary~\ref{cor:Weibulls-below4},  \cite[Theorem~1.1 and Corollary~1.3]{LatalaStrzeleckaMat}, and \cite[Lemma~2.19]{APSS}  imply the following.  
(One may also deduce the moreover part, with a constant $C_2$ depending on $p,q$ and $r$, 
from  \eqref{eq:higher-moments-weibull} via Markov's inequality.)

\begin{cor}	
\label{cor:higher-moments-tailsWeibull}
Let $X_{ij}$, $i\le m$, $j\le n$, be independent symmetric Weibull variables with 
parameter $r\in~(0,2]$. If  $p^*, q\in [2,\infty]$, then for every deterministic 
matrix $A=(a_{ij})_{i\le m,j\le n}$, and every $\rho \ge 1$ we have
\begin{align}	
\notag
(\Ex\|(a_{ij}X_{ij})_{i,j}\|_{p\to q}^\rho)^{1/\rho}
\sim_{p,q,r}& 
\max_i \|(a_{ij})_j\|_{p^*} +\max_j \|(a_{ij})_i\|_{q} 
\\ \label{eq:higher-moments-weibull}
& \quad \quad 
+\max_{k\ge 0}\inf_{|I|=k}\Log^{1/r} k\max_{(i,j)\notin I} |a_{ij}|
+ \rho^{1/r} \max_{i,j}|a_{ij}|.
\end{align}
Moreover, for every $t>0$,
\begin{multline*}
\Pr\Bigl(\|(a_{ij}X_{ij})_{i,j}\|_{p\to q} 
\ge  C_1(p,q,r)\Bigl(\max_i \|(a_{ij})_j\|_{p^*} 
+\max_j \|(a_{ij})_i\|_{q} 
\\
+\max_{k\ge 0}\inf_{|I|=k} \Log^{1/r}k\max_{(i,j)\notin I} |a_{ij}|\Bigr) + t\Bigr)
  \le e^{-t^r/( C_2(r)\max_{i,j}|a_{ij}|^r)}.
\end{multline*}
\end{cor}

\begin{rmk}
The upper bound in \eqref{eq:higher-moments-weibull} and,  as a consequence, the tail bound 
from Corollary~\ref{cor:higher-moments-tailsWeibull} hold under the weaker assumption that 
the variables $X_{ij}$ are independent, centered, 
and have uniformly bounded $\psi_r$-norm. 
This follows from a standard argument (see, e.g., the proof of \cite[Lemma~2.1]{LSChevet}).
\end{rmk}

The next result is a generalization of \cite[Theorem 2]{Latala-Some-estimates}. 
Its advantage is that we do not need to assume much about the distribution of the entries; 
however, two additional summands appear in the upper bound.
The proof of the following corollary may be found at the end of Section~\ref{sect:proof-main-thm}.

\begin{cor}	
\label{cor:generalentries}
If  $p^*, q\in [2,\infty)$, then for every matrix $(X_{ij})_{ i\le m, j\le n}$ with
independent centered entries,
\begin{align}
\label{eq:genentr2}
\Ex&\|(X_{ij})_{i,j}\|_{p\to q}
\lesssim_{p,q}\max_{i}\Bigl(\sum_{j}\Ex |X_{ij}|^{p^*}\Bigr)^{1/p^*}
+\max_{j}\Bigl(\sum_{i}\Ex |X_{ij}|^{q}\Bigr)^{1/q}
\\
\notag
&\qquad+\Bigl(\sum_i\Bigl(\sum_j\Ex |X_{ij}|^{2p^*}\Bigr)^{\frac{p^*\vee q}{p^*}}\Bigr)
^{\frac{1}{2(p^*\vee q)}}
+\Bigl(\sum_{j}\Bigl(\sum_i\Ex |X_{ij}|^{2q}\Bigr)^{\frac{p^*\vee q}{q}}\Bigr)^{\frac{1}{2(p^*\vee q)}}
\end{align}
and, as a consequence, 
\begin{align*}
\Ex\|(X_{ij})_{i,j}\|_{p\to q}
\lesssim_{p,q}&\max_{i}\Bigl(\sum_{j}\Ex |X_{ij}|^{p^*}\Bigr)^{1/p^*}
+\max_{j}\Bigl(\sum_{i}\Ex |X_{ij}|^{q}\Bigr)^{1/q}
\\
\notag
&\qquad+\Bigl(\sum_{i,j}\Ex |X_{ij}|^{2p^*}\Bigr)^{1/(2p^*)}
+\Bigl(\sum_{i,j}\Ex |X_{ij}|^{2q}\Bigr)^{1/(2q)}.
\end{align*} 
\end{cor}

\begin{rmk}
If $p^*, q\in [2,\infty)$, $\alpha \ge 1$, and independent centered random variables $X_{ij}$ of variance $1$ satisfy 
\begin{equation}	
\label{eq:weak-mom-reg}
(\Ex|X_{ij}|^{2\rho})^{1/(2\rho)}\leq
\alpha (\Ex|X_{ij}|^{\rho})^{1/\rho} \qquad \text{for } \rho\in \{p^*,q\},
\end{equation}
 then by the H{\"o}lder inequality we get
 \[
 \|X_{ij}\|_{2\rho} \sim_{\alpha}\|X_{ij}\|_\rho \sim_{\alpha, \rho} \|X_{ij}\|_2=1 \qquad \text{for } \rho\in \{p^*,q\}.
 \]
 Thus, estimate \eqref{eq:genentr2} yields
\begin{align*}
\Ex\|(X_{ij})_{i,j}\|_{p\to q}
& \lesssim_{p,q,\alpha} n^{1/p^*}+m^{1/q}
+m^{1/(2(p^*\vee q))}n^{1/(2p^*)}
+m^{1/(2q)}n^{1/(2(p^*\vee q))}
\\
&\lesssim_{ p,q,\alpha} n^{1/p^*}+m^{1/q},
\end{align*}
where in the last inequality we used  the AM-GM inequality.
One can repeat the argument from the proof of  the lower bound in 
\cite[Proposition~21]{LSiidmatr} to show that under the above assumptions 
\begin{align*}
\Ex\|(X_{ij})_{i,j}\|_{p\to q}
&\gtrsim_{p,q, \alpha} n^{1/p^*}\|X_{11}\|_{p^*}+m^{1/q}\|X_{11}\|_{q},
\end{align*} so in fact
\[
\Ex\|(X_{ij})_{i,j}\|_{p\to q} \sim_{p,q,\alpha}  n^{1/p^*}\|X_{11}\|_{p^*}+m^{1/q}\|X_{11}\|_{q}.
\]
 We refer to \cite{LSiidmatr} for more precise two-sided bounds (with constants not depending 
 on $p$ and $q$) under the stronger assumption that the entries $X_{ij}$ are iid centered 
 $\alpha$-regular random variables.
\end{rmk}

%%%%%%%%%%%%%%%%%%%%%%%%%%%%
%%%%%%%%%%%%%%%%%%%%%%%%%%%%%%
\subsection{Strategy of the proof of the main result}

Throughout the paper we  denote
\[
D_1\coloneqq \max_i \|(a_{ij})_j\|_{p^*},\quad D_2\coloneqq \max_j \|(a_{ij})_i\|_{q}
\]
to avoid  long formulas for the first two terms on the right-hand side of our main 
estimates.

Similarly as in \cite{LvHY}, Theorem~\ref{thm:main}  is a consequence of two  weaker  
estimates given in Propositions~\ref{prop:dimdependent} and \ref{prop:weakerD3-allpq}.

\begin{prop}
\label{prop:dimdependent}
Assume that $p^*, q\in [2,\infty)$. Then for  every deterministic matrix 
$A=(a_{ij})_{i\le m,j\le n}$,
\begin{equation}
\label{eq:pqsqrtlogmn}
\Ex\|G_A\|_{p\to q}
 \lesssim
 \alpha(p,q) \Bigl(\sqrt{p^*}D_1+\sqrt{q}D_2
 +\sqrt{\Log(mn)}\max_{i,j}|a_{ij}|\Bigr),
\end{equation}
where 
\[
 \alpha(p,q)=
\begin{dcases}
 (p^*\vee q)^{3/2}
& \mbox{if } \, 1/p+1/q^*\geq 3/2 ,
\\
(p^*\vee q)^{11/2}  
& \mbox{if }\,  1/p+1/q^*<3/2 .
\end{dcases}	
\]
\end{prop}

\begin{prop}
\label{prop:weakerD3-allpq}
For every $p^*, q\in [2,\infty)$ and every deterministic matrix $A=(a_{ij})_{i\le m,j\le n}$,
\begin{equation*}
\Ex \|G_A\|_{p\to q}  
 \lesssim 
 \beta(p,q) \bigl(\sqrt{p^*}D_1+\sqrt{q}D_2\bigr) 
 +   \beta'(p,q)\max_{i,j}(i+j)^{(8(p^*\vee q))^{-1}}|a_{ij}| ,
\end{equation*}
where 
$\beta'(p,q)=(p^*\vee q)\beta(p,q)$, $\beta(p,q) = \alpha(p,q)$, and $\alpha(p,q)$ is defined in Proposition~\ref{prop:dimdependent}.
\end{prop}

In the case $p=q=2$ estimate \eqref{eq:pqsqrtlogmn} was proven by Bandeira and van Handel 
in \cite{BvH}. We  cannot  use their combinatorial approach based on the trace method since for 
$(p,q)\neq(2,2)$ we no longer deal with the spectral norm.
Proving Proposition~\ref{prop:dimdependent} is one of the main difficulties and 
novelties of our paper. Using our new methods we also reprove Proposition~\ref{prop:dimdependent} 
in the case $p=q=2$ without relying on spectral tools.

To prove Proposition~\ref{prop:dimdependent} we first obtain the following weaker result with a bigger exponent of the logarithm. Then we perform the exponent reduction described in Section~\ref{subsect:logexp-reduction}.
\begin{prop}	
\label{prop:worseD3'-intro}
If $p^*,q \in [2,\infty)$, then
\begin{align*}
\Ex\|G_A\|_{p\to q} 
& \lesssim 
\sqrt{p^*}D_1+\sqrt{q}D_2+\Log^{ \gamma }(mn)\max_{i,j} |a_{ij}| ,
\end{align*}
where $\gamma \leq 5$ is a universal constant.
\end{prop}

To prove  Proposition~\ref{prop:worseD3'-intro} in the case  $( p^*,q)\notin[2,3]^2$ it suffices -- by duality -- to consider  the case $q\geq p^*$, so that $q>3$. 
We split every vector in $B_{q*}^m$ into a part with a small support and a part with a small supremum norm. 
To deal with vectors with small supports we use a standard net argument. 
In order to bound the supremum over vectors in $B_{q^*}^m$ with small coordinates we first introduce  a certain interpolation norm of $\ell_2^m$- and $\ell_q^m$-norms satisfying a crucial
hypercontractive property. 
To estimate the norm of $G_A$ acting between $\ell_p^n$ and the interpolation space we discretise the $B_p^n$ ball and use an inductive argument; in every induction step we use the Gaussian concentration and the hypercontraction. This approach is implemented in Subsection~\ref{subsect:proof-worseD3'}.

The proof of Proposition~\ref{prop:worseD3'-intro}  in the case  $( p^*,q)\in[2,3]^2$ is based on 
the ideas from the proof of the main result of \cite{vH} (a similar result for $p^*=q\in [2,4)$ 
was obtained in \cite{Matlak}). However, in this case performing the exponent reduction with 
constants which do not blow up when $p^*$ and $q$ are close to $2$ is much more involved than in 
the case $( p^*,q)\notin[2,3]^2$.
The exponent reduction in this case is described in Section~\ref{sect:appendix}.
This is the most challenging part of our new proof of two-sided bounds in the  case $p=q=2$.

Finally, to deduce Proposition~\ref{prop:weakerD3-allpq} we decompose 
the matrix $A$ into block diagonal matrices 
$A_{k}$  with blocks of a smaller size and matrices $B_l$ whose norms are 
 easier to control (due to Proposition~\ref{prop:sqrtIk} below), and use 
Proposition~\ref{prop:dimdependent} for each $A_{k}$ separately.  In the proof of the aforementioned Proposition~\ref{prop:sqrtIk} we use similar tools as in the proof of Proposition~\ref{prop:worseD3'-intro} in the case $( p^*,q)\notin[2,3]^2$ for vectors with small coordinates.

%%%%%%%%%%%%%%%%%%%%%%%%%%%%
%%%%%%%%%%%%%%%%%%%%%%%%%%%%%%

\subsection{Organization of the paper}
In Section~\ref{sect:proof-main-thm} we show how 
Propositions~\ref{prop:dimdependent} and \ref{prop:weakerD3-allpq} imply 
Theorem~\ref{thm:main} 
and then we prove Corollaries~\ref{cor:Weibulls-below4} and \ref{cor:generalentries}.
In Section~\ref{sect:dim-dep} we prove Proposition~\ref{prop:dimdependent} in the case  
$(p^*,q)\notin [2,3]^2$.
In Section~\ref{sect:decomposition} we show how to deduce
Proposition~\ref{prop:weakerD3-allpq} from Proposition~\ref{prop:dimdependent}. 
Finally, in Section~\ref{sect:appendix} we prove  Proposition~\ref{prop:dimdependent} 
in the case $p^*,q\in [2,3]$.

%%%%%%%%%%%%%%%%%%%%%%%%%%%%
%%%%%%%%%%%%%%%%%%%%%%%%%%%%%%

\section{Proof of Theorem~\ref{thm:main} and its corollaries}  
\label{sect:proof-main-thm}

In this section we first show how to deduce the most challenging part of Theorem~\ref{thm:main} 
from Propositions~\ref{prop:dimdependent} and \ref{prop:weakerD3-allpq}. 
Then we  give the proofs of  Theorem~\ref{thm:main} and Corollaries~\ref{cor:Weibulls-below4} 
and \ref{cor:generalentries}. 
 The proofs of Propositions~\ref{prop:dimdependent} and \ref{prop:weakerD3-allpq} 
may be  found in the next sections.

\begin{prop}	
\label{prop:weaker-esimates-suffice}
Let $p^*,q\ge 2$, and $\alpha_1,\alpha_2,\alpha_3,\beta_1,\beta_2,\beta_3\ge 1$. 
Assume that for every integers $m$ and $n$,  and every deterministic matrix 
$A=(a_{ij})_{ i\le m, j\le n}$,
\[
\Ex\|G_A\|_{p\to q}
\le \alpha_1D_1+\alpha_2D_2 +\alpha_3\sqrt{\Log(mn)}\max_{i,j}|a_{ij}| 
\]
and 
\[
\Ex \|G_A\|_{p\to q} 
\le  \beta_1D_1+\beta_2D_2+\beta_3\max_{i,j}\,(i+j)^{ (8 (p^*\vee q))^{-1}}|a_{ij}|.
\]
Then  for  every integers $m$ and $n$, and every deterministic matrix $A=(a_{ij})_{i\le m, j\le n}$,
\begin{align*}
\Ex \|G_A\|_{p\to q} 
&\lesssim (\alpha_1 + \beta_1+\beta_3)D_1+
 (\alpha_2 + \beta_2+\beta_3)D_2
\\ 
&\quad + \alpha_3 \max_{ k\geq 0}\inf_{|I|=k}\sqrt{\Log k}\max_{(i,j)\notin I} |a_{ij}| .
\end{align*}
\end{prop}

We will need the following well-known deterministic lemma about norms of  block diagonal matrices.
\begin{lem}	
\label{lem:block-matr}
Let $(c_{ij})_{i\le m, j\le n}$ be a block diagonal matrix with blocks $C_l$, and $1\le p\le q$. 
Then
\[
\|(c_{ij})_{i\le m, j\le n}\|_{p\to q} = \max_{l} \|C_l\|_{p\to q}.
\]
\end{lem}

\begin{proof}
Assume that the block $C_l$, $l\le l_0$, consists of entries $a_{ij}$ such that $i\in I_l$ 
and $j\in J_l$. Then
\begin{align*}
\|(c_{ij})_{i\le m, j\le n}\|_{p\to q} 
&= 
\sup_{s\in B_{q^*}^m, t\in B_{p}^n} \sum_{i,j}c_{ij}s_it_j
\\ 
&=
\sup_{x\in B_{q^*}^{l_0}, y\in B_{p}^{l_0}} \sum_{l=1}^{l_0} 
\sup_{s\in x_lB_{q^*}^{I_l}, t\in y_l B_p^{J_l}} \sum_{i\in I_l,j\in J_l}c_{ij}s_it_j
\\ 
&=
\sup_{x\in B_{q^*}^{l_0}, y\in B_{p}^{l_0}} \sum_{l=1}^{l_0} x_l y_l\|C_l\|_{p\to q}
= \sup_{y\in B_{p}^{l_0}}\Bigl( \sum_{l=1}^{l_0}  |y_l|^q\|C_l\|_{p\to q}^{q} \Bigr)^{1/q}.
\end{align*}
Since $p\le q$ and $q\ge 1$, the latter supremum is attained at $y=e_l$ for some $l\le l_0$, so
\begin{equation*}
\sup_{y\in B_{p}^{l_0}}\Bigl( \sum_{l=1}^{l_0}  |y_l|^q\|C_l\|_{p\to q}^{q} \Bigr)^{1/q}
= \max_l  \|C_{l}\|_{p\to q}.
\qedhere
\end{equation*}	
\end{proof}

\begin{proof}[Proof of Proposition~\ref{prop:weaker-esimates-suffice}]
Let
\[
D_\infty\coloneqq \max_{ k\geq 0}\inf_{|I|=k}\max_{ (i,j)\notin I} \sqrt{\Log k}\, |a_{ij}|,
\]
$N_0=1$, and $N_k=2^{2^k}$ for $k\geq 1$.
Without loss of generality  we may assume that $n=m=N_{k_0}$ for some $k_0$; if necessary, 
we simply add zero rows and columns.

We follow the ideas of the proof of \cite[Theorem 3.9]{LvHY} and  \cite[Remark 4.5]{LS}, 
starting with  constructing suitable permutations $(i_1,\ldots,i_{N_{k_0}})$  and 
$(j_1,\ldots,j_{N_{k_0}})$ of $\{1,\ldots, N_{k_0}\}$. 
(Note that the change of  order of rows and columns does not change $\Ex\|G_A\|_{p\to q}$.) 
Then we  decompose the matrix $G_A$ and bound each piece of this decomposition separately, 
each time using one of the assumptions of the proposition.

In the first step we choose $I_1=\{i_1,\ldots,i_{N_1}\}$ and $J_1=\{j_1,\ldots,j_{N_1}\}$ 
in such a way that
\begin{equation*}
\max_{i\notin I_1}\max_j |a_{ij}|\leq \frac{D_\infty}{\sqrt{\Log N_1}} \quad\mbox{and}\quad
\max_{j\notin J_1}\max_i |a_{ij}|\leq \frac{D_\infty}{\sqrt{\Log N_1}}. 
\end{equation*}

Suppose now that we have selected $I_k=\{i_1,\ldots,i_{N_k}\}$ and $J_k=\{i_1,\ldots,j_{N_k}\}$ 
for $k<k_0$. To construct $(i_{N_k+1},\ldots,i_{N_{k+1}})$ we choose  first $N_{k}N_{k-1}$ 
indices $i$  from  $[m]\setminus I_k$ that contain the 
$N_{k-1}$ largest moduli of entries $|a_{ij}|$ from each column $j\in J_k$. Next, among remaining indices we
choose $N_{k+1}-N_k-N_{k}N_{k-1}\geq N_{k}$ indices $i$ in such a way that  
$I_{k+1}=\{i_1,\ldots,i_{N_{k+1}}\}$ satisfies
\begin{equation}
\label{eq:estbyDinfty2}
\max_{i\notin I_{k+1}}\max_j |a_{ij}|\leq \frac{D_\infty}{\sqrt{\Log N_k}}.
\end{equation}
Similarly, to construct $(j_{N_k+1},\ldots,j_{N_{k+1}})$ we choose first $N_{k}N_{k-1}$ indices $j$ from  $[n]\setminus J_k$ 
 that contain the 
$N_{k-1}$ largest moduli of entries $|a_{ij}|$ from each row $i\in I_k$. Next, among remaining indices we
choose $N_{k+1}-N_k-N_{k}N_{k-1}\geq N_{k}$ indices $j$ in such a way that  
$J_{k+1}=\{j_1,\ldots,j_{N_{k+1}}\}$ satisfies
\begin{equation*}
%\label{eq:estbyDinfty3}
\max_{j\notin J_{k+1}}\max_i |a_{ij}|\leq \frac{D_\infty}{\sqrt{\Log N_k}}.
\end{equation*}

The above construction implies in particular that for $k\ge 1$,
\begin{align}
\label{eq:estoffdiag1}
|a_{ij}|&\leq D_2N_{k-1}^{-1/q}\le 2 D_22^{-2^{k-1}/q} &&\mbox{ if }j\leq N_k 
\\ \notag &&& \quad \mbox{ and }i\geq M_k\coloneqq N_k+N_kN_{k-1},
\\ 
\label{eq:estoffdiag2}
|a_{ij}|&\leq D_1N_{k-1}^{-1/p^*}\le 2 D_12^{-2^{k-1}/p^*} && \mbox{ if }i\leq N_k \mbox{ and }j\geq M_k.
\end{align}

We set (see Figure~\ref{fig:opdecomp}, which comes from \cite{LvHY}; we provide it here for the reader's convenience)\
\begin{gather*}
E_1\coloneqq [1,M_1]^2\cup\bigcup_{k\geq 1}[N_{2k}+1,M_{2k+1}\wedge N_{k_0}]^2,
\\
E_2\coloneqq \bigcup_{k\geq 1}[N_{2k-1}+1,M_{2k}\wedge N_{k_0}]^2\setminus E_1,
\qquad E_3\coloneqq [1,N_{k_0}]^2\setminus(E_1\cup E_2),
\end{gather*}
and write $G_A=U+V+W$, where 
\[
U_{ij}\coloneqq a_{ij}g_{ij}\ind_{\{(i,j)\in E_1\}},\quad V_{ij}\coloneqq a_{ij}g_{ij}\ind_{\{(i,j)\in E_2\}},
\quad W_{ij}\coloneqq a_{ij}g_{ij}\ind_{\{(i,j)\in E_3\}}.
\]

\begin{figure}[t]
\begin{center}
\vskip.2cm
\begin{tikzpicture}
\draw (0,0) rectangle (5,-5);

\draw (1,-1) rectangle (2.3,-2.3);
\draw (3,-3) rectangle (4.3,-4.3);

\draw[fill=white] (0,0) rectangle (1.3,-1.3);
\draw[fill=white] (2,-2) rectangle (3.3,-3.3);
\draw[fill=white] (4,-4) rectangle (5,-5);

\draw (.65,-.65) node {$U$};
\draw (2.65,-2.65) node {$U$};
\draw (4.5,-4.5) node {$U$};

\draw (1.65,-1.65) node {$V$};
\draw (3.65,-3.65) node {$V$};

\draw (1.25,-3.75) node {$W$};
\draw (3.75,-1.25) node {$W$};

\draw (-.05,0) to (-.1,0) node[left] {\scriptsize $1$};
\draw (-.05,-1) to (-.1,-1) node[left] {\scriptsize $N_1$};
\draw (-.05,-1.3) to (-.1,-1.3) node[left] {\scriptsize $M_1$};
\draw (-.05,-2) to (-.1,-2) node[left] {\scriptsize $N_2$};
\draw (-.05,-2.3) to (-.1,-2.3) node[left] {\scriptsize $M_2$};
\draw (-.05,-3) to (-.1,-3) node[left] {\scriptsize $N_3$};
\draw (-.05,-3.3) to (-.1,-3.3) node[left] {\scriptsize $M_3$};
\draw (-.1,-4) node[left] {\scriptsize $\vdots\,$};
\draw (-.05,-5) to (-.1,-5) node[left] {\scriptsize $n=N_{k_0}$};

\draw (0,.05) to (0,.1) node[above] {\scriptsize $1$};
\draw (1,.05) to (1,.1) node[above] {\scriptsize $N_1\ $};
\draw (1.3,.05) to (1.3,.1) node[above] {\scriptsize $\ M_1$};
\draw (2,.05) to (2,.1) node[above] {\scriptsize $N_2\ $};
\draw (2.3,.05) to (2.3,.1) node[above] {\scriptsize $\ M_2$};
\draw (3,.05) to (3,.1) node[above] {\scriptsize $N_3\ $};
\draw (3.3,.05) to (3.3,.1) node[above] {\scriptsize $\ M_3$};
\draw (4.15,.1) node[above] {\scriptsize $\cdots$};
\draw (5,.05) to (5,.1) node[above] {\scriptsize $n$};
\end{tikzpicture}
\vskip.2cm
\end{center}
\caption{\small Illustration of the decomposition into matrices $U$, $V$, and $W$. The figure is drawn for clarity on a  $\log\log$ scale.\label{fig:opdecomp}}
\end{figure}

The matrix $U$ is block diagonal with the first block $U_1=(X_{ij})_{i,j\in E_{1,1}}$ 
of dimension  $|E_{1,1}|=M_1$ and blocks $U_k=(X_{ij})_{i,j\in  E_{1,k}}$ for $k\geq 2$ 
of dimension $| E_{1,k}| \leq M_{2k-1}-~N_{2k-2}$. We have
\[
(\Ex\|U_1\|_{p\to q}^2)^{1/2}\leq (\Ex\|U_1\|_{2\to 2}^2)^{1/2}
 \leq\Bigl(\sum_{i,j\in [1,M_1]}\Ex g_{ij}^2a_{ij}^2\Bigr)^{1/2}\leq M_1\max_{i,j}|a_{ij}|\lesssim D_\infty.
\]

For $k=2,3,\ldots$ the Lipschitz constant of the function 
\[
\er^{mn}\ni z\mapsto \sup_{s\in B_{q^*}^m, t\in B_p^n} \sum_{i,j\in E_{1,k}} a_{ij}z_{ij}s_it_j = \|(a_{ij}z_{ij})_{i,j\in E_1,k}\|_{p\to q}
\]
is equal to
\[
\sup_{s\in B_{q^*}^m, t\in B_p^n} \Bigl(  \sum_{i,j\in E_{1,k}} a_{ij}^2 s_i^2 t_j^2 \Bigr)^{1/2}
\le \sup_{s\in B_{2}^m, t\in B_2^n} \Bigl(  \sum_{i,j\in E_{1,k}} a_{ij}^2  s_i^2 t_j^2 \Bigr)^{1/2}
=\max_{i,j\in E_{1,k} }|a_{ij}|.
\]
Thus, the Gaussian concentration (see, e.g.,  \cite[Theorem~5.6]{BLM2013}),  
the first assumption of the proposition, and property \eqref{eq:estbyDinfty2} imply 
\begin{align*}
\Big(\Ex\|U_k\|_{p\to q}^{2^k}\Big)^{2^{-k}}
&\lesssim \Ex\|U_k\|_{p\to q}+2^{k/2}\max_{i,j\in E_{1,k}}|a_{ij}|
\\
&\le \alpha_1 D_1+\alpha_2 D_2+\Bigl(\alpha_3\sqrt{2\log|E_{1,k}|}+2^{k/2}\Bigr)\max_{i,j\in E_{1,k}}|a_{ij}|
\\ &
\lesssim  \alpha_1 D_1+\alpha_2 D_2+ \alpha_3D_\infty.
\end{align*}
Thus, 
\[
M\coloneqq \sup_{k}\Bigr(\Ex\|U_k\|_{p\to q}^{2^k}\Bigl)^{2^{-k}}\lesssim 
\alpha_1 D_1+\alpha_2 D_2+ \alpha_3D_\infty,
\]
and Lemma~\ref{lem:block-matr} yields
\begin{align*}
\Ex\|U\|_{p\to q}
&=\Ex\sup_{k\geq 1}\|U_k\|_{p\to q}\leq
2M+M\Ex\sum_{k\geq 1}\frac{\|U_k\|_{p\to q}}{M}\ind_{\{\|U_k\|_{p\to q}\geq 2M\}}
\\
&\leq 2M+M\Ex\sum_{k\geq 1}2^{1-2^k}\Bigl(\frac{\|U_k\|_{p\to q}}{M}\Bigr)^{2^k}
\leq M\Bigl(2+\sum_{k\geq 1}2^{1-2^k}\Bigr)\leq 3M
\\
&\lesssim \alpha_1 D_1+\alpha_2 D_2+ \alpha_3D_\infty.
\end{align*} 
In a similar way we show that 
\[
\Ex\|V\|_{p\to q}\lesssim  \alpha_1 D_1+\alpha_2 D_2+ \alpha_3D_\infty.
\]

Finally, fix $(i,j)\in E_3$ and take $k\in \{0,1,\ldots \}$ such that $M_k\le i< M_{ k+1}$, 
where $M_0=1$.
Observe that if $M_k\leq i\le N_{k+1}$, then either $j\leq N_k$ or $j> M_{ k+1}$ and if 
$N_{k+1}+1 \leq i<M_{k+1}$, then either $j\leq N_k$ or $j> M_{k+2}$. 
Therefore, either $j\leq N_k$ or $j> M_{k+1}$. If $j\leq N_k$, then by \eqref{eq:estoffdiag1}
\[
(i+j)^{(8(p^*\vee q))^{-1}}|a_{ij}|\leq 2N_{k+2}^{(8(p^*\vee q))^{-1}}D_22^{-2^{k-1}/q}
\le 2  D_2.
\]
If on the other hand $M_{k+l}< j \le M_{k+l +1}$ for some $l\geq  1$,  then    $i<N_{k+l}$ and \eqref{eq:estoffdiag2} imply
\[
(i+j)^{(8(p^*\vee q))^{-1}}|a_{ij}|\leq 2N_{k+l+2}^{(8(p^*\vee q))^{-1}} D_12^{-2^{k+l-1}/p^*}
\le 2D_1.
\]
Thus,  the second assumption of the proposition yields
\begin{align*}
\Ex \|W\|_{p\to q} &\le \beta_1D_1+\beta_2D_2+\beta_3 \max_{(i,j)\in E_3}(i+j)^{(8(p^*\vee q))^{-1}}|a_{ij}| 
\\
&\le (\beta_1+2\beta_3) D_1+(\beta_2+2\beta_3)D_2.	\qedhere
\end{align*}
\end{proof}

\begin{proof}[Proof of Theorem~\ref{thm:main}]
The two-sided estimate between the expressions on the right-hand side of the 
first and the last line was proven in 
\cite[Section~5.4]{APSS}.
This also yields the lower bound in the first asserted two-sided estimate.

The second two-sided estimate follows from
\begin{equation} 
\label{eq:Emax-vs-max-subsets}
\Ex\max_{i,j}|a_{ij}g_{ij}|
\sim 
\max_{ k\geq 0}\inf_{|I|=k}\sqrt{\Log k}	\max_{(i,j)\notin I} |a_{ij}|
\end{equation}
(cf. \cite[ Lemmas~2.3 and 2.4]{vH}).

The upper bound from the first two-sided estimate follows from
Propositions~\ref{prop:dimdependent}, \ref{prop:weakerD3-allpq} and \ref{prop:weaker-esimates-suffice} and  estimate \eqref{eq:Emax-vs-max-subsets}.

Now we move to the proof of the third two-sided estimate from Theorem~\ref{thm:main}. 
We will show a more precise bound
\begin{equation}
\label{DvsD'}
\sqrt{p^*}D_1+\sqrt{q}D_2+D_\infty\sim \sqrt{p^*}D_1+\sqrt{q}D_2+D_\infty',
\end{equation}
where
\[
D'_\infty\coloneqq \max_{ k\geq 0}\inf_{|I|=|J|=k}\max_{i\notin I,j\notin J} \sqrt{\Log k}\, |a_{ij}|.
\]
The  lower bound in \eqref{DvsD'} is trivial,  since $I\subset P_1(I)\times P_2(I)$, 
where $P_1$ and $P_2$ are coordinate projections.  To establish the upper bound in \eqref{DvsD'} we need to show that for any $k\geq 1$,
\begin{equation}
\label{eq:estKk}
 \sqrt{\Log k}\inf_{|V|=k}\max_{(i,j)\notin V}|a_{ij}|
 \lesssim  \sqrt{p^*}D_1+ \sqrt{q}D_2+D_\infty'.
\end{equation}
If $k\leq 3$, then
\[
\sqrt{\Log k}\inf_{|V|=k}\max_{(i,j)\notin V} |a_{ij}|\lesssim \max_{i,j}|a_{ij}|\leq D_1.
\]
If $k\geq 3$, then put $k'\coloneqq \lfloor \sqrt{k/3}\rfloor$ and choose $|I_0|=|J_0|=k'$ such that
\[
\inf_{|I|=|J|=k'}\max_{i\notin I, j\notin J} |a_{ij}|=\max_{i\notin I_0, j\notin J_0} |a_{ij}|.
\]
Let $V_0\coloneqq (I_0\times J_0)\cup V_1\cup V_2$, where
\begin{align*}
V_1&\coloneqq \{(i,j)\colon i\in I_0, |a_{ij}|\geq (k')^{-1/p*}D_1\},
\\ 
V_2&\coloneqq \{(i,j)\colon j\in J_0, |a_{ij}|\geq (k')^{-1/q}D_2\}.
\end{align*}
Then $|V_0|\leq |I_0||J_0|+|I_0|k'+|J_0|k'=3(k')^2\leq k$, so that
\begin{align*}
\inf_{|V|=k}\max_{(i,j)\notin V} |a_{ij}|
&\leq \max_{(i,j)\notin V_0} |a_{ij}|
\\
&\leq \max_{i\notin I_0,j\notin J_0} |a_{ij}|+\max_{(i,j)\in (I_0\times [n])\setminus V_1}|a_{ij}|
+\max_{(i,j)\in ([m]\times J_0)\setminus V_2}|a_{ij}|
\\
&\leq (\Log k')^{-1/2}D_\infty'+(k')^{-1/p*}D_1+(k')^{-1/q}D_2.
\end{align*}
We have   $\sqrt{\Log k}\lesssim 
\min\{\sqrt{\Log k'}, \sqrt{p^*}(k')^{1/p^*}, \sqrt{q}(k')^{1/q}\}$, 
so \eqref{eq:estKk} follows. 
\end{proof}

In the proof of Corollary~\ref{cor:Weibulls-below4} we will also use the following lemma from \cite{LvHY}.

\begin{lem}[{\cite[Lemma~4.7]{LvHY}}]
	\label{lem:cfconvex}
Let $h_i$ and $h_i'$, $i=1,\ldots,k$ be independent centered random 
variables that satisfy 
\[
	C_1\rho^{\beta} \le\Ex[|h|^\rho]^{1/\rho} \le
	C_2\rho^{\beta}
	\qquad\mbox{for all }\rho\ge 2.
\]
 Then there exists a constant $C$
depending only on $C_1,C_2,\beta$ such that
\[
	\Ex[f(h_1,\ldots,h_k)] \le
	\Ex[f(Ch_1',\ldots,Ch_k')]
\]
for every symmetric convex function $f:\mathbb{R}^k\to\mathbb{R}$.
\end{lem}

\begin{proof}[Proof of Corollary~\ref{cor:Weibulls-below4}]
Assume first that $p^*,q<\infty$.
Let $s>0$ be such that $\frac 1r =\frac 12 +\frac 1s$ (if $r=2$, then $s=\infty$) and let $Y_{ij}=g_{ij}|\widetilde{g}_{ij}|^{2/s}$ where $(\widetilde{g}_{ij})_{i,j}$ is an independent copy of $(g_{ij})_{i,j}$. Then for every $\rho \ge 1$,
\begin{equation}
\label{eq:momXijYij}
(\Ex |X_{ij}|^\rho)^{1/\rho} 
\sim_r \rho^{1/r} 
\sim_r (\Ex |Y_{ij}|^\rho)^{1/\rho}, 
\end{equation}
so  Theorem~\ref{thm:main} and Lemma~\ref{lem:cfconvex}  (applied twice)  imply
\begin{align*}	
\Ex\|(a_{ij}X_{ij})_{i, j}\|_{p\to q} 
& \sim_{p,q,r} 
\Ex \max_i \|(a_{ij}|g_{ij}|^{2/s})_j\|_{p^*} +\Ex\max_j \|(a_{ij}|g_{ij}|^{2/s})_i\|_{q} 
\\ 
&\qquad \quad 
+ \Ex \max_{ k\geq 0}\inf_{|I|=k}\sqrt{\Log k}\max_{(i,j)\notin I} |a_{ij}||g_{ij}|^{2/s}
\\ 
&\sim_{p,q, r} 
\Ex \max_i \|(a_{ij}X_{ij})_j\|_{p^*} + \Ex \max_j \|(a_{ij}X_{ij})_i\|_{q}.
\end{align*}
Moreover,
\begin{align*}
\Ex  \max_i \|(a_{ij}|g_{ij}|^{2/s})_j\|_{p^*}
&= \Ex  \max_i \|(|a_{ij}|^{s/2}g_{ij})_j\|_{2p^*/s}^{2/s}
\\ 
&\sim_r   \Bigl(\Ex  \max_i \|(|a_{ij}|^{s/2}g_{ij})_j\|_{2p^*/s}\Bigr)^{2/s}
\end{align*}
and similarly
\[
\Ex\max_j \|(a_{ij}|g_{ij}|^{2/s})_i\|_{q} 
\sim_r  \Bigl(\Ex  \max_j \|(|a_{ij}|^{s/2}g_{ij})_i\|_{2q/s}\Bigr)^{2/s},
\]
so \cite[equation (5.11)]{APSS} yields
\begin{align*}
& \Ex \max_i \|(a_{ij}|g_{ij}|^{2/s})_j\|_{p^*} +\Ex\max_j \|(a_{ij}|g_{ij}|^{2/s})_i\|_{q} 
\\
& \qquad \sim_{p,q, r}  \max_i \|(|a_{ij}|^{s/2})_j\|_{2p^*/s}^{2/s} 
+ \max_j \|(|a_{ij}|^{s/2})_i\|_{2q/s}^{2/s}
+ \Bigl(\Ex \max_{i,j} |a_{ij}|^{s/2}|g_{ij}| \Bigr)^{2/s}
\\ 
& \qquad\sim_r \max_i \|(a_{ij})_j\|_{p^*} +\max_j \|(a_{ij})_i\|_{q} 
+\Ex \max_{i,j} |a_{ij}||g_{ij}|^{2/s}.
\end{align*}
 Note also that
\begin{align*}
\Ex \max_{i,j} |a_{ij}||g_{ij}|^{2/s}
&\leq
 \Ex \max_{ k\geq 0}\inf_{|I|=k}\sqrt{\Log k}\max_{(i,j)\notin I} |a_{ij}||g_{ij}|^{2/s}
\\
&\sim \Ex \max_{i,j}|a_{ij}||g_{ij}|^{2/s}|\tilde{g}_{i,j}|
 =\Ex\max_{i,j}|a_{ij}Y_{ij}|
 \sim_r \Ex\max_{i,j}|a_{ij}X_{ij}|,
\end{align*}
where the second estimate follows from the conditional application of \eqref{eq:Emax-vs-max-subsets} and the last one 
by \eqref{eq:momXijYij}  and Lemma~\ref{lem:cfconvex}. Thus,
\begin{align*}
\Ex\|(a_{ij}X_{ij})_{i, j}\|_{p\to q} 
& \sim_{p,q,r} 
\max_i \|(a_{ij})_j\|_{p^*} +\max_j \|(a_{ij})_i\|_{q} +\Ex\max_{i,j}|a_{ij}	X_{ij}|.
\end{align*}

Let $Z_{ij}=\ve_{ij}|{g}_{ij}|^{2/r}$, where $\ve_{ij}$, $i\le m$, $j\le n$, are iid symmetric Bernoulli variables, independent of $(g_{ij})_{i,j}$. 
Then 
\[
(\Ex |X_{ij}|^\rho)^{1/\rho} 
\sim_r \rho^{1/r} 
\sim_r (\Ex |Z_{ij}|^\rho)^{1/\rho}, 
\]
so Lemma~\ref{lem:cfconvex} and \eqref{eq:Emax-vs-max-subsets} yield
\begin{align}
\notag
\Ex\max_{i,j}|a_{ij} X_{ij}| & \sim_r \Ex\max_{i,j}|a_{ij} Z_{ij}| 
=\Ex\max_{i,j}|a_{ij}| |g_{ij}|^{2/r} 
\\ \notag
& \sim_r \bigl(\Ex\max_{i,j}|a_{ij}|^{r/2} |g_{ij}| \bigr)^{2/r} 
\sim \bigl(\max_{ k\geq 0}\inf_{|I|=k}\sqrt{\Log k}\max_{(i,j)\notin I} |a_{ij}|^{r/2}\bigr)^{2/r}
\\ \label{eq:proof-cor-Weibull-1}
&=\max_{ k\geq 0}\inf_{|I|=k}{\Log^{1/r} k}\max_{(i,j)\notin I} |a_{ij}|.
\end{align}

Finally, the third two-sided estimate in Corollary \ref{cor:Weibulls-below4} 
may be established in a similar way as
in the Gaussian case (see the last part of the proof of Theorem~\ref{thm:main}).

If $p^*=\infty$ (i.e., $p=1$), then we proceed similarly, using inequality~\eqref{eq:ineq-p=1} instead of Theorem~\ref{thm:main}, to prove that for every $q\in [2,\infty)$, 
\begin{align*}
\Ex\| (a_{ij}X_{ij}  )_{i\le m, j\le n}\|_{1\to q}
&\sim_{q,r} 
\max_j \|(a_{ij})_i\|_{q} +\Ex\max_{i,j}|a_{ij}X_{ij}|
\\ & \sim_{  r} 
\max_j \|(a_{ij})_i\|_{q} 
+\max_{k\ge 0}\inf_{|I|=k}\Log^{1/r} k\max_{(i,j)\notin I} |a_{ij}|
\\
&  \sim_{ q, r} 
\Ex\max_{i,j}|a_{ij}X_{ij}|+\Ex\max_j\|(a_{ij}X_{i,j})_i\|_{q}.
\end{align*}
Moreover, if $p^*=q=\infty$, then \eqref{eq:proof-cor-Weibull-1} yields that
\[
\Ex\| (a_{ij}X_{ij}  )_{i\le m, j\le n}\|_{1\to \infty} = \Ex \max_{i,j} |a_{i,j}X_{i,j}|
\sim_r \max_{k\ge 0}\inf_{|I|=k}\Log^{1/r} k\max_{(i,j)\notin I} |a_{ij}|.
\]

The last case, $q=\infty>p^*\ge 2$, follows by duality.
\end{proof}

\begin{proof}[Proof of Corollary~\ref{cor:generalentries}]
Let $(\ve_{ij})_{i,j}$ and $(g_{ij})_{i,j}$ be independent random matrices with 
independent symmetric $\pm 1$
and $\mathcal{N}(0,1)$ entries, respectively, and assume that they are independent of 
$(X_{ij})_{i,j}$. Let  $Y_{ij}=g_{ij}X_{ij}$  and 
$Z_{ij}^{(\rho)}=|Y_{ij}|^{\rho}-\Ex|Y_{ij}|^{\rho}$ for $\rho\in \{p^*,q\}$. 
Since $X_{ij}$ are centered,
\begin{align*}
\Ex\|(X_{ij})\|_{p\to q}
& \le 2 \Ex\|(\ve_{ij}X_{ij})\|_{p\to q}
\le \frac{2}{\Ex|g_{11}|} \Ex\|(Y_{ij})\|_{p\to q}
\\
&\sim_{p,q} \Ex\max_{i}\|(Y_{ij})_{j}\|_{p^*}+\Ex\max_{j}\|(Y_{ij})_{i}\|_{q},
\end{align*}
 where the last bound follows from Corollary \ref{cor:mixtures}.

Since  for every $a,b\ge 0$, $(a+b)^{1/q}\le a^{1/q}+b^{1/q}$, we have
\[
\Ex\max_{j}\|(Y_{ij})_{i}\|_{q}\leq \max_{j}\Bigl(\sum_{i}\Ex |Y_{ij}|^{q}\Bigr)^{1/q}
+\Ex\max_{j} \Bigl|\sum_{i}(|Y_{ij}|^{q}-\Ex|Y_{ij}|^q) \Bigr|^{1/q}.
\]
By the independence of $g_{ij}$ and $X_{ij}$ we have
\begin{align*}
 \max_{j}\Bigl(\sum_{i}\Ex |Y_{ij}|^{q}\Bigr)^{1/q}
& =(\Ex|g_{11}|^q)^{1/q}\max_{j}\Bigl(\sum_{i}\Ex |X_{ij}|^{q}\Bigr)^{1/q}
\\ & \le  \sqrt{q} \max_{j}\Bigl(\sum_{i}\Ex |X_{ij}|^{q}\Bigr)^{1/q}.
\end{align*}
Moreover,
by the Rosenthal inequality  \cite[Theorem 1.5.11]{dPG} we have for every $r\geq 2$,
\begin{align*}
\Ex\max_{j} \Bigl|\sum_{i}Z_{ij}^{(q)} \Bigr|^{1/q}
&\leq \Ex\Bigl(\sum_{j}\Bigl|\sum_{i}Z_{ij}^{(q)}\Bigr|^r\Bigr)^{1/(rq)}
\leq \Bigl(\sum_{j}\Ex\Bigl|\sum_{i}Z_{ij}^{(q)}\Bigr|^r\Bigr)^{1/(rq)}
\\
&\lesssim 
\Bigl( \frac{r}{\Log r}\Bigr)^{1/q} \Bigl(\sum_{j}\Bigl(\sum_{i}\Ex\bigl|Z_{ij}^{(q)}\bigr|^2\Bigl)^{r/2}
+\sum_{i,j}\Ex\bigl|Z_{ij}^{(q)}\bigr|^r\Bigr)^{1/(rq)}.
\end{align*}

Observe that for $u\geq 1$ we have
\[
\Ex|Z_{ij}^{(q)}|^u\leq 2^u\Ex|Y_{ij}|^{qu}=2^u\Ex|g_{ij}|^{qu}\Ex|X_{ij}|^{qu}
\leq 2^u(qu)^{(qu)/2}\Ex|X_{ij}|^{qu}.
\]
Therefore for any $r\geq 2$,
\begin{align*}
\Ex\max_{j}\|(Y_{ij})_{i}\|_{q}
&\lesssim_{q,r}
\max_{j}\Bigl(\sum_{i}\Ex |X_{ij}|^{q}\Bigr)^{1/q}
\\
&\quad\quad+\Bigl(\sum_{j}\Bigl(\sum_{i}\Ex|X_{ij}|^{2q}\Bigl)^{r/2}
+\sum_{i,j}\Ex|X_{ij}|^{rq}\Bigr)^{1/(rq)}.
\end{align*}

In a similar way we show that 
 for every $s\geq 2$,
\begin{align*}
\Ex\max_{i}\|(Y_{ij})_{j}\|_{p^*}
\lesssim_{p,s} &
\max_{i}\Bigl(\sum_{j}\Ex |X_{ij}|^{p^*}\Bigr)^{1/p^*}
\\
&\quad+\Bigl(\sum_{i}\Bigl(\sum_{j}\Ex|X_{ij}|^{2p^*}\Bigl)^{s/2}
+\sum_{i,j}\Ex|X_{ij}|^{sp^*}\Bigr)^{1/(sp^*)}.
\end{align*}
Estimate \eqref{eq:genentr2} follows if we take $r=2(p^*\vee q)/q$ and $s=2(p^*\vee q)/p^*$.
\end{proof}

%%%%%%%%%%%%%%%%%%%%%%%%
%%%%%%%%%%%%%%%%%%%%
%%%%%%%%%%%%%%%%%%%
\section{Dimension dependent bounds}
\label{sect:dim-dep}

In this section our aim is to  obtain the dimension dependent bound from 
Proposition~\ref{prop:dimdependent} in the case $(p^*,q)\notin[2,3]^2$.
We begin by proving a weaker estimate, asserted in Proposition~\ref{prop:worseD3'-intro}.
Then, in Subsections~\ref{subsect:logexp-reduction} and \ref{sect:degdepbounds}, 
we  show how to reduce the exponent of the logarithm appearing in this proposition 
in the case $p^*\vee q>2$. 
The proof of Proposition~\ref{prop:dimdependent} in the case $(p^*,q)\notin[2,3]^2$ 
is provided at the end of this section.

The methods provided in this section work in the whole range $p^*\vee q>2$, but they yield the constant which blows up when $p^*\vee q $ approaches 2.
In order to obtain the bound with a universal constant in the case when $p^*$ and $q$ are close to 2, we need  more involved combinatorial arguments and some additional estimates obtained by the Slepian-Fernique lemma to run the dimension reduction. 
We provide these additional tools in Section~\ref{sect:appendix}, together with the proofs of Propositions~\ref{prop:dimdependent} and \ref{prop:worseD3'-intro} in the case $p^*,q\in[2,3]$.

\subsection{Proof of Propostion~\ref{prop:worseD3'-intro} in the case \texorpdfstring{$(p^*,q)\notin [2,3]^2$}{p*,q not in [2,3]x[2,3]}.} 
\label{subsect:proof-worseD3'}

In order to prove Proposition~\ref{prop:worseD3'-intro} we split each 
$s\in B_{q^*}^m$ into two parts: one consisting of vectors with coordinates 
which do not exceed a certain level, and the second one consisting of vectors with 
non-vanishing coordinates having absolute value exceeding the same level.
The next proposition shows how  to control the suprema over points with large non-zero coordinates.

\begin{prop}
\label{prop:largecoordinates}
If $\gamma\ge 1$ and $p^*,q\ge 2$, then
\begin{align}
\Ex  \sup_{s\in B_{q*}^m}  
\sup_{t\in B_p^n}\sum_{i,j}a_{ij}g_{ij}s_it_j\ind_{\{|s_i|\ge \gamma^{-1}\}}
\label{eq:largecoordinates}
\lesssim  
\sqrt{p^*}D_1 +  \sqrt{\gamma^{q^*} \Log m} \max_{i,j}|a_{ij}| 
\end{align}
and 
\begin{align*}
\Ex\sup_{s\in B_{q*}^m}  \sup_{t\in B_p^n} 
\sum_{i,j}a_{ij}g_{ij}s_it_j \ind_{\{|t_j|\ge \gamma^{-1}\}}
\lesssim  
\sqrt{q}D_2 +  \sqrt{\gamma^{p} \Log n} \max_{i,j}|a_{ij}|.
\end{align*}
\end{prop}

In the proof of Proposition~\ref{prop:largecoordinates} we shall  use the following standard lemma; 
we formulate and prove it for the sake of completeness.

\begin{lem}
For every fixed $s\in B_{q^*}^m$, $q\in [2,\infty]$, and $p\in (1,2]$,
\begin{equation}
\label{eq:fixeds}
\Ex \sup_{t\in B_p^n}\sum_{i,j}a_{ij}g_{ij}s_it_j
=\Ex\Bigl\|\Bigl(\sum_{i}a_{ij}g_{ij}s_i\Bigr)_j\Bigr\|_{p^*}\leq \sqrt{p^*} D_1.
\end{equation}
Similarly, for every fixed $t\in B_{p}^n$, $p\in [1,2]$,  and $q\in [2,\infty)$, 
\begin{equation}
\label{eq:fixedt}
\Ex \sup_{s\in B_{q^*}^m}\sum_{i,j}a_{ij}g_{ij}s_it_j
=\Bigl\|\Bigl(\sum_{j}a_{ij}g_{ij}t_j\Bigr)_i\Bigr\|_{q}\leq \sqrt{q}D_2.
\end{equation}
\end{lem}

\begin{proof}
We have
\begin{align*}
\Ex\sup_{t\in B_p^n}\sum_{i,j}a_{ij}g_{ij}s_it_j
&=\Ex\Bigl(\sum_{j}\Bigl(\sum_{i}a_{ij}g_{ij}s_i\Bigr)^{p^*}\Bigr)^{1/p^*}
\leq \Bigl(\Ex\sum_{j}\Bigl(\sum_{i}a_{ij}g_{ij}s_i\Bigr)^{p^*}\Bigr)^{1/p^*}
\\
&=\|g_{1,1}\|_{p^*}\Bigl(\sum_{j}\Bigl(\sum_i a_{ij}^2s_i^2\Bigr)^{p^*/2}\Bigr)^{1/p^*}.
\end{align*}
Moreover, $\|g_{11}\|_{p^*}\leq \sqrt{p^*}$ and, since $\|s\|_2\leq \|s\|_{q^*}\leq 1$,
\begin{align*}
\Bigl(\sum_{j}\Bigl(\sum_i a_{ij}^2s_i^2\Bigr)^{p^*/2}\Bigr)^{2/p^*}
&=\Bigl\|\Bigl(\sum_{i}s_i^2a_{ij}^2\Bigr)_j\Bigr\|_{p^*/2}
\leq \sum_{i}s_i^2\|(a_{ij}^2)_j\|_{p^*/2}
\\
&\leq \max_{i}\|(a_{ij}^2)_j\|_{p^*/2}=\max_{i}\|(a_{ij})_j\|_{p^*}^{ 2} =D_1^2.
\end{align*}
Hence, \eqref{eq:fixeds} follows. Estimate \eqref{eq:fixedt} holds by an analogous argument. 
\end{proof}

We shall also use the  following lemma, which follows from the Gaussian concentration  
(see, e.g., \cite[Theorem~5.6]{BLM2013})  and \cite[Lemma~2.3]{vH}, applied with 
$X_i\coloneqq \sup_{t\in T_i} X_t - \Ex \sup_{t\in T_i} X_t$ and 
$\sigma_i\coloneqq\sup_{t\in T_i} (\Var X_t)^{1/2}$.
\begin{lem}
\label{lem:supgauss}
Let $X=(X_t)_{t\in T}$ be a Gaussian process and let $T_1,\ldots,T_k$ be nonempty subsets of 
$T$ such that $T=\bigcup_{i=1}^kT_i$. Then
\[
\Ex \sup_{t\in T}X_t \leq  \max_{i}\Ex\sup_{t\in T_i}X_{t}
+ C\max_i \sqrt{\Log i}\, \sup_{t\in T_i}(\Var X_t)^{1/2}.
\]
In particular, if $X$ is centered, then
\begin{equation}	
\label{eq:supgauss}
\Ex \sup_{t\in T}X_t 
\leq  \max_{i}\Ex\sup_{t\in T_i}X_{t}+ C \sqrt{\Log k}\, \sup_{t\in T}(\Ex X_t^2)^{1/2}.
\end{equation}
\end{lem}

\begin{proof}[Proof of Proposition~\ref{prop:largecoordinates}]
By the symmetry it is enough to show \eqref{eq:largecoordinates}. 
For a fixed nonempty set $I$ let
\[
X_I\coloneqq \|(a_{ij}g_{ij})_{i\in I,j\leq n}\|_{p\to q}.
\]
Let $S$ be a $1/2$-net in $B_{q^*}^I$ (with respect to the $\ell_{q^*}$ norm) of cardinality 
at most $5^{|I|}$. Then 
\[
X_I\leq 2\sup_{s\in S}\sup_{t\in B_p^n}\sum_{i,j}a_{ij}g_{ij}s_it_j.
\]
Thus, inequality \eqref{eq:supgauss}  yields
\[
\Ex X_I
\lesssim \sup_{s\in S}\Ex\sup_{t\in B_p^n}\sum_{i,j}a_{ij}g_{ij}s_it_j
+
\sqrt{\log |S|}\, \sup_{s\in S}\sup_{t\in B_p^n}\Bigl(\sum_{i,j}a_{ij}^2s_i^2t_j^2\Bigr)^{1/2}.
\]
Since $S\subset B_{q^*}^m\subset B_2^m$ and $B_p^n\subset B_2^n$ we have
\begin{equation} 
\label{eq:lipconst}
\sup_{s\in S}\sup_{t\in B_p^n}\Bigl(\sum_{i,j}a_{ij}^2s_i^2t_j^2\Bigr)^{1/2}
\leq \max_{i,j}|a_{ij}|.
\end{equation}

Estimate \eqref{eq:fixeds}  implies that for any fixed $s\in S$,
\[
\Ex\sup_{t\in B_p^n}\sum_{i,j}a_{ij}g_{ij}s_it_j
\leq \sqrt{p^*}D_1,
\]
hence,
\begin{equation}
\label{eq:supoverI}
\Ex X_I\lesssim \sqrt{p^*}D_1+\sqrt{|I|}\max_{i,j}|a_{ij}|.
\end{equation}

 There are at most $Cm^{\gamma^{q^*}}$ subsets of $[m]$ of cardinality at most $\gamma^{q^*}$. 
Thus, estimates \eqref{eq:supgauss}, \eqref{eq:lipconst}, and \eqref{eq:supoverI}  
yield
\begin{align*}
\Ex  \max_{ |I|\leq \gamma^{q^*}}X_I
&\lesssim \max_{|I|\leq \gamma^{q^*}}\Ex X_I
 +\sqrt{\gamma^{q^*} \Log m}\sup_{s\in B_{q*}^m,t\in B_p^n}\Bigl(\sum_{i,j}a_{ij}^2s^2_it_j^2\Bigr)^{1/2}
\\
&\lesssim \sqrt{p^*}D_1
+ \sqrt{\gamma^{q^*} \Log m} \max_{i,j}|a_{ij}|.
 \end{align*}
To  derive \eqref{eq:largecoordinates} it suffices  to  note that for $s\in B_{q^*}^m$, 
$|\{i\colon |s_i|\geq \gamma^{-1}\}|\leq \gamma^{q^*}$,
so
 \[
 \Ex  \sup_{s\in B_{q*}^m}  
\sup_{t\in B_p^n}\sum_{i,j}a_{ij}g_{ij}s_it_j\ind_{\{|s_i|\ge \gamma^{-1}\}}
\le 
\Ex  \sup_{s\in B_{q*}^m, |\supp(s)|\leq  \gamma^{q^*}}\ \sup_{t\in B_p^n}\sum_{i,j}a_{ij}g_{ij}s_it_j.
\qedhere
 \]
\end{proof}
	
The next proposition shows how to control the suprema over points with small coordinates. It is a crucial step in the proof of Proposition~\ref{prop:worseD3'-intro}.

\begin{prop}
\label{prop:bdd-coordinates-q}
If $q>2$, $ p^*\ge 2$, and $a\in (0,1]$, then
\begin{equation}
\label{eq:toshow2}
\Ex\sup_{s\in B_{q^*}^m\cap aB_\infty^m}\sup_{t\in B_p^n}
\sum_{i\leq m,j\leq n}a_{ij}g_{ij}s_it_j
\lesssim \sqrt{q}D_2+a^{(2-q^*)/2}D_1\Log^{3/2}n .
\end{equation}
If $p^*>2$, $q\ge 2$, and $b\in (0,1]$, then
\begin{equation}
\label{eq:toshow2-dual}
\Ex\sup_{s\in B_{q^*}^m}\sup_{t\in B_p^n\cap bB_\infty^n}
\sum_{i\leq m,j\leq n}a_{ij}g_{ij}s_it_j
\lesssim \sqrt{p^*}D_1+b^{ (2-p)/2}D_2\Log^{3/2}m .
\end{equation}

\end{prop}

To prove Proposition~\ref{prop:bdd-coordinates-q} we need the following two lemmas. For a given $\lambda\in (0,1]$, let us first  define the following interpolation norm on $\er^m$,
 \[
\|x\|_{2,q,\lambda}\coloneqq \inf\bigl\{\bigl(\lambda^2\|y\|_2^2+\|z\|_q^2\bigr)^{1/2}\colon\ x=y+z\bigr\}.
\]
Observe that for $q\ge 2$, $B_{q^*}^m\cap aB_\infty^m\subset a^{(2-q^*)/2}B_2^m$, thus
\begin{equation}
\label{eq:redtointnorms}
\sup_{s\in B_{q^*}^m\cap aB_\infty^m}\sum_{i\leq m}s_ix_i\leq \sqrt{2}\|x\|_{2,q,a^{(2-q^*)/2}}.
\end{equation}

The following lemma presents a crucial property of this norm. This will allow us to perform the induction step in the proof of Proposition~\ref{prop:bdd-coordinates-q}.
We will use its second assertion later, in Section~\ref{sect:decomposition}.

\begin{lem}
\label{lem:interpol}
For every $q\ge  2$, $b,c\in \er^m$, $d\ge 0$, and $\lambda\in (0,1]$,
\begin{equation}
\label{eq:intnormprop}
\Ex\bigl(d+\|(b_i+c_ig_i)_{i\leq m}\|_{2,q,\lambda}^2\bigr)^{1/2}
\leq
\bigl(d+\|b\|_{2,q,\lambda}^2+q\|c\|_q^2\bigr)^{1/2}
\end{equation}
and
\begin{equation}
\label{eq:intnormprop-q-norm}
\Ex\Bigl( d+ \|(b_i+c_ig_i)_{i\leq m}\|_q^2\Bigr)^{1/2}
\leq 
\bigl(d+\|b\|_q^2+q\|c\|_q^2\bigr)^{1/2}.
\end{equation}
\end{lem}

\begin{proof}
Let $b',b''\in \er^m$. 
Then  Jensen's inequality implies
\begin{multline*}
\Ex\bigl(d+\lambda^2 \|b'\|_2^2+\|(b_i''+c_ig_i)_{i\leq m}\|_{q}^2\bigr)^{1/2}
\\ \leq 
\Bigl(d+\lambda^2 \|b'\|_2^2+\Bigl(\sum_{i\leq m}\Ex|b_i''+c_ig_i|^{q}\Bigr)^{2/q}\Bigr)^{1/2}.
\end{multline*}
The hypercontractivity of Gaussian variables and the triangle inequality in $L_{q/2}$ yield
\begin{align*}
\Bigl(\sum_{i\leq m}\Ex|b_i''+c_ig_i|^q\Bigr)^{2/q}
&\leq \Bigl(\sum_{i\leq m}(\Ex|b_i''+\sqrt{q}c_ig_i|^2)^{q/2}\Bigr)^{2/q}
= \|(b_i''^2+qc_i^2)_i\|_{q/2}
\\
&\leq \|(b_i''^2)_i\|_{q/2}+q\|(c_i^2)_i\|_{q/2}=\|b''\|_{q}^2+q\|c\|_q^2.
\end{align*}
Thus,
\begin{align*}
\Ex\bigl(d+\lambda^2 \|b'\|_2^2+\|(b_i''+c_ig_i)_{i\leq m}\|_{q}^2\bigr)^{1/2}
&\leq 
\bigl(d+\lambda^2 \|b'\|_2^2+\|b''\|_{q}^2+q\|c\|_q^2\bigr)^{1/2}.
\end{align*}

Inequality~\eqref{eq:intnormprop} follows by taking $b'$ and $b''$ such that $b=b'+b''$ and
\[
\|b\|_{2,q,\lambda}^2=\lambda^2 \|b'\|_2^2+\|b''\|_q^2,
\]
and observing that
\[
\|(b_i+c_ig_i)_{i\leq m}\|_{2,q,\lambda}^2
\leq \lambda^2 \|b'\|_2^2+\|(b_i''+c_ig_i)_{i\leq m}\|_{q}^2,
\]
while inequality~\eqref{eq:intnormprop-q-norm} follows by taking $b'=0$ and $b''=b$. 
\end{proof}

For a given $p\ge 1$ and $k=0,1,\ldots$ let
\[
T_{\leq k}=\bigl\{t\in \er^n\colon\ \|t\|_p\leq 1, |t_j|\in \{0,1,1/2,\ldots,1/2^k\},\ 
1\leq j\leq n\bigr\}.
\]
Choose $k_0=k_0(n)\coloneqq  \lceil \log_2 n\rceil+ 2\sim \Log n$.
It is standard to see that for any norm $\|\cdot\|$ on $\er^n$
\begin{equation}
\label{eq:redtoT}
\sup_{\|t\|_{p}\leq 1}\|t\|\leq 4\max_{t\in T_{\leq k_0}}\|t\|.
\end{equation}
 Indeed, we decompose $t\in B_p^n$ as $t=t'+t''$, where
\[
t'_i=
\begin{cases}
2^{-k}\mathrm{sgn}(t_i)& \mbox{if }  2^{-k}\leq |t_i|<2^{-k+1} \mbox{ for }k=0,1,\ldots,k_0,
\\
0&\mbox{if }|t_i|<2^{-k_0}.
\end{cases}
\]
and $t''=t-t'$. Then $t'\in T_{\leq k_0}$. Moreover, $\|t''\|_p\leq 3/4$, since 
$|t''_i|\leq \frac{1}{2}|t_i|$ if $|t_i|\geq 2^{-k_0}$ and 
$\sum_{i\colon |t_i|<2^{-k_0}}|t_i|^p\leq 2^{-k_0p}n\leq 2^{- 2p}$.
Hence,
\[
\sup_{\|t\|_{p}\leq 1}\|t\|\leq \max_{t'\in T_{\leq k_0}}\|t'\|+\sup_{\|t''\|_{p}\leq 3/4}\|t''\|_p
= \max_{t\in T_{\leq k_0}}\|t\|+\frac{3}{4}\sup_{\|t\|_{p}\leq 1}\|t\|_p.
\]

The next lemma provides the crucial bound we need in the proof of Proposition~\ref{prop:bdd-coordinates-q}.

\begin{lem}
\label{lem:supoverTk}
For any  $p\leq 2$ and $k=0,1,\ldots$ we have
\[
\Ex\max_{t\in T_{\leq k}}
\Bigl\|\Bigl(\sum_{j\leq n}t_ja_{ij}g_{ij}\Bigr)_{i\leq m}\Bigr\|_{2,q,\lambda}
\leq \sqrt{q}D_2+C\lambda D_1\sum_{l=0}^k2^{-lp/2}\sqrt{\Log|T_{\leq l}|}.
\]
\end{lem}

\begin{proof}
We will establish by induction on $k$ the following stronger bound:
\begin{align}
\label{eq:strongbound}
\Ex\max_{t\in T_{\leq k}}
\Bigl(q(1-\|t\|_2^2)D_2^2
&+\Bigl\|\Bigl(\sum_{j\leq n}t_ja_{ij}g_{ij}\Bigr)_{i\leq m}\Bigr\|_{2,q,\lambda}^2\Bigr)^{1/2}
\\
\notag
&\leq \sqrt{q}D_2+C\lambda D_1 \sum_{l=0}^k 2^{-lp/2}\sqrt{\Log|T_{\leq l}|}.
\end{align}

We will first show the induction step, i.e., that for $k=1,2,\ldots$,
\begin{align}
\label{eq:indstep}
\Ex\max_{t\in T_{\leq k}}
&\Bigl(q(1-\|t\|_2^2)D_2^2
+\Bigl\|\Bigl(\sum_{j\leq n}t_ja_{ij}g_{ij}\Bigr)_{i\leq m}\Bigr\|_{2,q,\lambda}^2\Bigr)^{1/2}
\\
\notag
&\leq \Ex\max_{t\in T_{\leq k-1}}
\Bigl(q(1-\|t\|_2^2)D_2^2
+\Bigl\|\Bigl(\sum_{j\leq n}t_ja_{ij}g_{ij}\Bigr)_{i\leq m}\Bigr\|_{2,q,\lambda}^2\Bigr)^{1/2}
\\
\notag
&\qquad +C\lambda D_12^{-kp/2}\sqrt{\Log|T_{\leq k}|}.
\end{align}

For $t\in T_{\leq k}$ define
\begin{align*}
J_k(t) & \coloneqq\{j\leq n\colon |t_j|=2^{-k}\},
\quad \pi_k(t)\coloneqq(t_j\ind_{\{j\in J_k\}})_j,
\\
J_{<k}(t)&\coloneqq\{j\leq n\colon |t_j|>2^{-k}\},\quad \pi_{<k}(t)\coloneqq(t_j\ind_{\{j\in J_{<k}\}})_j,
\\
X_i(k,t)&\coloneqq\sum_{j\in J_k(t)}t_ja_{ij}g_{ij},\quad Y_i(k,t)\coloneqq\sum_{j\in J_{<k}(t)}t_ja_{ij}g_{ij},
\\
Z_k(t)&\coloneqq\Bigl(q(1-\|t\|_2^2)D_2^2+\bigl\|\bigl(X_i(k,t)+Y_i(k,t)\bigr)_{i\leq m}\bigr\|_{2,q,\lambda}^2\Bigr)^{1/2},
\\
Z_{<k}(t)&\coloneqq\Bigl(q(1-\|t\|_2^2+\|\pi_{k}(t)\|_2^2)D_2^2
+\bigl\|\bigl(Y_i(k,t)\bigr)_{i\leq m}\bigr\|_{2,q,\lambda}^2\Bigr)^{1/2}.
\end{align*}

Observe that
\begin{align}
\label{eq:suprepr1}
\max_{t\in T_{\leq k}}
&\Bigl(q(1-\|t\|_2^2)D_2^2
+\Bigl\|\Bigl(\sum_{j\leq n}t_ja_{ij}g_{ij}\Bigr)_{i\leq m}\Bigr\|_{2,q,\lambda}^2\Bigr)^{1/2}
=\max_{t\in T_{\leq k}}Z_k(t)
\end{align}
and
\begin{align}
\label{eq:suprepr2}
\max_{t\in T_{\leq k-1}}
&\Bigl(q(1-\|t\|_2^2)D_2^2
+\Bigl\|\Bigl(\sum_{j\leq n}t_ja_{ij}g_{ij}\Bigr)_{i\leq m}\Bigr\|_{2,q,\lambda}^2\Bigr)^{1/2}
=\max_{t\in T_{\leq k}}Z_{<k}(t).
\end{align}

Let $\Ex_{k,t}$ denote the integration with respect to variables 
$(g_{i,j})_{i\leq m,j\in J_k(t)}$.
Conditional application of \eqref{eq:intnormprop} yields for any $t\in T_{\leq k}$,
\begin{align*}
\Ex_{k,t}Z_{k}(t)\leq
\Bigl(q(1-\|t\|_2^2)D_2^2+q\|((\Ex X_i^2(k,t))^{1/2})_{i\le m}\|_q^2+
\bigl\|\bigl(Y_{i}(k,t)\bigr)_{i\leq m}\bigr\|_{2,q,\lambda}^2\Bigr)^{1/2}.
\end{align*}
Convexity of the function $x\to |x|^{q/2}$ yields
\begin{align*}
\|((\Ex X_i^2(k,t))^{1/2})_{i\le m}\|_q^q
&=\sum_{i\leq m}\Bigl|\sum_{j\in J_k(t)}a_{ij}^2t_j^2\Bigr|^{q/2}
\leq \|\pi_k(t)\|_2^q\sum_{j\in J_k(t)}\frac{t_j^2}{\|\pi_k(t)\|_2^2}\sum_{i\leq m}|a_{ij}|^q
\\
&\leq \|\pi_k(t)\|_2^q\max_{j\in J_k(t)}\sum_{i\leq m}|a_{ij}|^q\leq \|\pi_k(t)\|_2^qD_2^q.
\end{align*}
Hence 
\[
\Ex_{k,t}Z_{k}(t)\leq Z_{<k}(t).
\]
For a fixed $d\geq 0$ and $y\in \er^m$ define functions $f\colon\er^m\to \er$ and 
$f_1,\ldots,f_m\colon\er^{|J_k(t)|}\to \er$ via 
\[
f(z)\coloneqq\Bigl(d+\bigl\|\bigl(z_i+y_i\bigr)_{i\leq m}\bigr\|_{2,q,\lambda}^2\Bigr)^{1/2},\quad
f_i(x_i)\coloneqq\sum_{j\in J_{k}(t)}a_{ij}t_jx_{ij}.
\]

The triangle inequalities for the Euclidean norm in $\er^2$ and for the interpolation norm $\|\cdot\|_{2,q,\lambda}$ imply that any $z,z'\in\er^m$.
\begin{align*}
	|f(z)-f(z')| 
	& \le \Bigl((\sqrt{d}-\sqrt{d})^2+\bigl(\bigl\|\bigl(z_i+y_i\bigr)_{i\leq m}\bigr\|_{2,q,\lambda}-\bigl\|\bigl(z_i'+y_i\bigr)_{i\leq m}\bigr\|_{2,q,\lambda}\bigr)^2\Bigr)^{1/2}
	\\ &\le \bigl\|\bigl((z_i-z_i')+y_i\bigr)_{i\leq m}\bigr\|_{2,q,\lambda}.
\end{align*}
Thus, $f$ is Lipschitz with a constant $\sup_{\|z\|_2\leq 1}\|z\|_{2,q,\lambda}\leq   \sup_{\|z\|_2\leq 1 (\lambda\|z\|_2) =} \lambda$. Moreover, since $ (\tfrac p{2-p})^*=\tfrac{p^*}{2}$,
$f_i$ is Lipschitz with a constant
\begin{align*}
\Bigl(\sum_{j\in J_k(t)}a_{ij}^2t_j^2\Bigr)^{1/2}
& \leq 2^{-kp/2}\Bigl(\sum_{j\leq n}a_{ij}^2|t_j|^{2-p}\Bigr)^{1/2}
\leq 2^{-kp/2} \sup_{y\in B_{p/(2-p)}^n}\Bigl(\sum_{j\leq n}a_{ij}^2y_j\Bigr)^{1/2}
\\ & \leq 2^{-kp/2}D_1.
\end{align*}
Hence, the function $x\to f(f_1(x_1),\ldots,f_m(x_m))$ is $\lambda 2^{-kp/2}D_1$-Lipschitz 
on $\er^{|J_k|m}$. Therefore, the Gaussian concentration yields
\[
\Pr_{k,t} \bigl(Z_k(t)\geq Z_{<k}(t)+u\lambda 2^{-kp/2}D_1 \bigr)\leq e^{-u^2/2} \quad \mbox{for } u\geq 0
\]
and thus,
\begin{align*}
\Pr\bigl(\max_{t\in T_{\leq k}}Z_k(t)\geq \max_{t\in T_{\leq k}} & Z_{<k}(t)+u\lambda 2^{-kp/2}D_1 \bigr)
\\ &\leq \sum_{t\in T_{\leq k}}\Ex\Pr_{k,t}\bigl(Z_k(t)\geq Z_{<k}(t)+u\lambda 2^{-kp/2}D_1\bigr)
\leq |T_{\leq k}|e^{-u^2/2}.
\end{align*}
Integration by parts yields 
\[
\Ex \Bigl(\max_{t\in T_{\leq k}}Z_k(t)- \max_{t\in T_{\leq k}} Z_{<k}(t)\Bigr)_+\leq 
C\lambda 2^{-kp/2}D_1\sqrt{\log |T_{\leq k}|},
\]
so 
\[
\Ex\max_{t\in T_{\leq k}}Z_k(t)\leq \Ex\max_{t\in T_{\leq k}}Z_{<k}(t)
+C\lambda 2^{-kp/2}D_1\sqrt{\log |T_{\leq k}|},
\]
which, in view of \eqref{eq:suprepr1} and \eqref{eq:suprepr2}, implies \eqref{eq:indstep}.

In a similar (in fact a bit simpler)  way we show that
\[
\Ex\max_{t\in T_{\leq 0}}
\Bigl(q(1-\|t\|_2^2)D_2^2
+\Bigl\|\Bigl(\sum_{j\leq n}t_ja_{ij}g_{ij}\Bigr)_{i\leq m}\Bigr\|_{2,q,\lambda}^2\Bigr)^{1/2}
\leq \sqrt{q}D_2+C\lambda D_1\sqrt{\Log|T_{\leq 0}|},
\]
which together with \eqref{eq:indstep} completes the inductive proof of \eqref{eq:strongbound}.
\end{proof}

\begin{lem}
\label{lem:cardTk}
We have $\log|T_{\leq k}|\lesssim 2^{pk}\Log( n(k+1))$ for  $p<\infty$ and $k=0,1,\ldots$.
\end{lem}

\begin{proof}
Each $t\in T_{\leq k}$ has support of cardinality at most $2^{pk}$. Every nonzero coordinate
of such $t$ may be chosen in at most $ 2k+1$ ways, so
\[
|T_{\leq k}|\leq \sum_{0\le l \le 2^{pk}}\binom{n}{l}( 2k+1)^l
\leq\sum_{0\le l \le 2^{pk}}(2n(k+1))^l\leq (2n(k+1)))^{2^{pk}+1}.
\qedhere
\]
\end{proof}

\begin{proof}[Proof of Proposition~\ref{prop:bdd-coordinates-q}]
It suffices to prove \eqref{eq:toshow2}, since \eqref{eq:toshow2-dual} follows by duality. 

Choose $k_0=k_0(n)\coloneqq \lceil \log_2 n\rceil+2 \sim \Log n$. Then by \eqref{eq:redtointnorms}, \eqref{eq:redtoT},  and Lemmas~ \ref{lem:supoverTk} and \ref{lem:cardTk} we get
\begin{align*}
\Ex\sup_{s\in B_{q^*}^m\cap aB_\infty^n}\sup_{t\in B_2^n}
\sum_{i\leq m,j\leq n}a_{ij}g_{ij}s_it_j
&\leq \sqrt{2}\Ex\sup_{t\in B_2^n}
\Bigl\|\Bigl(\sum_{j\leq n}a_{ij}g_{ij}t_j\Bigr)_{i\leq m}\Bigr\|_{2,q,a^{(2-q^*)/2}}
\\
&\lesssim
\Ex\sup_{t\in T_{\leq k_0}}
\Bigl\|\Bigl(\sum_{j\leq n}a_{ij}g_{ij}t_j\Bigr)_{i\leq m}\Bigr\|_{2,q,a^{(2-q^*)/2}
}
\\
&\lesssim
\sqrt{q}D_2+a^{(2-q^*)/2}D_1\sum_{k=0}^{ k_0} 2^{-pk/2}\sqrt{\Log|T_{\leq k}|}
\\
&\lesssim 
\sqrt{q}D_2+a^{(2-q^*)/2}D_1(k_0+1)\sqrt{\Log n}
\\
&\lesssim \sqrt{q}D_2+a^{(2-q^*)/2}D_1\Log^{3/2} n.
\qedhere
\end{align*}
\end{proof}

Now we are ready to prove Proposition~\ref{prop:worseD3'-intro} in the case  
$(p^*,q)\notin [2,3]^2$. We postpone the proof in the case $(p^*,q)\in [2,3]^2$ to
Section~\ref{sect:appendix}.

\begin{proof}[{{Proof of Proposition~\ref{prop:worseD3'-intro} in the case $(p^*,q)\notin [2,3]^2$.}}]
By duality it suffices to consider the case $p^*\le q$. Then $q> 3$.

Let $a\coloneqq\Log^{-3/(2-q^*)}(mn)$.   Then Proposition~\ref{prop:bdd-coordinates-q} implies
\begin{align*}
\Ex \sup_{s\in B_{q^*}^m\cap aB_{\infty}^m}\ \sup_{t\in B_{p}^n}  
\sum_{i,j}a_{ij}g_{ij}s_it_j 
&\lesssim \sqrt{q}D_2+ a^{(2-q^*)/2}\Log^{3/2}(mn) D_1 
\\ &= \sqrt{q}D_2+D_1.
\end{align*}

Moreover, Proposition~\ref{prop:largecoordinates} (applied with $\gamma =\Log^{3/(2-q^*)}(mn)$) 
yields
\begin{align*}
\Ex  \sup_{s\in B_{q^*}^m}\ & \sup_{t\in B_{p}^n}  
\sum_{i,j}a_{ij}g_{ij}s_it_j\ind_{\{|s_i|\ge \Log^{-3/(2-q^*)}(mn)\}}
\\
& \lesssim \sqrt{p^*}D_1  +\Log^{1/2+3q^*/(2(2-q^*))} (mn)\max_{i,j}|a_{ij}| .
\end{align*}
The two displayed inequalities yield the assertion (with $\gamma= 5$, 
since $q^*\leq 3/2$).
\end{proof}

\subsection{Exponent reduction} 
\label{subsect:logexp-reduction}
Proposition~\ref{prop:worseD3'-intro} shows that
\begin{equation*}	
\Ex\|G_A\|_{p\to q}
\lesssim \sqrt{p^*}D_1 +\sqrt{q}D_2+ \Log^{\gamma}(mn) \max_{i,j}|a_{ij}|.
\end{equation*}
In this and the next subsection we show how to reduce the 
exponent $\gamma$ and obtain Proposition~\ref{prop:dimdependent}.
The argument is based on the analysis of the graph associated to the matrix $A$. 
Such an analysis was performed in the Bernoulli case in \cite{Lbern}.
To run the exponent reduction procedure we also need to 
obtain some weaker estimates with constants depending on the degree of this graph
(we do this in Subsection~\ref{sect:degdepbounds} in the case when $p^*\vee q>2$, 
and in Section~\ref{sect:appendix} when $p^*,q\in [2,4)$).
Since we work in the Gaussian setting and not with bounded Bernoulli entries as in \cite{Lbern}, 
we face some new difficulties.
It is possible to deal with them making the advantage of the fact that  $p^*\wedge q >2$,  
whereas in the case $p^*,q\in [2,4)$ one may use more involved combinatorial arguments and 
exploit the bounds following from the Slepian-Fernique lemma. The exponent reduction procedure 
in the latter case is run in Section~\ref{sect:appendix}.

With an $m\times n$ matrix $A=(a_{ij})_{i\le m, j\le n}$ we associate the set 
\[
E_A\coloneqq\{(i,j)\colon a_{ij}\neq 0\}\subset [m]\times [n].
\]
We set
\begin{align*}
d_{1,A}&\coloneqq\max_{i\le m}|\{j\le n \colon (i,j)\in E_A\}|,
\\ d_{2,A}&\coloneqq\max_{j\le n}|\{i\le m \colon (i,j)\in E_A\}|,
\\ d_A&\coloneqq d_{1,A}\vee d_{2,A}.
\end{align*}
We do not assume that $A$ is symmetric, but we may treat $([m],[n],E_A)$ as a bipartite  graph. 
Then $d_A$ is its degree. 
We write $i\sim_A j$ if $(i,j)\in E_A$, and $I\sim_A j$ (for $I\subset[m]$)  
if there exists $i\in I$ such that $i\sim_A j$.

By $\rho=\rho_A$ we denote the distance on $ [m]\sqcup[n]$ induced by $E_A$. 
A subset $I\subset [m]$ or $J\subset [n]$ is called $r$-connected if
it is a connected subset of the graph 
 $G_A(r)\coloneqq( [m]\sqcup [n],E_A(r))$, where $( u,v)\in E_A(r)$ if and only if $\rho_A( u,v)\leq r$  for $u,v\in[m]\sqcup [n]$.

We denote by $\mathcal{I}_r(k)=\mathcal{I}_r(k,A)$ 
(respectively, $\mathcal{J}_r(k)=\mathcal{J}_r(k,A)$) the family of all $r$-connected 
subsets of $[m]$ (respectively, $[n]$) of cardinality $k$. 
Note that the maximal degree of  $G_A(r)$ is at most 
$d_A+d_A(d_A-1)+\cdots+d_A(d_A-1)^{r-1}\le d_A^r$.
Thus, \cite[Lemma~3.2]{Lbern} implies
\begin{equation} 
\label{eq:size-Ir(k)}
|\mathcal{I}_r(k)|\leq m4^kd_A^{rk} \quad \text{and} \quad |\mathcal{J}_r(k)|\leq n4^kd_A^{rk}.
\end{equation}

For $I\subset [m]$  we define 
\[
I'=\{j\in [n]\colon\ \exists_{i\in I}\, (i,j)\in E_A \}.
\]
In a similar way we define  $J'\subset [m]$  for $J\subset [n]$.

The next proposition reveals how one may reduce the exponent at the logarithmic term to  deduce
the desired bound~\eqref{eq:pqsqrtlogmn} from a weaker 
estimate depending on $d_{1,A}$ and $d_{2,A}$.

\begin{prop}
\label{prop:fromuptopowerdA}
Let $p^*,q\in [2,\infty)$, $ 0\le \gamma_1< \frac{1}{p^*}$, $0\le \gamma_2<\frac{1}{q}$, 
$\gamma_3>0$ and 
$\alpha_1,\alpha_2,\alpha_3\geq 1$  be such that for every $m\times n$ matrix $A$, 
\begin{align}	
\label{eq:uptopowersdA}
\Ex\|G_A\|_{p\to q}
\leq \alpha_1D_1+\alpha_2D_2
+\alpha_3\bigl(d_{1,A}^{\gamma_1}+d_{2,A}^{\gamma_2}+\sqrt{\Log(mn)}\bigr)\max_{i,j}|a_{ij}|
\end{align}
and
\begin{equation}	
\label{eq:assump-worseexplog}
\Ex\|G_A\|_{p\to q}
\le \alpha_1D_1+\alpha_2D_2 + \alpha_3 \Log^{ \gamma_3}(mn) \max_{i,j}|a_{ij}|.
\end{equation}
Then for every $m\times n$ matrix $A$, 
\begin{multline*}
\Ex\|G_A\|_{p\to q} \lesssim \frac{\Log\gamma_3}{-\ln((\gamma_1p^*)\vee (\gamma_2 q)))} 
\\
\cdot \Bigl(  (\alpha_1+\alpha_3)  D_1+ (\alpha_2+\alpha_3)D_2
+\alpha_3 \sqrt{\Log(mn)}\max_{i,j}|a_{ij}| \Bigr).
\end{multline*}
\end{prop}

\begin{proof}
Let $\delta=(( p^*\gamma_1)\vee (q\gamma_2 ))^{-1}-1$. 
Then $p^*\gamma_1 ,q\gamma_2 \leq (1+\delta)^{-1}$.
Let
\[
k_0\coloneqq \inf\{k\geq 1\colon (1+\delta)^{k+1}\geq  2\gamma_3\}
\]
and define
\[
u_k\coloneqq (\Log(mn))^{-\frac{1}{2}(1+\delta)^{k+1}},\quad k=0,1,\ldots,k_0.
\]
Let 
\[
M\coloneqq D_1+D_2,\quad \widetilde{M}\coloneqq\alpha_1D_1+\alpha_2D_2
+\alpha_3\sqrt{\Log(mn)} \max_{i,j}|a_{ij}|.
\]
Define matrices $A_k=(a_{ij}(k))_{i\leq m,j\leq n}$, $k=0,1,\ldots,k_0+1$ by
\begin{gather*}
a_{ij}(0)=a_{ij}\ind_{\{|a_{ij}|\geq u_0M\}}, \qquad
a_{ij}(k_0+1)=a_{ij}\ind_{\{|a_{ij}|< u_{k_0} M\}},
\\ 
\text{and}\quad 
a_{ij}(k)=a_{ij}\ind_{\{u_{k}M\leq |a_{ij}|< u_{k-1}M\}} \quad \text{for }k=1,\ldots,k_0.
\end{gather*}

Then $A=\sum_{k=0}^{k_0+1}A_k$, so 
\[
\|G_A\|_{p\to q}\leq \sum_{k=0}^{k_0+1}\|G_{A_k}\|_{p\to q} .
\]

Observe that for any $u\geq 0$,
\begin{align*}
uM\max_{i}|\{j\colon |a_{ij}|\geq uM\}|^{1/p^*}&\leq \max_i \|(a_{ij})_j\|_{p^*}\leq M,
\\
uM\max_{j}|\{i\colon |a_{ij}|\geq uM\}|^{1/q}&\leq \max_j \|(a_{ij})_i\|_{q}\leq M.
\end{align*}
Thus, 
\[
d_{1,A_k}\leq u_k^{-p^*},\ d_{2,A_k}\leq u_k^{-q} \quad \text{for } k=0,1,\ldots,k_0.
\]

Since
\[
d_{1,A_0}^{\gamma_1}+d_{2,A_0}^{\gamma_2}\leq u_0^{-p^*\gamma_1}+u_0^{-q\gamma_2}
\leq 2u_0^{-1/(1+\delta)}=2\sqrt{\Log(mn)} ,
\]
assumption \eqref{eq:uptopowersdA} (applied to the matrix $A_0$) yields
\[
\|G_{A_0}\|_{p\to q}\lesssim  \widetilde{M}.
\]

Moreover, assumption  \eqref{eq:uptopowersdA}, applied to the matrix $A_k$,  
yields for $1\leq k\leq k_0$, 
\begin{align*}
\|G_{A_k}\|_{p\to q}
&\leq  \widetilde{M}+\alpha_3 u_{k-1}M(d_{1,A_k}^{\gamma_1}+d_{2,A_k}^{\gamma_2})
\leq  \widetilde{M}+\alpha_3Mu_{k-1}(u_k^{-p^*\gamma_1}+u_k^{-q\gamma_2})
\\
&\leq  \widetilde{M}+2\alpha_3Mu_{k-1}u_k^{-1/(1+\delta)}= \widetilde{M}+2\alpha_3M.
\end{align*}

Finally,  we apply  assumption \eqref{eq:assump-worseexplog} to get
\[
\|G_{A_{k_0+1}}\|_{p\to q} \le \alpha_1D_1+\alpha_2D_2 +\alpha_3M.
\]
We finish the proof by noting that $k_0\lesssim  \frac{\Log\gamma_3}{-\ln((\gamma_1p^*)\vee (\gamma_2 q))}$.
\end{proof} 

\subsection{Degree dependent bounds}	
\label{sect:degdepbounds}

In order to provide a weaker  degree dependent estimate \eqref{eq:uptopowersdA} we are going to 
use the following proposition.
For $\gamma>1$ we define the set of $\gamma$-flat vectors from $B_{u}^r$ with support $I$ by
\[
K(u,r,I,\gamma)
=\Bigl\{s\in B_{u}^r\colon \mathrm{supp}(s)=I,\ 
\max_{i\in I}|s_i|\leq \gamma\min_{i\in I}|s_i|\Bigr\}.
\]
Note that if $s\in K(u,r,I,\gamma)$, then $\max_{i\in I}|s_i|\le \gamma |I|^{-1/r}$.
For an $m\times n$ matrix $A$ we also put
\[
Y_{k,l}(\gamma)=Y_{k,l}(\gamma,A,p,q)=\max_{I\in \mathcal{I}_4(k)}\max_{J\in \mathcal{J}_4(l)}\
\sup_{s\in K(q^*,m,I,\gamma)}\ \sup_{t\in K(p,n,J,\gamma)}
\sum_{i\in I,j\in J}a_{ij}g_{ij}s_it_j.
\]

\begin{prop}
\label{prop:redtosupflat}
For every $\ve\in (0,1]$ and $1\leq p\leq 2\leq q\leq \infty$,
\begin{align}
\label{eq:redtosupflat1}
\Ex \|G_A\|_{p\to q}
&\lesssim \frac{1}{\ve}\Ex\max_{1\leq k\leq m}\max_{1\leq l\leq n}Y_{k,l}(d_A^\ve)+
\Ex\max_{i,j}|a_{ij}g_{ij}|
\\
\label{eq:redtosupflat2}
&\lesssim \frac{1}{\ve}\Bigl(\max_{1\leq k\leq m}\max_{1\leq l\leq n}\Ex Y_{k,l}(d_A^\ve)
+\sqrt{\Log (mn)}\,\|(a_{ij})\|_\infty\Bigr).
\end{align}
\end{prop}

\begin{proof}
Fix $\ve\in (0,1]$.
For  $s\in B_{q^*}^m$, $t\in B_p^n$ and $k,l=1,2,\ldots$ define
\begin{align*}
I_k(s)&\coloneqq\{i\in [m]\colon\  d_A^{-k\ve}<|s_i|\leq d_A^{(1-k)\ve}\},\\
J_l(t)&\coloneqq\{j\in [n]\colon\  d_A^{-l\ve}<|t_j|\leq d_A^{(1-l)\ve}\}.
\end{align*}
Then
\begin{equation*}
\sum_{k\geq 1}d_A^{-q^*k\ve}|I_k(s)|\leq \|s\|_{q^*}^{q^*},\quad
\sum_{l\geq 1}d_A^{-pl\ve}|J_l(t)|\leq \|t\|_p^p
\end{equation*}
and
\[
\Ex\|G_A\|_{p\to q}=
\Ex\sup_{\|s\|_{q*}\leq 1}\sup_{\|t\|_p\leq 1}
\sum_{k,l\geq 1}\sum_{i\in I_k(s)}\sum_{j\in J_l(t)}a_{ij}g_{ij}s_it_j.
\]

Define
\[
M_A=\max_{i,j}|a_{ij}g_{ij}|.
\]
Observe that for any $s\in B_{q^*}^m$ and $t\in B_p^n$,
\begin{align*}
\Biggl|\sum_{k\geq 1}\sum_{l\geq k+1/\ve+1}\sum_{i\in I_k(s)}\sum_{j\in J_l(t)} &
a_{ij}g_{ij}s_it_j\Biggr|
\leq \sum_{k\geq 1}\sum_{i\in I_k(s)}|s_i|\sum_{l\geq k+1/\ve+1}\sum_{j\in J_l(t)}
|a_{ij}g_{ij}||t_j|
\\
&\leq M_A\sum_{k\geq 1}\sum_{i\in I_k(s)}|s_i|d_A^{-k\ve-1}\sum_{j}\ind_{E_A}(i,j)
\\
&\leq M_A\sum_{k\geq 1}\sum_{i\in I_k(s)}s_i^2 =  M_A\|s\|_2^2\leq M_A\|s\|_{q^*}^2\leq M_A.
\end{align*}
Similarily,
\[
\Biggl|\sum_{l\geq 1}\sum_{k\geq l+1/\ve+1}\sum_{i\in I_k(s)}\sum_{j\in J_l(t)}
a_{ij}g_{ij}s_it_j\Biggr|\leq M_A\|t\|_2^2\leq M_A.
\]
Therefore, it is enough  to estimate
\begin{align*}
\sum_{|r|\leq 1/\ve +1}&\Ex \sup_{\|s\|_{q^*}\leq 1}\sup_{\|t\|_p\leq 1}
\sum_{\substack{k,l\geq 1\\l-k=r}}\sum_{i\in I_k(s)}\sum_{j\in J_l(t)}a_{ij}g_{ij}s_it_j
\\
&\leq
(2/\ve+3)\max_{|r|\leq 1/\ve +1} \Ex \sup_{\|s\|_{q^*}\leq 1}\sup_{\|t\|_p\leq 1}
\sum_{\substack{k,l\geq 1\\l-k=r}}\sum_{i\in I_k(s)}\sum_{j\in J_l(t)}a_{ij}g_{ij}s_it_j.
\end{align*}

Let us fix $|r|\leq 1/\ve+1$, $s\in B_{q^*}^m$, and $t\in B_p^n$. For  $k,l\geq 1$ with $l-k=r$ let 
$I_{k,1},\ldots,I_{k,u_k}$ be $2$-connected components of $I_k(s)\cap J_l(t)'$ and 
$J_{k,u}\coloneqq I_{k,u}' \cap J_{l}(t)$, $1\leq u\leq u_k$. 
Then the sets $(I_{k,u})_{k\geq 1,1\leq u\leq u_k}$ are nonempty pairwise disjoint 
$2$-connected subsets of $[m]$ and the sets $(J_{k,u})_{k\geq 1,1\leq u\leq u_k}$ are 
nonempty pairwise disjoint $4$-connected subsets of $[n]$.
Define  vectors $s_{k,u},\bar{s}_{k,u}\in \er^m$ and $t_{k,u},\bar{t}_{k,u}\in \er^n$ by
\begin{gather*}
s_{k,u}\coloneqq(s_i\ind_{\{i\in I_{k,u}\}})_{i\leq m},\qquad \bar{s}_{k,u}\coloneqq\frac{s_{k,u}}{\|s_{k,u}\|_{q^*}},\\
t_{k,u}\coloneqq(t_j\ind_{\{j\in J_{k,u}\}})_{j\leq n},\qquad \bar{t}_{k,u}\coloneqq\frac{t_{k,u}}{\|t_{k,u}\|_{p}}.
\end{gather*}
Since $I_{k,u}\subset I_k(s)$ and $J_{k,u}\subset J_l(t)$, $s_i$ and $t_j$ do not vanish if 
$i\in I_{k,u}$ and $j\in J_{k,u}$. Thus,   $\bar{s}_{k,u}\in K(q^*,m,I_{k,u},d_A^{\ve})$ and
$\bar{t}_{k,u}\in K(p,n,J_{k,u},d_A^{\ve})$, so 
\begin{align*}
\sum_{\substack{k,l\geq 1\\l-k=r}}\sum_{i\in I_k(s)}\sum_{j\in J_l(t)}a_{ij}g_{ij}s_it_j
&=\sum_{k}\sum_{1\leq u\leq u_k}\sum_{i\in I_{k,u}}\sum_{j\in J_{k,u}}a_{ij}g_{ij}s_it_j
\\
& \leq 
\sum_{k}\sum_{1\leq u\leq u_k} Y_{|I_{k,u}|,|J_{k,u}|}(d_A^\ve)\|s_{k,u}\|_{q^*}\|t_{k,u}\|_p
\\
&\leq 
\max_{1\leq  \widetilde{k}\leq m,1\leq  \widetilde{l}\leq n}Y_{\widetilde{k},\widetilde{l}}
(d_A^\ve)\sum_{k}\sum_{1\leq u\leq u_k}\|s_{k,u}\|_{q^*}\|t_{k,u}\|_p.
\end{align*}
Moreover,
\begin{align*}
\sum_{k}\sum_{1\leq u\leq u_k}\|s_{k,u}\|_{q^*}\|t_{k,u}\|_p
&\leq \Bigl(\sum_{k}\sum_{1\leq u\leq u_k}\|s_{k,u}\|_{q^*}^2\Bigr)^{1/2}
\Bigl(\sum_{k}\sum_{1\leq u\leq u_k}\|t_{k,u}\|_{p}^2\Bigr)^{1/2}
\end{align*}
\begin{align*}
&\leq \Bigl(\sum_{k}\sum_{1\leq u\leq u_k}\|s_{k,u}\|_{q^*}^{q^*}\Bigr)^{1/2}
\Bigl(\sum_{k}\sum_{1\leq u\leq u_k}\|t_{k,u}\|_{p}^p\Bigr)^{1/2}
\\
&=\|s\|_{q^*}^{q^*/2}\|t\|_{p}^{p/2}\leq 1.
\end{align*}

The above calculations show that
\[
 \|G_A\|_{p\to q}\lesssim \frac{1}{\ve}\max_{1\leq k\leq m,1\leq l\leq n}Y_{k,l}(d_A^\ve)
+M_A,
\]
 which yield  bound \eqref{eq:redtosupflat1}.

Moreover,  $\Ex M_A\lesssim \Log^{1/2}(mn)\|(a_{ij})\|_\infty$ and, 
by estimate \eqref{eq:supgauss} and inclusions $B_{q^*}^m\subset B_2^m$ and $B_p^n\subset B_2^n$,
\begin{align*}
\Ex \max_{1\leq k\leq m,1\leq l\leq n}Y_{k,l}(d_A^\ve)
&\lesssim \max_{1\leq k\leq m,1\leq l\leq n}\Ex Y_{k,l}(d_A^\ve)
\\ 
&\qquad+\sqrt{\Log (mn)}\sup_{s\in B_{q^*}^m}\sup_{t\in B_p^n}
\Bigl(\sum_{i,j}a_{ij}^2s_i^2t_j^2\Bigr)^{1/2}
\\
&=\max_{1\leq k\leq m,1\leq l\leq n}\Ex Y_{k,l}(d_A^\ve)+\sqrt{\Log (mn)}\, \|(a_{ij})\|_\infty,
\end{align*}
so estimate \eqref{eq:redtosupflat2} follows.
\end{proof}

Now we need a bound for the expectation of $Y_{k,l}(\gamma)$.
It is derived in the following proposition.

\begin{prop}
\label{prop:est2Ykl}
For every $p^*,q\in [2,\infty)$ and $\gamma \geq 1$,
\begin{align}
\notag
\Ex Y_{k,l}(\gamma)
&\lesssim 
\sqrt{p^*}D_1+\sqrt{q}D_2+\sqrt{\Log (mn)}\max_{i,j}|a_{ij}|
\\
\label{eq:est2Yklall}
&\qquad +\gamma^{3/2} k^{-1/q^*}l^{-1/p}\min\{kl^{1/2},k^{1/2}l,kd_{1,A}^{1/2},ld_{2,A}^{1/2}\}
\sqrt{\Log(d_A)}\max_{i,j}|a_{ij}|.
\end{align}
\end{prop}

\begin{proof}
Observe that
\[
Y_{k,l}(\gamma)\leq \max_{I\in \mathcal{I}_4(k)}X_I,
\]
where
\begin{align*}
X_I&=X_I(p,q,k,l,\gamma)
\\
&\coloneqq\sup\Bigl\{\sum_{i\in I,j}a_{ij}g_{ij}s_it_j\colon 
s\in B_{q*}^I\cap \gamma k^{-1/q^*}B_\infty^I,\ t\in B_p^n\cap \gamma l^{-1/p}B_\infty^n\Bigr\}.
\end{align*}

Define 
\begin{align*}
\sigma_{k,l}&=\sigma_{k,l}(p,q,\gamma)
\\
&\coloneqq\sup \Biggl\{\sqrt{\sum_{i,j}a_{ij}^2s_i^2t_j^2} 
\colon s\in B_{q*}^{m}\cap \gamma k^{-1/q^*}B_\infty^{m},
\ t\in B_p^n\cap \gamma l^{-1/p}B_\infty^n\Biggr\}.
\end{align*}
Let us fix $I\in \mathcal{I}_4(k)$ and choose  a $1/2$-net $ S$ in 
$B_{q^*}^I\cap \gamma k^{-1/q^*}B_\infty^I$, with respect to the norm determined by this set, of cardinality $5^k$. Then
\[
\Ex X_I
\leq 2\Ex\sup_{s\in  S}\sup_{t\in B_p^n\cap \gamma l^{-1/p}B_\infty^m}
\sum_{i\in I,j}a_{ij}g_{ij}s_it_j.
\]
By inequality \eqref{eq:supgauss}  we have
\begin{align*}
\Ex X_I
&\lesssim \sup_{s\in  S}
\Ex \sup_{t\in B_p^n\cap \gamma l^{-1/p}B_\infty^m}\sum_{i\in I,j}a_{ij}g_{ij}s_it_j
+\sqrt{\Log |S|}\, \sigma_{k,l}.
\end{align*}
By \eqref{eq:fixeds}
\[
\sup_{s\in  S}\Ex \sup_{t\in B_p^n\cap \gamma l^{-1/p}B_\infty^m}\sum_{i\in I,j}a_{ij}g_{ij}s_it_j
\leq  \sqrt{p^*}D_1 .
\]
Thus,
\[
\Ex X_I\lesssim \sqrt{p^*}D_1+\sqrt{k}\sigma_{k,l}.
\]
Applying estimate \eqref{eq:supgauss} again and using \eqref{eq:size-Ir(k)} we get 
\begin{align}
\notag
\Ex Y_{k,l}
&\lesssim \max_{I\in \mathcal{I}_4(k)}\Ex X_I
+\sqrt{\Log| \mathcal{I}_4(k)|}\, \sigma_{k,l}
\\
\label{eq:estYkla}
&\lesssim  \sqrt{p^*}D_1+(\sqrt{k\Log(d_A)}+\sqrt{\Log m})\sigma_{k,l}.
\end{align}

To estimate $\sigma_{k,l}$ observe first that $B_{q^*}^I\subset B_2^I$ and 
$B_p^n\subset B_2^n$, thus
\[
\sup_{t\in  B_{q*}^I}\sup_{s\in B_p^n} \sum_{i\in I,j}a_{ij}^2s_i^2t_j^2
\leq \max_{i,j}|a_{ij}|^2.
\]
Moreover, for any  $s\in B_{q*}^{m}\cap \gamma k^{-1/q^*}B_\infty^{m}$ and 
$t\in B_p^n\cap \gamma l^{-1/p}B_\infty^n$ we have
\begin{align*}
\sum_{i,j}a_{ij}^2s_i^2t_j^2
&\leq \|s\|_\infty^{2-q^*}\|t\|_\infty^{2-p}\max_{i,j}|a_{ij}|^2\sum_{i,j}|s_i|^{q^*}|t_j|^p
\\
&\leq \gamma^{4-q^*-p}k^{1-2/q^*}l^{1-2/p}\max_{i,j}|a_{ij}|^2
\\
&\leq \gamma^{2}k^{1-2/q^*}l^{1-2/p}\max_{i,j}|a_{ij}|^2
\end{align*}
and
\begin{align*}
\sum_{i,j}a_{ij}^2s_i^2t_j^2
&\leq \|s\|_\infty^{2-q^*}\|t\|_\infty^{2}\sum_{i,j}a_{ij}^2|s_i|^{q*}
\leq  \|s\|_\infty^{2-q^*}\|t\|_\infty^{2}\max_{i}\sum_{j}a_{ij}^2
\\
&\leq  \|s\|_\infty^{2-q^*}\|t\|_\infty^{2}d_{1,A}\max_{i,j}|a_{ij}|^2
\leq \gamma^{4-q^*}k^{1-2/q^*}l^{-2/p}d_{1,A}\max_{i,j}|a_{ij}|^2
\\
&\leq 
\gamma^{3}k^{1-2/q^*}l^{-2/p}d_{1,A}\max_{i,j}|a_{ij}|^2.
\end{align*}
Therefore
\begin{equation}
\label{eq:estsigmakl}
\sigma_{k,l}
\leq \min\{1,\gamma k^{1/2-1/q^*}l^{1/2-1/p},\gamma^{3/2}k^{1/2-1/q^*}l^{-1/p}d_{1,A}^{1/2}\}
\max_{i,j}|a_{ij}|.
\end{equation}

Estimates \eqref{eq:estYkla} and \eqref{eq:estsigmakl} yield
\begin{align*}
\Ex Y_{k,l}(\gamma)
&\lesssim \sqrt{p^*}D_1
+\sqrt{\Log m}\max_{i,j}|a_{ij}|
\\
&\qquad +\gamma^{3/2} k^{1-1/q^*}l^{-1/p}\min\{l^{1/2},d_{1,A}^{1/2}\}
\sqrt{\Log(d_A)}\max_{i,j}|a_{ij}|.
\end{align*}

In a similar way we show that
\begin{align*}
\Ex Y_{k,l}(\gamma)
&\lesssim \sqrt{q}D_2
+\sqrt{\Log n}\max_{i,j}|a_{ij}|
\\
&\qquad +\gamma^{3/2} k^{-1/q^*}l^{1-1/p}\min\{k^{1/2},d_{2,A}^{1/2}\}
\sqrt{\Log(d_A)}\max_{i,j}|a_{ij}|.
\qedhere
\end{align*}
\end{proof}

To make use of the two previous propositions we consider the cases  
$\frac{1}{p}+\frac{1}{q^*}\geq 3/2$ and  $\frac{1}{p}+\frac{1}{q^*}\leq 3/2$ separately.

\begin{cor}
\label{cor:uptodAepsiloncase1}
Suppose that $p^*,q\in [2,\infty)$ are such that $\frac{1}{p}+\frac{1}{q^*}\geq 3/2$. Then for every $\ve\in (0,1]$,
\begin{align*}
\Ex\|G_A\|_{p\to q}
\lesssim \ve^{-1}\Bigl( \sqrt{p^*}D_1+\sqrt{q}D_2
+\bigl(\ve^{-1/2}d_A^\ve+\sqrt{\Log(mn)}\bigr)\max_{i,j}|a_{ij}|\Bigr).
\end{align*}
\end{cor}

\begin{proof}
By Proposition \ref{prop:redtosupflat} we have
\[
\Ex \|G_A\|_{p\to q}
\lesssim \frac{1}{\ve}\Bigl(\max_{1\leq k\leq m}\max_{1\leq l\leq n}\Ex Y_{k,l}(d_A^{\ve/3})
+\sqrt{\Log (mn)}\, \|(a_{ij})\|_\infty\Bigr).
\]

Observe that in the case $p^*,q\in [2,\infty)$ and $\frac{1}{p}+\frac{1}{q^*}\geq 3/2$ we have for any $k,l$,
\begin{align*}
 k^{-1/q^*}l^{-1/p}\min\{kl^{1/2},k^{1/2}l,kd_{1,A}^{1/2},ld_{2,A}^{1/2}\}
 &\leq \min\{k^{1-1/q^*}l^{1/2-1/p},k^{1/2-1/q^*}l^{1-1/p}\}
 \\
 &\leq (k\wedge l)^{3/2-1/q^*-1/p}\leq 1.
\end{align*}
Hence, estimate \eqref{eq:est2Yklall} yields that 
\begin{align*}
\Ex Y_{k,l}(d_A^{\ve/3})
\lesssim \sqrt{p^*}D_1+\sqrt{q}D_2
+\bigl(d_A^{\ve/2}\sqrt{\Log d_A}+\sqrt{\Log(mn)}\bigr)\max_{i,j}|a_{ij}|.
\end{align*}
Finally we observe that for $\ve\in(0,1)$,
\begin{equation}
\label{eq:sup-eps}
\sup_{x\geq 1}x^{-\ve}  \Log x=\frac 1{1\wedge(e\ve)} \le \frac 1\ve,
\end{equation} 
so $d_A^{\ve/2}\sqrt{\Log d_A}\le d_A^{\ve}\ve^{-1/2}$.
\end{proof}

\begin{cor}
\label{cor:uptodAepsiloncase4}
If $p^*,q\in [2,\infty)$ are such that $\frac{1}{p}+\frac{1}{q^*}\le\frac32$ and 
$q\vee p^*>2$, then 
\begin{align*}
&\Ex\|G_A\|_{p\to q}\lesssim \frac{(q\vee p^*)^3}{q\vee p^* -2}\biggl[
\sqrt{p^*}D_1+\sqrt{q}D_2
\\
&\qquad+\Bigl(\frac{(q\vee p^*)^{3/2}}{(q\vee p^* -2)^{1/2}}(d_{1,A}^{\frac3{2p^*}
-\frac q{2(p^*+q)}} + d_{2,A}^{\frac3{2q} -\frac{p^*}{2(p^*+q)} } )
+\sqrt{\Log(mn)}\Bigr)\max_{i,j}|a_{ij}|\biggr].
\end{align*}
\end{cor}

\begin{proof}
We proceed as in the  proof of the previous  corollary.  The only difference is 
a more delicate estimate of
\[
\pi_{k,l}\coloneqq k^{-1/q^*}l^{-1/p}\min\{kl^{1/2},k^{1/2}l,kd_{1,A}^{1/2},ld_{2,A}^{1/2}\}.
\]
Suppose $2\le p^*\le q$ and $q>2$ (the case when $2\le q\le p^*$ and $p^*>2$ follows by duality).
Let $\rho>0$ be a constant to be chosen later. 
We consider two cases. 

\textbf{Case 1.} $k\leq \rho l$. Assumption $\frac{1}{p}+\frac{1}{q^*}\le \frac 32$ implies that $\frac12+\frac1{p^*}-\frac1{q^*}\ge 0$, and so if $k\le \rho d_{1,A}$, then
\begin{align*}
\pi_{k,l}
&\leq k^{1/p^*-1/q^*}\min\{k^{1/2} (k/l)^{1/2-1/p^*} ,d_{1,A}^{1/2} (k/l)^{1/p}\}
\\ &\leq k^{1/p^*-1/q^*} \rho^{1/2-1/p^*} \bigl(\min\{k,\rho d_{1,A}\}\bigr)^{1/2}
 = k^{1/2+1/p^*-1/q^*} \rho^{1/2-1/p^*} 
\\&\leq (\rho d_{1,A})^{1/p^*-1/q^*+1/2}\rho^{1/2-1/p^*} = \rho^{1/q} d_{1,A}^{1/p^*+1/q-1/2}.
\end{align*}
Moreover, since $1/p^*\le 1/2\le 1/q^*$, we have $1/p^*-1/q^*\le 0$, and so if $k\ge \rho d_{1,A}$, then
\begin{align*}
\pi_{k,l}
&\leq k^{1/p^*-1/q^*}\min\{k^{1/2} (k/l)^{1/2-1/p^*} ,d_{1,A}^{1/2} (k/l)^{1/p}\}
\\ &\leq k^{1/p^*-1/q^*} \rho^{1/2-1/p^*} \bigl(\min\{k,\rho d_{1,A}\}\bigr)^{1/2}
 = k^{1/p^*-1/q^*} \rho^{1/2-1/p^*}  (\rho d_{1,A})^{1/2}
\\& \le (\rho d_{1,A})^{1/p^*-1/q^*}\rho^{1/2-1/p^*}  (\rho d_{1,A})^{1/2}
= \rho^{1/q} d_{1,A}^{1/p^*+1/q-1/2}.
\end{align*}
Therefore, in both cases $\pi_{k,l}\le \rho^{1/q} d_{1,A}^{1/p^*+1/q-1/2}$.

\textbf{Case 2.} $k\geq \rho l$. Then 
\begin{align*}
\pi_{k,l}
&\leq  k^{1/p^*-1/q^*}\bigl(\min\{k,d_{2,A}\}\bigr)^{1/2} (l/k)^{1/p^*}
\leq \rho^{-1/p^*}  k^{1/p^*-1/q^*}\bigl(\min\{k,d_{2,A}\}\bigr)^{1/2}
\\
&\leq  \rho^{-1/p^*}d_{2,A}^{1/p^*-1/q^*+1/2}=\rho^{-1/p^*}d_{2,A}^{1/p^*+1/q-1/2}.
\end{align*}

Choosing $\rho\coloneqq(d_{2,A}/d_{1,A})^{1-p^*q/(2(p^*+q))}$ we obtain in both cases
\[
\pi_{k,l}\leq d_{1,A}^{1/p^*-q/(2(p^*+q))}d_{2,A}^{1/q-p^*/(2(p^*+q))}
\leq 
\frac{1}{2}\Bigl( d_{1,A}^{2/p^*-q/(p^*+q)}+d_{2,A}^{2/q-p^*/(p^*+q)}\Bigr).
\]
Therefore, Propositions \ref{prop:redtosupflat} and \ref{prop:est2Ykl} (applied to $\ve
\coloneqq\ve/3$ and $\gamma\coloneqq d_A^{\ve/3}$) yield  for every $\ve \in (0,1)$,
\begin{align*}
\Ex\|G_A\|_{p\to q}\lesssim \frac{1}{\ve}\Bigl(
& \sqrt{p^*}D_1+\sqrt{q}D_2
\\
&+\Bigl(d_{A}^{\frac{\ve }2}\sqrt{\Log d_A}\Bigl(  d_{1,A}^{\frac2{p^*}-\frac q{p^*+q}}+d_{2,A}^{\frac 2q-\frac{p^*}{p^*+q}}\Bigr)+\sqrt{\Log(mn)}\Bigr)\max_{i,j}|a_{ij}|\Bigr).
\end{align*}

Recall that estimate \eqref{eq:sup-eps} implies
\[
d_{A}^{\ve /2}\sqrt{ \Log d_A}\le \ve^{-1/2}d_A^{\ve},
\]
and $d_A=d_{1,A}\vee d_{2,A}$.
Note that $\frac2{p^*}-\frac q{p^*+q} < \frac1{p^*}$ and $\frac2q-\frac{p^*}{p^*+q} < \frac1q$, 
so we may take
\[
\ve_0\coloneqq \frac12 \min\Bigl\{\frac{q}{p^*+q}-\frac{1}{p^*} \, , \, \frac{p^*}{p^*+q}-\frac{1}q
\Bigr\} \quad \text{and}\quad \ve\coloneqq \frac{p^*}q \ve_0.
\]
For such a choice of $\ve$ we have
$\frac2{p^*}-\frac q{p^*+q} +\ve \le \frac2{p^*}-\frac q{p^*+q} +\ve_0 
\le  \frac{3}{2p^*} -\frac q{2(p^*+q)} < \frac1{p^*}$ and
$\frac2q-\frac{p^*}{p^*+q} +\ve \le \frac2q-\frac{p^*}{p^*+q} +\ve_0 
\le  \frac{3}{2q} -\frac{ p^*}{2(p^*+q)} <\frac1q$. 
Moreover, the assumption $1/p+1/q^*\leq 3/2$ yields that $1/p^*+1/q\geq 1/2$ and 
thus $2/p^*-q/(p^*+q)\geq 0$ and $2/q-p^*/(p^*+q)\geq 0$. 
Hence, the AM-GM inequality implies  
\begin{align*}
d_{1,A}^{\frac2{p^*}-\frac{q}{p^*+q}}d_{2,A}^{\ve} 
& \le \frac{\frac2{p^*}-\frac{q}{p^*+q} }{\frac2{p^*}-\frac{q}{p^*+q} +\ve_0} 
d_{1,A}^{\frac2{p^*}-\frac{q}{p^*+q} +\ve_0} 
+\frac{\ve_0 }{\frac2{p^*}-\frac{q}{p^*+q} +\ve_0}
d_{2,A}^{ \frac{\ve}{\ve_0} (\frac2{p^*}-\frac{q}{p^*+q}+\ve_0 )}
\\
&\le 
d_{1,A}^{\frac2{p^*}-\frac{q}{p^*+q} +\ve_0}+d_{2,A}^{\frac2q-\frac{p^*}{p^*+q} +\ve_0}
\end{align*}
and
\begin{align*}
d_{1,A}^{\ve}d_{2,A}^{\frac2q-\frac{p^*}{p^*+q}} 
& \le \frac{\ve_0 }{\frac2q-\frac{p^*}{p^*+q} +\ve_0} 
d_{1,A}^{\frac{\ve}{\ve_0} (\frac2q-\frac{p^*}{p^*+q}+\ve_0)} 
+\frac{\frac2q-\frac{p^*}{p^*+q} }{\frac2q-\frac{p^*}{p^*+q} +\ve_0}
d_{2,A}^{\frac2q-\frac{p^*}{p^*+q} +\ve_0}
\\ 
&\le 
d_{1,A}^{\frac{\ve}{\ve_0} (\frac2q-\frac{p^*}{p^*+q}) +\ve}+
d_{2,A}^{\frac2q-\frac{p^*}{p^*+q} +\ve_0} 
\le 
d_{1,A}^{\frac{\ve_0}{\ve} (\frac2q-\frac{p^*}{p^*+q}) +\ve}+
d_{2,A}^{\frac2q-\frac{p^*}{p^*+q} +\ve_0} 
\\
& \le 
d_{1,A}^{\frac2{p^*}-\frac{q}{p^*+q}+\ve_0}+d_{2,A}^{\frac2q-\frac{p^*}{p^*+q} +\ve_0}.
\end{align*}

To finish the proof it suffices to note that 
$\ve^{-1}\lesssim \frac{q^3}{(p^*-1)(q-1)-1} \le \frac{q^3}{q-2}$.
\end{proof}

 Now we are ready to prove  Proposition~\ref{prop:dimdependent} in  the range 
$(p,q)\notin [2,3]^2$. The case 
$(p,q)\in [2,3]^2$ is considered in Section~\ref{sect:appendix}.

\begin{proof}[{{Proof of Proposition~\ref{prop:dimdependent} in the case 
$(p^*,q)\notin [2,3]^2$}}]
The assertion follows from
\begin{itemize}
\item Corollary~\ref{cor:uptodAepsiloncase1}, applied with $\ve=(2(p^*\vee q))^{-1}$, and
Propositions~\ref{prop:worseD3'-intro} and \ref{prop:fromuptopowerdA}
if  $1/p+1/q^*\geq 3/2$.

\item Corollary~\ref{cor:uptodAepsiloncase4} and
Propositions~\ref{prop:worseD3'-intro} and \ref{prop:fromuptopowerdA}
if $1/p+1/q^*< 3/2$.
\end{itemize}

To get the claimed bounds on $\alpha(p,q)$ we observe that in the case $1/p+1/q^* < 3/2$
we have $1>\frac{3}{2}-\frac{qp^*}{2(p^*+q)}\geq 1/2$, so that
\[
-\ln\Bigl(\frac{3}{2}-\frac{qp^*}{2(p^*+q)}\Bigr)\sim \frac{qp^*}{2(p^*+q)}-\frac{1}{2}
\gtrsim \frac{p^*\vee q-2}{p^*\vee q}. \qedhere
\]
\end{proof}

%%%%%%%%%%%%%%%%%%
%%%%%%%%%%%%%%%%%%%
\section{Proposition~\ref{prop:dimdependent} implies Proposition~\ref{prop:weakerD3-allpq}} 	
\label{sect:decomposition}

To prove Proposition~\ref{prop:weakerD3-allpq} we decompose the underlying matrix $A$ into 
block  diagonal matrices $A_k$  (with blocks of appropriately small size) 
and matrices $B_l$ whose norm  may be controlled by the following proposition providing 
a crude, but dimension-independent bound.

\begin{prop}
\label{prop:sqrtIk}
 Let $p^*,q\in [2,\infty)$. Then for every partition $J_1,J_2,\ldots,J_{k_0}$ of $[n]$,
\begin{equation}
\label{eq:sqrtJk}
\Ex\|G_A\|_{p\to q}\leq \Ex\|G_A\|_{2\to q} \lesssim \sqrt{q}D_2
+\max_{1\le k \le k_0}\max_{1\le l\le k}k^{3}|J_l|^{1/2}\max_{i\leq m,j\in J_k}|a_{ij}|.
\end{equation}
Similarily, for every partition $I_1,\ldots,I_{k_0}$ of $[m]$,
\begin{equation}
\label{eq:sqrtIk}
\Ex\|G_A\|_{p\to q}\leq \Ex\|G_A\|_{p\to 2} \lesssim \sqrt{p^*}D_1
+\max_{1\le k \le k_0}\max_{1\le l\le k}k^{3}|I_l|^{1/2}\max_{i\in I_k,j\leq n}|a_{ij}|.
\end{equation}
\end{prop}

In order to show Proposition \ref{prop:sqrtIk} we need estimate~\eqref{eq:intnormprop-q-norm} 
and the following  lemma; they allow us to perform the induction step.

\begin{lem}
\label{lem:supgaussJ}
For every $q\geq 2$,  $J\subset J'\subset [n]$, a finite set $T\subset B_2^{J'}$, and  functions 
$c= (c_1,\ldots,c_m)\colon T\to \er^m$ and $d\colon T\to [0,\infty)$ we have
\begin{align*}
\Ex &\sup_{t\in T}\Bigl(\Bigl(\sum_{i\leq m}\Bigl|c_i(t)+\sum_{j\in J}a_{ij}g_{ij}t_j\Bigr|^q\Bigl)^{2/q}
+d(t)\Bigr)^{1/2}
\\
&\leq \sup_{t\in T}\Ex\Bigl(\Bigl(\sum_{i\leq m}
\Bigl|c_i(t)+\sum_{j\in J}a_{ij}g_{ij}t_j\Bigr|^q\Bigl)^{2/q}+d(t)\Bigr)^{1/2}
+C\sqrt{\Log|T|}\max_{i\leq m,j\in J}|a_{ij}|.
\end{align*}
\end{lem}

\begin{proof}
Observe that
\[
\Bigl(\Bigl(\sum_{i\leq m}\Bigl|c_i(t)+\sum_{j\in J}a_{ij}g_{ij}t_j\Bigr|^q\Bigl)^{2/q}
+d(t)\Bigr)^{1/2}
=\sup_{s\in B_{q^*}^m,v\in B_2^2}X_{s,v,t},
\]
where
\[
X_{s,v,t}\coloneqq
v_1\sum_{i\leq m}s_i\Bigl(c_i(t)+\sum_{j\in J}a_{ij}g_{ij}t_j\Bigr)
+v_{2} \sqrt{d(t)}.
\]
We have
\begin{align*}
\sup_{t\in T}\sup_{s\in B_{q^*}^m,v\in B_2^2}\Var(X_{s,v,t})
&=\sup_{t\in T}\sup_{s\in B_{q^*}^m,v\in B_2^2}v_{1}^2\sum_{i\leq m,j\in J}a_{ij}^2s_i^2t_j^2
\\
&\leq \sup_{t\in B_2^J}\sup_{s\in B_{q^*}^m}\sum_{i\leq m,j\in J}a_{ij}^2s_i^2t_j^2
=\max_{i\leq m,j\in J}|a_{ij}|^2.
\end{align*}
Hence  the assertion follows from Lemma \ref{lem:supgauss}.
\end{proof}

\begin{proof}[Proof of Proposition \ref{prop:sqrtIk}] 
We will prove \eqref{eq:sqrtJk}; the second estimate \eqref{eq:sqrtIk} follows by duality.

 For $1\leq k\leq k_0$ let $T_{k}$ be a $2^{-2k}$-net in 
$B_{2}^{J_k}$ of cardinality at most $2^{4k|J_k|}$. 
Set 
\[
J_{\leq k}\coloneqq\bigcup_{l\leq k}J_l,\quad J_{> k}\coloneqq\bigcup_{l> k}J_l
\]
and for $t\in \er^n$,
\[
\pi_k(t)\coloneqq(t_j)_{j\in J_k}\in \er^{J_k},\quad
\pi_{\leq k}(t)\coloneqq(t_j)_{j\in J_{\leq k}}\in \er^{J_{\leq k}}.
\]
Define
\begin{align*}
T_{\leq k}&\coloneqq \{t\in B_2^{J_{\leq k}}\colon\  \pi_l(t)\in T_{l}, 1\leq l\leq k\},\quad
\\
T&\coloneqq T_{\leq k_0}=\{t\in B_2^{n}\colon\  \pi_k(t)\in T_{k}, 1\leq k\leq k_0\}.
\end{align*}
Then $T$ is a $\frac 12$-net in $B_2^n$, and hence,
\[
\Ex\|G_A\|_{p\to q} \leq 
\Ex\|G_A\|_{2\to q} \leq 
2\Ex\sup_{t\in T}\Bigl(\sum_{i}\Bigl|\sum_{j}a_{ij}g_{ij}t_j\Bigr|^q\Bigr)^{1/q}.
\]

We will show by the reverse induction on $k$ that for $0\leq k\leq k_0$,
\begin{align}
\notag
\Ex\sup_{t\in T}\Bigl(\sum_{i}\Bigl|\sum_{j}a_{ij}g_{ij}t_j\Bigr|^q\Bigr)^{1/q}
&\leq \Ex\sup_{t\in T}
\Bigl(\Bigl(\sum_{i}\Bigr|\sum_{j\in J_{\leq k}}a_{ij}g_{ij}t_j\Bigr|^{q}\Bigr)^{2/q}+qD_2^2\sum_{j\in J_{>k}}t_j^2\Bigr)^{1/2}
\\ \label{eq:revindk}
& \qquad+C\sum_{l=k+1}^{k_0}\sqrt{\Log|T_{\leq l}|}\max_{i}\max_{j\in J_{l}}|a_{ij}|.
\end{align}

For $k=k_0$  inequality \eqref{eq:revindk} holds with equality.
Observe that
\begin{align*}
\sup_{t\in T}
\Bigl(\Bigl(\sum_{i}\Bigr| & \sum_{j\in J_{\leq k}}a_{ij}g_{ij}t_j\Bigr|^{q}\Bigr)^{2/q}+qD_2^2\sum_{j\in J_{>k}}t_j^2\Bigr)^{1/2}
\\
&=
\sup_{\tilde{t}\in T_{\leq k}}
\Bigl(\Bigl(\sum_{i}\Bigr|\sum_{j\in J_{\leq k}}a_{ij}g_{ij}\tilde{t}_j\Bigr|^{q}\Bigr)^{2/q}
+d(\tilde{t})\Bigr)^{1/2},
\end{align*}
where
\[
d(\tilde{t})\coloneqq qD_2^2\sup_{t\in T\colon \pi_{\leq k}(t)=\tilde{t}}\sum_{j\in J_{>k}}t_j^2.
\]
Let $\Ex_{J_k}$ denote the integration with respect to random variables $(g_{ij})_{i\leq m,j\in J_k}$. Conditional application of Lemma \ref{lem:supgaussJ} yields
\begin{align*}
&\Ex_{J_k}\sup_{\tilde{t}\in T_{\leq k}}
\Bigl(\Bigl(\sum_{i}\Bigr|\sum_{j\in J_{\leq k}}a_{ij}g_{ij}\tilde{t}_j\Bigr|^{q}\Bigr)^{2/q}
+d(\tilde{t})\Bigr)^{1/2}
\\
&\leq \sup_{\tilde{t}\in T_{\leq k}}\Ex_{J_k}\Bigl(\Bigl(\sum_{i}\Bigr|\sum_{j\in J_{\leq k}}a_{ij}g_{ij}\tilde{t}_j\Bigr|^{q}\Bigr)^{2/q}+d(\tilde{t})\Bigr)^{1/2}
+C\sqrt{\Log|T_{\leq k}|}\max_{i}\max_{j\in J_k}|a_{ij}|.
\end{align*}

Inequality~\eqref{eq:intnormprop-q-norm}, applied conditionally, implies that 
for any $\tilde{t}\in T_{\leq k}$,
\begin{align*}
&\Ex_{J_k}\Bigl(\Bigl(\sum_{i}\Bigr|\sum_{j\in J_{\leq k}}a_{ij}g_{ij}\tilde{t}_j\Bigr|^{q}
\Bigr)^{2/q}+d(\tilde{t})\Bigr)^{1/2}
\\ & \qquad\qquad \le
\Bigl(\Bigl(\sum_{i}\Bigr|\sum_{j\in J_{\leq k-1}}a_{ij}g_{ij}\tilde{t}_j\Bigr|^{q}\Bigr)^{2/q}
+\tilde{d}(\tilde{t})\Bigr)^{1/2},
\end{align*}
where
\begin{align*}
\tilde{d}(\tilde{t})
=q\Bigl(\sum_{i}\Bigr|\sum_{j\in J_{k}}a_{ij}^2\tilde{t}_j^2\Bigr|^{q/2}\Bigr)^{2/q}
+d(\tilde{t}).
\end{align*}
Note that
\begin{align*}
\Bigl(\sum_{i}\Bigr|\sum_{j\in J_{k}}a_{ij}^2\tilde{t}_j^2\Bigr|^{q/2}\Bigr)^{2/q}
&\leq \sum_{j\in J_{k}}\tilde{t}_j^2\sup_{x\in B_1^{J_k}}\Bigl(\sum_{i}\Bigr|\sum_{j\in J_{k}}a_{ij}^2x_j\Bigr|^{q/2}\Bigr)^{2/q}
\\
&= \sum_{j\in J_{k}}\tilde{t}_j^2\max_{j\in J_k}\Bigl(\sum_{i}|a_{ij}|^q\Bigr)^{2/q}
 \leq D_2^2\sum_{j\in J_{k}}\tilde{t}_j^2,
\end{align*}
so
\begin{align*}
\tilde{d}(\tilde{t})
\le qD_2^2\sup_{t\in T\colon \pi_{\leq k}(t)=\tilde{t}}\sum_{j\in J_{> k-1}}t_j^2.
\end{align*}
Thus, 
\begin{align*}
&\Ex\sup_{t\in T}
\Bigl(\Bigl(\sum_{i}\Bigr|\sum_{j\in J_{\leq k}}a_{ij}g_{ij}t_j\Bigr|^{q}\Bigr)^{2/q}+qD_2^2\sum_{j\in J_{>k}}t_j^2\Bigr)^{1/2}
\\
&\leq \Ex\sup_{\tilde{t}\in T_{\leq k}}
\Bigl(\Bigl(\sum_{i}\Bigr|\sum_{j\in J_{\leq k-1}}a_{ij}g_{ij}\tilde{t}_j\Bigr|^{q}\Bigr)^{2/q}+
\tilde{d}(\tilde{t}) \Bigr)^{1/2}
+C\sqrt{\Log|T_{\leq k}|}\max_{i}\max_{j\in J_k}|a_{ij}|
\\
&\le\Ex\sup_{t\in T}
\Bigl(\Bigl(\sum_{i}\Bigr|\sum_{j\in J_{\leq k-1}}a_{ij}g_{ij}t_j\Bigr|^{q}\Bigr)^{2/q}+qD_2^2\sum_{j\in J_{>k-1}}t_j^2\Bigr)^{1/2}
\\
&\qquad+C\sqrt{\Log|T_{\leq k}|}\max_{i}\max_{j\in J_k}|a_{ij}|,
\end{align*}
so the reverse induction step immediately follows.

Estimate \eqref{eq:revindk} for $k=0$ gives
\[
\Ex\sup_{t\in T}\Bigl(\sum_{i}\Bigl|\sum_{j}a_{ij}g_{ij}t_j\Bigr|^q\Bigr)^{1/q}
\leq \sqrt{q}D_2+C\sum_{k=1}^{k_0}\sqrt{\Log|T_{\leq k}|}\max_{i}\max_{j\in J_{k}}|a_{ij}|.
\]
We have
\[
\Log|T_{\leq k}|\leq \sum_{l=1}^k\Log|T_{l}|\lesssim \sum_{l=1}^k l|J_l|,
\]
so that
\begin{align*}
\sum_{k=1}^{k_0}\sqrt{\Log|T_{\leq k}|}\max_{i}\max_{j\in J_{k}}|a_{ij}|
&\lesssim \sum_{1\leq l\leq k\leq k_0}\sqrt{l}|J_l|^{1/2}\max_{i}\max_{j\in J_{k}}|a_{ij}|
\\
&\leq \max_{1\leq l\leq k\leq k_0}k^3|J_l|^{1/2}\max_{i}\max_{j\in J_{k}}|a_{ij}|
\cdot\sum_{1\leq l\leq k\leq k_0}\frac{\sqrt{l}}{k^3}
\\
&\lesssim \max_{1\leq l\leq k\leq k_0}k^3|J_l|^{1/2}\max_{i}\max_{j\in J_{k}}|a_{ij}|.
\qedhere
\end{align*}
\end{proof}

Now we are ready to prove that Proposition~\ref{prop:weakerD3-allpq} follows from 
Proposition~\ref{prop:dimdependent}.

\begin{proof}[Proof of Proposition~\ref{prop:weakerD3-allpq}]
Let $r=8(p^*\vee q)$ and
\[
D_\infty'\coloneqq \max_{i,j} (i+j)^{1/r}|a_{ij}|.
\]
Define the sequence $(n_{k})_{k\geq 0}$ by
\[
n_0\coloneqq 0,\  n_k\coloneqq \bigl\lceil e^{r^k}\bigr\rceil  \quad\mbox{for }k\geq 1.
\]
Then
\begin{equation}
\label{eq:grownk}
 \sqrt{\Log(n_kn_l)}\leq \sqrt{\Log n_k}+\sqrt{\Log n_l}
\lesssim r  (n_{k-1}+n_{l-1}+2)^{1/r} \mbox{ for }k,l\geq 1.
\end{equation}

For $k=1,2,\ldots$ define
\begin{align*}
I_{k}&\coloneqq \{i\in [m]\colon n_{k-1}<i\leq n_{k}\}, \\
J_{k}&\coloneqq \{j\in [n]\colon n_{k-1}<j\leq n_{k}\}. 
\end{align*}
The definition of $D_\infty'$ implies that
\begin{equation}	
\label{eq:Dinfty-subsets}
\max_{i\in I_k, j\in J_l}(n_{k-1}+n_{l-1}+2)^{1/r}  |a_{ij}| 
\le \max_{i\in I_k, j\in J_l}  (i+j)^{1/r} |a_{ij}|
\le D_\infty'.
\end{equation}

Let
\[
A_{k,l}\coloneqq (a_{ij})_{i\in I_{k},j\in J_{l}}.
\]
Then  Proposition~\ref{prop:dimdependent} and estimates~\eqref{eq:grownk} and 
\eqref{eq:Dinfty-subsets} yield
\begin{align}	
\notag
\Ex \|G_{A_{k,l}}\|_{p\to q} &\lesssim  \alpha(p,q)\Bigl(\sqrt{p^*}D_1+\sqrt{q}D_2
+\sqrt{\Log(|I_k||J_l|)}\max_{i\in I_k,j\in J_l}|a_{ij}|\Bigr)
\\ \label{eq:single-block}
&  \lesssim  \alpha(p,q)(\sqrt{p^*}D_1+\sqrt{q}D_2+ rD_\infty').
\end{align}

For $u\in \zet$ set
\[
A_{u}=\sum_{ k>\max\{0,-u\}} (a_{ij}\ind_{\{i\in I_k,j\in J_{k+u}\}} )_{i\leq m,j\leq n}.
\]
For every  $u$ (after deleting some zero rows and columns) 
$A_u$ is block diagonal with blocks $A_{k,k+u}$.
Hence, Lemmas~\ref{lem:block-matr} and \ref{lem:supgauss}, 
together with estimates \eqref{eq:Dinfty-subsets} and \eqref{eq:single-block}, imply
\begin{align*}
\Ex \|G_{A_u}\|_{p\to q}
&=\Ex\max_{k}\|G_{A_{k,k+u}}\|_{p\to q}
\\
&\leq \max_{k}\Ex\|G_{A_{k,k+u}}\|_{p\to q}
+\max_k\sqrt{\Log k}\max_{i\in I_{k},j\in J_{k+u}}|a_{ij}|
\\
& \lesssim  \alpha(p,q)(\sqrt{p^*}D_1+\sqrt{q}D_2+  rD_{\infty}').
\end{align*}
We have
\[
A=\sum_{u=-1}^{1} A_u+B_1+B_2,
\]
where
\[
B_{1}=\sum_{l}\sum_{k\geq l+2}(a_{ij}\ind_{\{i\in I_k,j\in J_{l}\}}),
\qquad B_{2}=\sum_{k}\sum_{l\geq k+2}(a_{ij}\ind_{\{i\in I_k,j\in J_{l}\}}).
\]
Estimate \eqref{eq:sqrtJk} applied to the matrix $B_1$ yields
\begin{align*}
\Ex\|G_{B_1}\|_{p\to q}
&\lesssim \sqrt{q}D_2
+ \max_{l}\max_{l'\leq l}l^3\sqrt{ |J_{l'}|}\max_{k\geq l+2}\max_{i\in I_k}\max_{j\in J_l}|a_{ij}|
\\
&\leq \sqrt{q}D_2
+\max_{l}l^3\sqrt{n_l}\max_{k\geq l+2}\max_{i\in I_k}\max_{j\in J_l}|a_{ij}|.
\end{align*}
Observe that if $i\in I_{k}, j\in J_l$ and $k\geq l+2$ then 
\[
l^3\sqrt{n_l} \lesssim n_{l+1}^{1/r}\leq n_{k-1}^{1/r}\leq i^{1/r} \leq (i+j)^{1/r},
\]
Hence,
\[
\Ex\|G_{B_1}\|_{p\to q}  \lesssim \sqrt{q}D_2+D_\infty'.
\]
Similarily, estimate \eqref{eq:sqrtIk} implies
\[
\Ex\|G_{B_2}\|_{p\to q} \lesssim \sqrt{p}D_1+D_\infty'.
\qedhere
\]
\end{proof}

%%%%%%%%%%%%%%%%%%%%%%%%%%%%%%%%
%%%%%%%%%%%%%%%%%%%%%%%%%%%%%%%%

\section{Estimates in the range \texorpdfstring{$p^*,q\in [2,4)$}{p*,q in [2,4)}}
\label{sect:appendix}

%%%%%%%%%%%%%%%%%%%%%%%%%%%%%%%%

The aim of this section is to give  proofs of Propositions \ref{prop:dimdependent} 
and \ref{prop:worseD3'-intro} in the range $p^{*},q\in[2,3]$.
The arguments presented in this section work  in the whole regime $p^*,q\in [2,4)$, but they yield a constant which blows up when $p^*$ or $q$ approaches $4$.

The crucial new tool we need to establish Proposition~\ref{prop:dimdependent} in the case 
$p^*,q\in[2,3]$  is the following result, which allows to run the exponent reduction; 
this is a counterpart of Corollaries \ref{cor:uptodAepsiloncase1} and
\ref{cor:uptodAepsiloncase4} from Section~\ref{sect:degdepbounds}.

\begin{prop}
\label{prop:pqdA}
 If  $p^*,q\in [2, 4-\delta)$ and $\ve\in(0,1/2]$, then
\begin{equation*}
\Ex\|G_A\|_{p\to q}
 \lesssim \ve^{-2} \Bigl[ D_1+D_2 +
\Bigl(\Bigl(\ve^{-1/2}+\Bigl(\frac{C}{\delta}\Bigr)^{2/\delta} \Bigr)d_A^{1/4+\ve}
+\sqrt{\Log(mn)}\Bigr)\max_{i,j}|a_{ij}| \Bigr].
\end{equation*}
\end{prop}

To prove Proposition~\ref{prop:pqdA}  (and Proposition \ref{prop:worseD3'-intro}
for $p^{*},q\in[2,3]$) we will need the following result.

\begin{prop}
\label{prop:slepianp*qbelow4gen}
If $ p^*, q\in[2,4)$ and $a,b\in (0,1]$, then
\begin{multline*}
\Ex \sup_{s\in B_{q^*}^m\cap aB_\infty^m}\ \sup_{t\in B_p^n\cap bB_\infty^n}
\sum_{i,j}a_{ij}g_{ij}s_it_j
\\
\lesssim D_1+D_2
+\Bigl(a^{(2-q^*)/2}\Log^{\frac{2}{4-p^*}}(ma^{q^*})
+b^{(2-p)/2}\Log^{\frac{2}{4-q}}(nb^p)\Bigr)\max_{i,j} |a_{ij}| .
\end{multline*}
\end{prop}

\begin{proof}[Proof of Propositions~\ref{prop:dimdependent} and \ref{prop:worseD3'-intro} for 
$2\leq p^{*},q\leq 3$]
Proposition~\ref{prop:slepianp*qbelow4gen}, applied with $a=b=1$, yields 
Proposition~\ref{prop:worseD3'-intro} (with $\gamma=2$). 
Proposition~\ref{prop:dimdependent} follows from Proposition~\ref{prop:fromuptopowerdA} together 
with  Propositions~\ref{prop:worseD3'-intro} and \ref{prop:pqdA}.
\end{proof}

To prove Proposition~\ref{prop:slepianp*qbelow4gen} we will use a modification of van Handel's argument from \cite{vH}. 
Let $A\circ A=(a_{ij}^2)_{i\le m, j\le n}$ be the variance  profile of 
$G_A=(a_{ij}g_{ij})_{i\le m,j\le n}$ and
\[
B=(b_{ij})_{i,j\leq m+n}
=\begin{pmatrix}
0 &  A\circ A    
\\
(A\circ A)^T& 0 
\end{pmatrix}.
\]
Let $Y$ be an $(m+n)$-dimensional Gaussian vector with mean $0$ and covariance matrix $B^-$  
being the negative part of $B$, i.e., $B^-=-\sum_{i=1}^{m+n} (\lambda_i\wedge 0)u_iu_i^T$, 
where $B=\sum_{i=1}^{m+n} \lambda_i u_iu_i^T$ is the spectral decomposition of $B$. 
The proof of \cite[Corollary~4.2]{vH} yields that  $Y_k$ are Gaussian random variables with 
variance $\Ex Y_k^2\leq (\sum_{l}b_{k,l}^2)^{1/2}$. Hence, 
\begin{equation}
\label{eq:VarYi}
\begin{split}\mathrm{Var}(Y_i) &\leq  \Bigl(\sum_{j}a_{ij}^4\Bigr)^{1/2},
\quad 1\leq i\leq m,
\\ 
\mathrm{Var}(Y_{j+m}) &\leq  \Bigl(\sum_{i}a_{ij}^4\Bigr)^{1/2},
\quad 1\leq j\leq n.
\end{split}
\end{equation}

The proof of Proposition~\ref{prop:slepianp*qbelow4gen} uses  the following two lemmas. The first one is based on the ideas from \cite{vH} and the second one is quite  standard  (cf.\ \cite[Lemma~14]{Lbern} for cases $\rho=1,2$).

 \begin{lem}
 \label{lem:KtimesL}
 For any bounded, nonempty sets $K\subset \er^m$ and $L\subset \er^n$,
 \begin{align*}
 \Ex\sup_{s\in K,t\in L}a_{ij}g_{ij}s_it_j
 &\leq 
 \Ex\sup_{s\in K,t\in L}\sum_{i=1}^m s_ig_i\sqrt{\sum_{j=1}^n t_j^2a_{ij}^2}
 \\ & \qquad+\Ex\sup_{s\in K,t\in L}\sum_{j=1}^n t_jg_{m+j}\sqrt{\sum_{i=1}^m s_i^2a_{ij}^2}
 \\
 &\qquad
 +\frac{1}{2}\Ex\sup_{s\in K}\sum_{i=1}^ms_i^2Y_i
 +\frac{1}{2}\Ex\sup_{t\in L}\sum_{j=1}^n t_j^2Y_{m+j}.
 \end{align*}
 \end{lem}
 
 \begin{proof}
 Let us define the symmetric Gaussian matrix
 \[
X=(X_{ij})_{i,j\le m+n} 
= \begin{pmatrix}
0 &  G_A    
\\
(G_A)^T& 0
\end{pmatrix}.
\] 
Then $B$ is the variance profile of $X$. Consider two centered Gaussian processes  
indexed by $v\in \er^{n+m}$:
\[
G_{v}=\sum_{k,l\le m+n} X_{ kl}v_kv_l \quad \text{ and }\quad 
Z_v=2\sum_{k=1}^{m+n} v_kg_k\sqrt{\sum_{l=1}^{m+n} v_l^2b_{k,l}^2}+\sum_{k=1}^{m+n}v_k^2Y_k,
\] 
It was shown in  \cite[proof of Theorem~4.1]{vH} that for any $v,v'\in \er^{n+m}$,
$\Ex |G_v-G_{v'}|^2\leq \Ex |Z_v-Z_{v'}|^2$ and, as a consequence, the Slepian-Fernique 
lemma (see, e.g., \cite[Theorem 13.3]{BLM2013}) yields
\[
\Ex\sup_{v\in K\times L} G_v \le \Ex \sup_{v\in K\times L} Z_v.
\]
Thus, to finish the proof it is enough to observe that
\begin{align*}
\sup_{v\in K\times L} G_v&=
\sup_{v\in K\times L} \sum_{k,l\le m+n} X_{kl}v_kv_l 
\\ 
&= \sup_{s\in K, t\in L} \Biggl( \sum_{i\le m, j\le n} (G_A)_{ij}s_it_j 
+ \sum_{i\le m, j\le n} ((G_A)^T)_{ji}t_j s_i\Biggr)
\\
&=2\sup_{s\in K,t\in L} \sum_{i\leq m,j\leq n}a_{ij}g_{ij}s_it_j
\end{align*}
and
\begin{align*}
\sup_{v\in K\times L} Z_v
=  \sup_{s\in K,t\in L}\Biggr(2\sum_{i=1}^m & s_ig_i\sqrt{\sum_{j=1}^n t_j^2a_{ij}^2}
 +2\sum_{j=1}^n t_jg_{m+j}\sqrt{\sum_{i=1}^m s_i^2a_{ij}^2}
\\
&
 +\sum_{i=1}^ms_i^2Y_i+\sum_{j=1}^n t_j^2Y_{m+j}\Biggr).
\qedhere
\end{align*}
 \end{proof}

\begin{lem}
\label{lem:ellqmaxk}
Suppose that $1\leq \rho\leq \rho_0$, $c\in (0,1]$, $m_j,\sigma_j\geq 0$ and nonnegative random variables $Z_1,\ldots,Z_n$  satisfy
\[
\Pr(Z_j\geq m_j+t\sigma_j)\leq e^{-t^2/2} \quad \mbox{for every } 1\leq j\leq n.
\]
Then
\[
\Ex \sup_{u\in B_2^n\cap cB_\infty^n}\Bigl(\sum_{j=1}^n u_j^2 Z_j^\rho\Bigr)^{1/\rho}
\lesssim_{\rho_0} \max_{j}m_j+\sqrt{\Log(nc^2)}\max_{j}\sigma_j.
\]
\end{lem}

\begin{proof}
Let $k$ be a positive integer such that $\frac{1}{k+1}< c^2\leq \frac{1}{k}$ and 
$(Z_1^*,\ldots,Z_n^*)$ be the nonincreasing rearrangement of $(Z_1,\ldots,Z_n)$. 
Let $m=\max_j m_j$, $\sigma=\max_j\sigma_j$, and $M=(m+\sqrt{2\Log(n/k)}\sigma)^\rho$. 
Then

\begin{align*}
\Bigl( \frac 1k \sum_{l=1}^k & \Ex (Z_l^*)^\rho \Bigr)^{1/\rho}
\le  2
\Bigl(  M 
+ \frac 1k \sum_{l=1}^n\Ex  \bigl(Z_l-m-\sqrt{2\Log(n/k)}\, \sigma\bigr)_{+}^\rho\Bigr)^{1/\rho}
\\ 
& \lesssim
\Bigl( M  
+ \frac 1k \sum_{l=1}^n\sigma^\rho\int_0^\infty\rho t^{\rho-1}\
\Pr\Bigl(Z_l\geq m+\sigma\bigl(t+\sqrt{2\Log(n/k)}\bigr)\Bigr)dt\Bigr)^{1/\rho}
\\
&  \le 
\Bigl( M  + 
\frac nk\sigma^\rho \int_{0}^\infty \rho t^{\rho -1}e^{-(t+\sqrt{2\Log(n/k)})^2/2}dt \Bigr)^{1/\rho}
\\ 
&  
\le \Bigl( M  +  
\sigma^\rho\int_{0}^\infty \rho t^{\rho -1} e^{-t^2/2}dt \Bigr)^{1/\rho}
\lesssim_{\rho_0}  m+\sqrt{\Log(n/k)}\, \sigma,
\end{align*}
so
\begin{align*}
\Ex\sup_{u\in B_2^n\cap cB_\infty^n}
\Bigl(\sum_{j=1}^n  u_j^2 Z_j^\rho\Bigr)^{1/\rho}
&\leq
\Ex\Bigl(\frac{1}{k}\sum_{l=1}^k(Z_l^*)^\rho\Bigr)^{1/\rho}
\lesssim_{\rho_0} m+\sqrt{\Log(n/k)}\, \sigma
\\
&\sim m+\sqrt{\Log(nc^2)}\, \sigma.
\qedhere
\end{align*}

\end{proof}

\begin{proof}[Proof of Proposition \ref{prop:slepianp*qbelow4gen}]
We apply Lemma~\ref{lem:KtimesL} with $K=B_{q^*}^m\cap a B_\infty^m$ and $L=B_{p}^n\cap b B_\infty^n$. 
Observe that $\|s\|_2\leq \|s\|_\infty^{(2-q^*)/2}\|s\|_{q*}^{q^*/2}$. This, together with a similar 
calculation for $\ell_p$-norms, yields
\begin{equation} \label{eq:slepianp*qbelow4gen-1}
K\subset a^{(2-q^*)/2}(B_2^m\cap a^{q^*/2}B_\infty^m),\quad
L\subset b^{(2-p)/2}(B_2^n\cap b^{p/2}B_\infty^n).
\end{equation}

Estimate \eqref{eq:VarYi} implies that $Y_1,\ldots,Y_m$ are centered Gaussians with 
$\mathrm{Var}(Y_i)\leq \|(a_{ij})_j\|_4^2$. Thus, \eqref{eq:slepianp*qbelow4gen-1} and 
Lemma~\ref{lem:ellqmaxk}, applied with $Z_j=|Y_j|$, $\rho=1$, $c=a^{q^*/2}$, $m_j=0$ and 
$\sigma_{i}=\|(a_{ij})_j\|_4$, yield  that
\begin{align*}
\Ex\sup_{s\in K}\sum_{i=1}^ms_i^2Y_i 
&
\leq a^{2-q^*}\Ex \sup_{s\in B_2^m\cap a^{q^*/2}B_\infty^m}\sum_{i=1}^ms_i^2|Y_i|
\\ 
&\lesssim  a^{2-q^*}\sqrt{\Log(ma^{q^*})}\max_{i}\|(a_{ij})_j\|_4
\\
&\leq 
a^{2-q^*}\sqrt{\Log(ma^{q^*})}\max_{i}\|(a_{ij})_j\|_{p^*}^{p^*/4}
\max_{i,j}|a_{ij}|^{(4-p^*)/4}
\\
&\leq \frac{p^*}{4} D_1
+\frac{4-p^*}{4}a^{4(2-q^*)/(4-p^*)}\Log^{2/(4-p^*)}(ma^{q^*})\max_{i,j}|a_{ij}|
\\
& \leq D_1+a^{(2-q^*)/2}\Log^{2/(4-p^*)}(ma^{q^*})\max_{i,j}|a_{ij}|.
\end{align*}
 
In a similar way we show that
\begin{align*}
\sup_{t\in L}\sum_{j=1}^n t_j^2Y_{m+j}
&\lesssim b^{2-p}\sqrt{\Log(nb^p)}\max_{j}\|(a_{ij})_i\|_4
\\
&\leq  D_2+b^{(2-p)/2}\Log^{2/(4-q)}(nb^p)\max_{i,j}|a_{ij}|.
\end{align*}

By \eqref{eq:slepianp*qbelow4gen-1} and the convexity of 
the function $x\mapsto |x|^{q/2}$  we get
\begin{align}
\notag
\sup_{s\in K,t\in L}\sum_{i=1}^m s_ig_i &\sqrt{\sum_{j=1}^n t_j^2a_{ij}^2}
  \leq \sup_{t\in  L}\Bigl(\sum_{i=1}^m |g_i|^q \Bigl(\sum_{j=1}^n t_j^2a_{ij}^2\Bigr)^{q/2} \Bigr)^{1/q}
\\ \notag
&\leq b^{(2-p)/2}\sup_{t\in B_2^n\cap b^{p/2}B_\infty^n}\Bigl(\sum_{i=1}^m |g_i|^q \Bigl(\sum_{j=1}^n t_j^2a_{ij}^2\Bigr)^{q/2} \Bigr)^{1/q}
\\ \notag
&
\leq b^{(2-p)/2}\sup_{t\in B_2^n\cap b^{p/2}B_\infty^n}
\Bigl(\sum_{j=1}^n t_j^2\|(a_{ij}g_i)_i\|_q^q\Bigr)^{1/q}
\Bigr( \sum_{j=1}^n t_j^2\Bigr)^{1/2-1/q}
\\	\label{eq:estimate-strong-term-slepian-gen}
&\leq b^{(2-p)/2}\sup_{t\in B_2^n\cap b^{p/2}B_\infty^n}
\Bigl(\sum_{j=1}^n t_j^2\|(a_{ij}g_i)_i\|_q^q\Bigr)^{1/q}.
\end{align}
Note that
\[
\Ex\|(a_{ij}g_i)_i\|_q\leq (\Ex\|(a_{ij}g_i)_i\|_q^q)^{1/q}\leq \sqrt{q}\|(a_{ij})_i\|_q,
\]
and
\[
\sup_{s\in B_2^n}\|(a_{ij}s_i)_i\|_q=\max_{i}|a_{ij}|.
\]
Therefore, the Gaussian concentration (see, e.g., \cite[Theorem~5.6]{BLM2013}) yields
\[
\Pr\bigl(\|(a_{ij}g_i)_i\|_q\geq \sqrt{q}\|(a_{ij})_i\|_q+t\max_{i}|a_{ij}|\bigr)\leq e^{-t^2/2}.
\]
Estimate~\eqref{eq:estimate-strong-term-slepian-gen} and Lemma \ref{lem:ellqmaxk}, applied with $Z_j=\|(a_{ij}g_i)_i\|_q$, $\rho = q$, $c=b^{p/2}$, and $\rho_0=4$, imply
\begin{align*}
\Ex\sup_{s\in K,t\in L}\sum_{i=1}^m s_ig_i & \sqrt{\sum_{j=1}^n t_j^2a_{ij}^2}
\lesssim \max_j\|(a_{ij})_i\|_q+b^{(2-p)/2}\sqrt{\Log(nb^p)}\max_{i,j}|a_{ij}|
\\
&\leq D_2+b^{(2-p)/2}\Log^{ 2/(4-q)}(nb^p)\max_{i,j}|a_{ij}|.
\end{align*}

In a similar way we show that 
\[
\Ex\sup_{s\in K,t\in L}\sum_{j=1}^n t_jg_j  \sqrt{\sum_{i=1}^m s_i^2a_{ij}^2}
\lesssim D_1
+a^{(2-q^*)/2}\Log^{ 2/(4-p^*)}(ma^{q^*})\max_{i,j}|a_{ij}|.
\qedhere
\]
\end{proof}

Now we go back to the proof of Proposition~\ref{prop:pqdA}. We will need a refinement of the arguments from the proof of the main result of \cite{APSS}.
Let
\[
K_{p,n}\coloneqq 
\conv \Bigl\{ \frac{1}{|J|^{1/p}}\bigl( \varepsilon_j \ind_{\{j\in J\}} \bigr)_{j=1}^n \colon 
J\subset\{1,\dots,n\}, J\neq \emptyset, (\varepsilon_j )_{j=1}^n\in \{-1,1\}^n\Bigr\}.
\]
Recall \cite[Lemma~2.3]{APSS}.
\begin{lem}
If $p\ge 1$, then $B_p^n \subset \ln(en)^{1/p^*} K_{p,n}$.
\end{lem}
Following the ideas from \cite{APSS} we obtain the following bound.

\begin{lem}
\label{lem:slepian-flat-small}
Assume that  $p^*,q  \geq 2$. Then
\begin{equation*}
\Ex\sup_{s\in K_{q^*,m}, t\in K_{p,n}} \sum_{i,j} a_{ij}g_{ij}s_it_j 
\lesssim   \sqrt{p^*}D_1 +\sqrt{q}D_2+\sqrt{\Log (mn)}\max_{i,j}|a_{ij}|.
\end{equation*}
\end{lem}

\begin{proof}
The proof of \cite[Proposition~3.1]{APSS}  (see estimate (3.6) and the last formula on 
page 3492 therein, based on the Slepian-Fernique lemma) shows that for every 
$1\le k\le m$ and $1\le l\le n$,
\begin{multline*}	
\Ex \max_{I, J} \sup_{\eta\in B_\infty^m}\sup_{\eta'\in B_\infty^n}
\sum_{i\in I, j\in J}  a_{ij} g_{ij} \eta_i\eta_j'
\lesssim
\max_{I, J}  \sum_{i\in I}  \sqrt{\sum_{j\in J} a_{ij}^2}
+  \max_{I, J} \sum_{j\in J} \sqrt{\sum_{i\in I} a_{ij}^2 }
\\
+  \Ex  \max_{I,J} \sum_{i\in I} g_i \sqrt{\sum_{j\in J} a_{ij}^2 }
+  \Ex  \max_{I,J} \sum_{j\in J} \widetilde{g}_j \sqrt{\sum_{i\in I} a_{ij}^2 },
\end{multline*}
where the maxima are taken over all sets
$I\subset \{1,\ldots,m\}$, $J\subset \{1,\ldots,n\}$ such that $|I| = k$, $|J| = l$, 
and $g_1,\ldots, g_m,  \widetilde{g}_1,\ldots, \widetilde{g}_n$ are independent 
standard Gaussian variables.
Moreover,
\begin{align*}	
\frac{1}{k^{1/q^*}l^{1/p}} \max_{I, J}  \sum_{i\in I}  \sqrt{\sum_{j\in J} a_{ij}^2} 
& \le \sup_{s\in B_{q^*}^m}\sup_{t\in B_{p}^n}  
\sum_{i=1}^m s_i  \sqrt{\sum_{j=1}^n t_j^2 a_{ij}^2} 
\\ & = \sup_{t\in B_{p}^n} \Bigl( \sum_{i=1}^m \Bigl(\sum_{j=1}^n t_j^2 a_{ij}^2 \Bigr)^{q/2} \Bigr)^{1/q}
=\max_{j}\|(a_{ij})_{ij}\|_q
= D_2.
\end{align*}
Similarly,
\[
\frac{1}{k^{1/q^*}l^{1/p}} \max_{I, J}  \sum_{j\in J}  \sqrt{\sum_{i\in I} a_{ij}^2} 
\le  D_1,
\]
and
\begin{align}	
\label{eq:maxlqcolumns}
\frac{1}{k^{1/q^*}l^{1/p}} 	\Ex \max_{I, J} \sum_{i\in I} g_i \sqrt{\sum_{j\in J} a_{ij}^2 }
\le \Ex \max_{j} \|(a_{ij}g_i)_i\|_q.
\end{align}	
Estimate~\eqref{eq:supgauss} yields  for every $q\in [2,\infty)$,
\begin{align*}
\Ex \max_{j\le n} \| (a_{ij}g_i)_i\|_q 
& \lesssim \ \max_{j\le n}  \Ex  \| (a_{ij}g_i)_i\|_q +
 \sqrt{\Log n}\max_{i,j}|a_{ij}|  
\\ 
& \le \ \max_{j\le n} \bigl( \Ex  \| (a_{ij}g_i)_i\|_q^q\bigr)^{1/q} +
 \sqrt{\Log n}\max_{i,j}|a_{ij}|  
\\ 
&
\le \sqrt{q}D_2+  \sqrt{\Log n}\max_{i,j}|a_{ij}| .	
\end{align*}
This and \eqref{eq:maxlqcolumns} yield
\begin{align*}	
\frac{1}{k^{1/q^*}l^{1/p}} 	\Ex \max_{I, J} \sum_{i\in I} g_i \sqrt{\sum_{j\in J} a_{ij}^2 }
\lesssim \sqrt{q}D_2+  \sqrt{\Log n}\max_{i,j}|a_{ij}| .	
\end{align*}
Similarly, 
\[
 \frac{1}{k^{1/q^*}l^{1/p}} 	\Ex \max_{I, J} \sum_{j\in J} \widetilde{g}_j \sqrt{\sum_{i\in I} a_{ij}^2 } 
\lesssim  \sqrt{p^*}D_1+ \sqrt{\Log m}\max_{ i,j}|a_{ij}|.
\]

This  implies that for every $k,l\ge 1$,
\begin{equation}	
\label{eq:slepian-flat-fixedsize}
\Ex W_{k,l} 
\lesssim  \sqrt{p^*}D_1 +\sqrt{q}D_2 +\sqrt{\Log (mn)}\max_{i,j}|a_{ij}|,
\end{equation}
 where
\[
W_{k,l} \coloneqq  \sup_{s\in K_{q^*,m}^{(k)}, t\in K_{p,n}^{(l)}} \sum_{i,j} a_{ij}g_{ij}s_it_j 
\] 
and 
\begin{align*}
K_{p,n}^{(l)}  \coloneqq \{l^{-1/p} (\eta_j\ind_{j\in J}) \colon \eta\in\{-1,1\}^n, J\subset[n], |J|=l\}.
\end{align*}
Hence,  by inequality \eqref{eq:supgauss}  we obtain
\begin{align*}	
\Ex  \sup_{s\in K_{q^*,m}, t\in K_{p,n}} & \sum_{i,j} a_{ij}g_{ij}s_it_j 
 =\Ex\max_{1\le k\le m, 1\le l\le n} W_{k,l}	
\\ 
&\lesssim	\max_{1\le k\le m, 1\le  l\le n} \Ex W_{k,l} 
+\sqrt{\Log(mn)} \sup_{ s\in B_{q^*}^m, t\in B_p^n} \sqrt{ \sum_{i,j}a_{ij}^2 s_i^2 t_j^2}
\\
& = \max_{ 1\le k\le m,  1\le  l\le n} \Ex W_{k,l} 
+\sqrt{\Log(mn)} \max_{i,j}|a_{ij}|.
\end{align*}
 Thus, \eqref{eq:slepian-flat-fixedsize} yields the assertion.
\end{proof}

As in Section~\ref{sect:degdepbounds}, we estimate 
$\Ex\|G_A\|_{p\to q}$ using Propostion~\ref{prop:redtosupflat}, so we need to upper bound 
$\Ex Y_{k,l}(d_A^\ve)$. The rest of this section is devoted to do so.
We begin with an easy consequence of  Lemma~\ref{lem:slepian-flat-small} and 
Proposition~\ref{prop:redtosupflat}.

\begin{cor}
\label{cor:logdA}
If $2\le p^*,q<\infty$, then
\begin{multline*}
\Ex \|G_A\|_{p\to q}
\\ \lesssim \Log(d_A)
\Bigl(\sqrt{p^*}\max_{i}\|(a_{ij})_j\|_{p^*}+\sqrt{q}\max_{j}\|(a_{ij})_i\|_{q}+\sqrt{\Log (nm)}\max_{i,j}|a_{ij}| \Bigr).
\end{multline*}
\end{cor}

\begin{proof}
If $I\subset [m]$, $J\subset [n]$, and $\gamma>1$, then for every  $(b_i)_{i=1}^m$ and 
$(c_{j})_{j=1}^n$,
\[
\sup_{s\in  K(q^*,m,I,\gamma)}\sum_{i=1}^m b_is_i\leq \gamma\sup_{s\in K_{q^*,m}}\sum_{i=1}^m b_is_i
\]
and
\[
\sup_{t\in  K(p,n,J,\gamma)}\sum_{j=1}^n c_jt_j\leq \gamma\sup_{t\in K_{p,n}}\sum_{j=1}^n c_jt_j.
\]
Hence,
\[
Y_{k,l}(d_A^\ve)\leq d_A^{2\ve}\sup_{s\in K_{ q^*,m}}\sup_{t\in K_{ p,n}}
\sum_{i,j}a_{ij}g_{ij}s_it_j.
\]
We choose $\ve =1/\Log(d_A)$ and apply Lemma~\ref{lem:slepian-flat-small} and Proposition~\ref{prop:redtosupflat}.
\end{proof}

To derive  better bounds for $\Ex Y_{k,l}(d_A^\ve)$ we  first need to introduce  some
additional notation. For $I\subset [m]$ and $\alpha\geq 0$ define
\[
B(p,I,\alpha)=B(p,I,\alpha,A)\coloneqq \Bigl\{t\in B_p^n\colon\ \max_{i\in I}\sum_{j=1}^na_{ij}^2t_j^2\leq \alpha^2\Bigr\},
\]
and for $I\subset [m]$ and $J\subset [n]$ let 
\[
	I''=\{i\in [m]\colon\ (i,j)\in E_A \text{ for some }   j\in I'\} 
\]
 and similarly
 \[
  J''=\{j\in [n]\colon\   (i,j)\in E_A  \text{ for some } i\in J'  \}.
\]
Without loss of generality we may assume that matrix $A$ has no zero rows and columns, so 
$I\subset I''$ and $J\subset J''$ for any $I\subset [m]$ and $J\subset [n]$.

\begin{lem}
\label{lem:decomp}
For every  $p\leq 2$, $1\leq r\leq k\leq m$ and $1\leq l\leq n$ we have 
\begin{align}	
\notag
\Ex& Y_{k,l}(d_A^\ve)
\\ \notag
& \leq \Ex\max_{I\in \mathcal{I}_4(k)}\max_{I_0\subset I,|I_0|=r}\max_{J\in \mathcal{J}_4(l)}\
\sup_{s\in K(q^*,m,I,d_A^\ve)}\ \sup_{t\in K(p,n,J,d_A^\ve)}
\sum_{i\in I_0''\cap I,j\in I_0'\cap J}a_{ij}g_{ij}s_it_j
\\ \label{eq:decomp}
&\qquad+\Ex \max_{I\in \mathcal{I}_4(k)}\sup_{s\in B_{q^*}^m}\sup_{t\in B(p,I,\alpha_r)}\sum_{i\in I,j\in [n]}a_{ij}g_{ij}s_it_j,
\end{align}
where
\[
\alpha_r=\min\{1,d_A^\ve l^{1/2-1/p}r^{-1/2}\}\max_{i,j}|a_{ij}|.
\]

\end{lem}

\begin{proof}
Let us fix $I\in \mathcal{I}_4(k)$ and $J\in \mathcal{J}_4(l)$ and
consider the following greedy algorithm with output being a subset $I_0=\{i_1,\ldots,i_r\}$ 
of $I$ of size $r$.
\begin{itemize}
\item	In the first step we pick a vertex $i_1\in I$ with maximal number of neighbours in $J$.
\item  Once we chose $\{i_1,\ldots,i_u\}$ for $u<r$, we pick  
$i_{u+1}\in I\setminus \{i_1,\ldots,i_u\}$ with maximal number of neighbours in  
$J\setminus \{i_1,\ldots,i_u\}'$.
\end{itemize}
If $l_u$ denotes the number of neighbours of $i_u$ in $J\setminus \{i_1,\ldots,i_{u-1}\}'$, 
then $l_1\ge l_2\ge \ldots \ge l_r$, so $rl_r \le |J|=l$.
Hence, using this algorithm  we get a subset $I_0\subset I$ with
cardinality $r$ such that for every $i\in I\setminus I_0$, 
$|\{ j\in J\setminus I_0'\colon (i, j)\in E_A \}|\leq l/r.$

Observe that if $i\in I$ and $j\in I_0'\cap J$ are such that $a_{ij}\neq 0$, then 
$i\in I_0''\cap I$,  and if $i\in I$ and $j\in J\setminus I_0'$ are such that $a_{ij}\neq 0$, then $i\in I\setminus I_0$.
Hence, for any $s\in B_{q^*}^m$ and $t\in B_p^n$,
\[
\sum_{i\in I,j\in J}a_{ij}g_{ij}s_it_j
=\sum_{i\in I_0''\cap I,j\in I_0'\cap J}a_{ij}g_{ij}s_it_j
+\sum_{i\in I\setminus I_0,j\in J\setminus I_0'}a_{ij}g_{ij}s_it_j.
\]
Moreover, for any $i\in I\setminus I_0$ and $t\in K(p,n,J,d_A^\ve)$,
\[
\sum_{j\in J\setminus I_0'}a_{ij}^2t_j^2
\leq \max_{i,j}a_{ij}^2\min\Bigl\{1,\frac{l}{r}\max_{j\in J}t_j^2\Bigr\}
\leq \max_{i,j}a_{ij}^2\min\{1,d_A^{2\ve}l^{1-2/p}r^{-1}\}=\alpha_r^2. 
\]
Therefore,
\begin{multline*}
\max_{I\in \mathcal{I}_4(k)}\max_{J\in \mathcal{J}_4(l)}\
 \sup_{s\in K(q^*,m,I,d_A^\ve)}\ \sup_{t\in K(p,n,J,d_A^\ve)}
\sum_{i\in I\setminus I_0,j\in J\setminus I_0'}a_{ij}g_{ij}s_it_j
\\
\leq 
\max_{I\in \mathcal{I}_4(k)}\sup_{s\in B_{q^*}^m}\sup_{t\in B(p,I,\alpha_r)}\sum_{i\in I,j\in [n]}a_{ij}g_{ij}s_it_j.
\qedhere
\end{multline*}
\end{proof}

We begin by estimating the second term on the right-hand side of \eqref{eq:decomp}.

\begin{lem}
\label{lem:outsideI0}
For $p^*,q\in [2,\infty)$, $1\leq k\leq m$, and $\alpha \in[0,1]$ we have
\begin{multline*}
\Ex \max_{I\in \mathcal{I}_4(k)}\sup_{s\in B_{q^*}^m}\sup_{t\in B(p,I,\alpha)}\sum_{i\in I,j\in [n]}a_{ij}g_{ij}s_it_j
\\
\lesssim \sqrt{p^*}D_1+\alpha\bigl(\sqrt{k\Log d_A}+\sqrt{\Log m}\bigr).
\end{multline*}
\end{lem}

\begin{proof}
Observe that for any $I\in \mathcal{I}_4(k)$, $s\in B_{q^*}^m\subset B_2^m$, and 
$t\in B(p,I,\alpha)$ we have
\[
\sum_{i\in I,j\in [n]}a_{ij}^2s_i^2t_j^2\leq \max_{i\in I}\sum_{j}a_{ij}^2t_j^2\leq \alpha^2.
\]

For $s\in B_{q^*}^m$ define
\[
D(p,s,\alpha)\coloneqq \Bigl\{t\in B_{p}^n\colon \sum_{i,j}a_{ij}^2s_i^2t_j^2\leq \alpha^2\Bigr\}.
\]
For $I\in\mathcal{I}_4(k)$ let us choose a $1/2$-net $S_I$ in 
$B_{q^*}^I=\{s\in B_{q^*}^m\colon \mathrm{supp}(s)\subset I\}$ (in $\ell_{q^*}$-norm) 
of cardinality at most $5^k$ and put $S\coloneqq \bigcup_{I\in\mathcal{I}_4(k)}S_I$. 
Then
\begin{align*}
\max_{I\in \mathcal{I}_4(k)}\sup_{s\in B_{q^*}^m}\sup_{t\in B(p,I,\alpha)}\sum_{i\in I,j\in [n]}a_{ij}g_{ij}s_it_j
&\leq
2\max_{I\in \mathcal{I}_4(k)}\max_{s\in S_I}\sup_{t\in B(p,I,\alpha)}\sum_{i\in I,j\in [n]}a_{ij}g_{ij}s_it_j
\\
&\leq
2\max_{s\in S}\sup_{t\in D(p,s,\alpha)}\sum_{i,j}a_{ij}g_{ij}s_it_j.
\end{align*}
Thus,  estimate~\eqref{eq:supgauss}  implies
\begin{multline*}
\Ex\max_{I\in \mathcal{I}_4(k)}\sup_{s\in B_{q^*}^m}\sup_{t\in B(p,I,\alpha)}
\sum_{i\in I,j\in [n]}a_{ij}g_{ij}  s_it_j
\lesssim \max_{s\in S}\Ex\sup_{t\in D(p,s,\alpha)}\sum_{i,j}a_{ij}g_{ij}s_it_j
\\
 +\sqrt{\Log  |S|}\max_{s\in S}\sup_{t\in D(p,s,\alpha)}
 \Bigl(\sum_{i,j}a_{ij}^2s_i^2t_j^2\Bigr)^{1/2}.
\end{multline*}

For any fixed $s\in S\subset B_{q^*}^m$  estimate \eqref{eq:fixeds} yields
\[
\Ex\sup_{t\in D(p,s,\alpha)}\sum_{i,j}a_{ij}g_{ij}s_it_j
\leq \Ex\sup_{t\in B_p^n}\sum_{i,j}a_{ij}g_{ij}s_it_j
\leq \sqrt{p^*}D_1.
\]
Moreover, by \eqref{eq:size-Ir(k)}, $|S|\leq 5^k|\mathcal{I}_4(k)|\leq m20^kd_A^{4k}$, so
$\sqrt{\Log   |S|}\lesssim \sqrt{\Log m}+\sqrt{k\Log(d_A)}$. Finally, 
\[
\max_{s\in S}\sup_{t\in  D(p,s,\alpha)}\Bigl(\sum_{i,j}a_{ij}^2s_i^2t_j^2\Bigr)^{1/2}\leq \alpha.
\qedhere
\]
\end{proof}

Now we estimate the first term on the right-hand side of \eqref{eq:decomp}.

\begin{lem}
\label{lem:I_0r}
If $l\geq r$, $\ve\in (0,1/2]$, and $ p^*, q\in[2,4)$, then
\begin{align*}
\Ex&\max_{I\subset [m]}\max_{I_0\subset I,|I_0|=r}\max_{J\in \mathcal{J}_4(l)}\
\sup_{s\in K(q^*,m,I,d_A^\ve)}\ \sup_{t\in K(p,n,J,d_A^\ve)}
\sum_{i\in I_0''\cap I,j\in I_0'\cap J}a_{ij}g_{ij}s_it_j
\\
&\lesssim
\ve^{-1}\Bigl[ D_1+D_2
\\ 
&\quad
+\Bigl((5\Log(d_A))^{2/(4-p^*\vee q)}+\sqrt{\Log(mn)}+\sqrt{\Log d_A}d_A^{1/2+3\ve}\Bigl(\frac rl\Bigr)^{1/p}\Bigr)\max_{i,j}|a_{ij}| \Bigr].
\end{align*}
\end{lem}

\begin{proof}
Let us fix  $I\subset [m]$, $I_0\subset I$ with $|I_0|=r$, $J\in \mathcal{J}_4(l)$,
$s\in K(q^*,m,I,d_A^\ve)$, and $t\in K(p,n,J,d_A^\ve)$.
Let $I_{0,1},\ldots,I_{0,V}$ be $4$-connected componets of $I_0$. Then $(I_{0,1}',\ldots,I_{0,V}')$ is a partition of $I_0'$, and $(I_{0,1}'',\ldots,I_{0,V}'')$ is a partition of $I_0''$. Hence,
\[
\sum_{\substack{i\in I_0''\cap I\\ j\in I_0'\cap J}}a_{ij}g_{ij}s_it_j
=\sum_{v=1}^{ V} \sum_{\substack{i\in I_{0,v}''\cap I\\ j\in I_{0,v}'\cap J}}a_{ij}g_{ij}s_it_j.
\]

Let 
\[
\mathcal{V}=\Bigl\{v\leq V\colon |I_{0,v}'\cap J|< d_A^{-2\ve p}\frac{l}{r}|I_{0,v}|
= d_A^{-2\ve p}\frac{|J|}{|I_0|}|I_{0,v}|\Bigr\}.
\]
Then
\begin{align*}
\sum_{v\in \mathcal{V}}\sum_{\substack{i\in I_{0,v}''\cap I\\ j\in I_{0,v}'\cap J}}  a_{ij}g_{ij}s_it_j
&\leq \|G_A\|_{p\to q}\sum_{v\in \mathcal{V}}\|(s_i)_{i\in I_{0,v}''\cap I}\|_{q^*}
\|(t_j)_{j\in I_{0,v}'\cap J}\|_{p}
\\
&\leq \|G_A\|_{p\to q}\Bigl(\sum_{v\in \mathcal{V}}\|(s_i)_{i\in I_{0,v}''\cap I}\|_{q^*}^2\Bigr)^{1/2}
\Bigl(\sum_{v\in \mathcal{V}}\|(t_j)_{j\in I_{0,v}'\cap J}\|_{p}^2\Bigr)^{1/2}
\\
&\leq \|G_A\|_{p\to q}\Bigl(\sum_{v\in \mathcal{V}}\|(s_i)_{i\in I_{0,v}''\cap I}\|_{q^*}^{q^*}\Bigr)^{1/q^*}
\Bigl(\sum_{v\in \mathcal{V}}\|(t_j)_{j\in I_{0,v}'\cap J}\|_{p}^p\Bigr)^{1/p}
\\
&\leq \|G_A\|_{p\to q}\|s\|_{q^*}
\Bigl(\sum_{v\in \mathcal{V}}\frac{d_{A}^{\ve p}}{|J|}|I_{0,v}'\cap J|\Bigr)^{1/p}
\\
&\leq d_A^{-\ve}\|G_A\|_{p\to q}\Bigl(\sum_{v\in \mathcal{V}}\frac{1}{|I_0|}|I_{0,v}|\Bigr)^{1/p}
\leq d_A^{-\ve}\|G_A\|_{p\to q}.
\end{align*}

For a nonempty $ I_0\subset [m]$, $1\leq u\leq n$, and $\gamma \ge 1$ define
\begin{multline*}
Z_{I_0,u}(\gamma)\\ \coloneqq \max_{I_0\subset I\subset [m]}\max_{J\subset [n], |J\cap I_0'|\geq u}\
\sup_{s\in K(q^*,m,I,\gamma)}\ \sup_{t\in K(p,n,J,\gamma)}
\frac{\sum_{i\in I_0''\cap I,j\in I_0'\cap J}a_{ij}g_{ij}s_it_j}
{\|(s_i)_{i\in I_0''\cap I}\|_{q^*}\|(t_j)_{j\in I_0'\cap J}\|_p}.
\end{multline*}
Moreover,  for $1\leq r\leq m$,  $1\leq u\leq n$, and $\gamma\ge 1$, set
\[
Z_{r,u}(\gamma)\coloneqq \max_{I_0\in \mathcal{I}_4(r)}Z_{I_0,u}(\gamma),
\]
and for $\alpha\in (0,1]$ and $\gamma\ge 1$, let
\[
Z_{\alpha}(\gamma)\coloneqq \max_{1\leq r\leq m}\max_{u\geq \alpha r}Z_{r,u}(\gamma).
\]
Observe that for $v\notin \mathcal{V}$, $|I_{0,v}'\cap J|\geq \alpha|I_{0,v}|$, where 
$\alpha\coloneqq d_A^{-2\ve p}\frac{l}{r}\geq d_A^{-2}$, so 
\begin{align*}
\sum_{v\notin \mathcal{V}}\sum_{\substack{i\in I_{0,v}''\cap I\\ j\in I_{0,v}'\cap J}}  a_{ij}g_{ij}s_it_j
& \leq Z_\alpha(d_A^\ve)\sum_{v\notin \mathcal{V}}\|(s_i)_{i\in I_{0,v}''\cap I}\|_{q^*}
\|(t_j)_{j\in I_{0,v}'\cap J}\|_{p}
\end{align*}
\begin{align*}
&\leq Z_\alpha(d_A^\ve)\Bigl(\sum_{v}\|(s_i)_{i\in I_{0,v}''\cap I}\|_{q^*}^2\Bigr)^{1/2}
\Bigl(\sum_{v}\|(t_j)_{j\in I_{0,v}'\cap J}\|_{p}^2\Bigr)^{1/2}
\\
&\leq Z_\alpha(d_A^\ve)\Bigl(\sum_{v}\|(s_i)_{i\in I_{0,v}''\cap I}\|_{q^*}^{q^*}\Bigr)^{1/q^*}
\Bigl(\sum_{v}\|(t_j)_{j\in I_{0,v}'\cap J}\|_{p}^p\Bigr)^{1/p}
\\
&\leq Z_\alpha(d_A^\ve)\|s\|_{q^*}\|t\|_p\leq Z_\alpha(d_A^\ve).
\end{align*}

The above argument shows that
\begin{align*}
\max_{I \subset [m]}\max_{I_0\subset I,|I_0|=r}\max_{J\in \mathcal{J}_4(l)}\
\sup_{s\in K(q^*,m,I,d_A^\ve)}\ \sup_{t\in K(p,n,J,d_A^\ve)}
\sum_{i\in I_0''\cap I,j\in I_0'\cap J}a_{ij}g_{ij}s_it_j
\\
\leq d_A^{-\ve}\|G_A\|_{p\to q}+Z_\alpha(d_A^\ve).
\end{align*}

Corollary \ref{cor:logdA} and estimate~\eqref{eq:sup-eps} yield   that
\begin{align*}
d_A^{-\ve}\Ex\|G_A\|_{p\to q} 
&\lesssim 
\sup_{x\ge 1} (x^{-\ve} \Log x)\Bigl( D_1+D_2+\sqrt{\Log(mn)}\max_{i,j}|a_{ij}| \Bigr)
\\ 
& \sim 
\ve^{-1}\Bigl( D_1+D_2+\sqrt{\Log(mn)}\max_{i,j}|a_{ij}| \Bigr).
\end{align*}
Therefore, the assertion follows from part iii) of the next lemma.
\end{proof}

\begin{lem}
\label{lem:fixedI0}
Let  $\gamma\ge 1$ and $ p^*,q\in [2,4)$.
\\
i) If $I_0\subset [m]$ and $1\leq u\leq n$ satisfy $u\geq |I_0|d_A^{-2}$, then
\[
\Ex Z_{I_0,u}(\gamma)
\lesssim D_1+D_2+(3\Log(d_A\gamma))^{2/(4-p^*\vee q)}\max_{i,j}|a_{ij}|.
\]
ii) If $1\leq r\leq m$ and  $1\leq u\leq n$ satisfy $u\geq rd_A^{-2}$, then
\begin{multline*}
\Ex Z_{r,u}(\gamma)\lesssim D_1+D_2 
\\
 +
\Bigl((3\Log(d_A\gamma))^{2/(4-p^*\vee q)}+\sqrt{\Log m}+
\sqrt{r\Log d_A}\, \gamma d_A^{1/2}u^{-1/p}\Bigr)\max_{i,j}|a_{ij}|.
\end{multline*}
iii) If $\alpha\geq d_A^{-2}$, then
\begin{multline*}
\Ex Z_\alpha(\gamma)
\lesssim  D_1+D_2 
\\
\quad +
\Bigl((3\Log(d_A\gamma))^{2/(4-p^*\vee q)}+\sqrt{\Log(mn)}+\sqrt{\Log d_A}\, \gamma d_A^{1/2}\alpha^{-1/p}\Bigr)\max_{i,j}|a_{ij}|.
\end{multline*}
\end{lem}

\begin{proof}
Let us fix $I_0\subset I\subset [m]$, $J\subset [n]$ with $|I_0|=r$, $|J\cap I_0'|\geq u$, 
$s\in K(q^*,m,I,\gamma)$
and $t\in K(p,n,J,\gamma)$. Define
\[
\bar{s}=\|(s_i)_{i\in I_0''\cap I}\|_{q^*}^{-1}(s_i)_{i\in I_0''\cap I},\quad
\bar{t}=\|(t_j)_{j\in I_0'\cap J}\|_{p}^{-1}(t_j)_{j\in I_0'\cap J}.
\]
Recall that $I_0\subset I_0''$, so $\|\bar{s}\|_{q^*} =1$,  
$\|\bar{s}\|_{\infty}\leq \gamma|I_0''\cap I|^{-1/q^*}\leq \gamma |I_0|^{-1/q^*} 
\leq \gamma r^{-1/q^*}$,
$\|\bar{t}\|_{p} =1$, and $\|\bar{t}\|_{\infty}\leq \gamma|I_0'\cap J|^{-1/p}\leq \gamma u^{-1/p}
\leq \gamma d_A^{2/p}r^{-1/p}$.
Hence,  part i) of the assertion  follows from 
Proposition \ref{prop:slepianp*qbelow4gen} applied with   the matrix $(a_{ij})_{i\in I_0'',j\in I_0'}$,
$m= |I_0''|\leq d_A^2r$, 
$n= |I_0'|\leq d_Ar$, $a=(\gamma r^{-1/q^*})\wedge 1$ 
and $b=( \gamma d_A^{2/p}r^{-1/p})\wedge 1$.

To show part ii) observe that 
\begin{align*}
\sum_{i\in I_0''\cap I,j\in I_0'\cap J}a_{ij}^2\bar{s}_i^2\bar{t}_j^2
&\leq \max_{i\in I}\sum_{j\in J}a_{ij}^2\bar{t}_j^2
\leq \max_{i,j}a_{ij}^2\min\{1,d_A\|\bar{t}\|_\infty^2\}
\\
& \leq \max_{i,j}a_{ij}^2\min\{1,d_A\gamma^2u^{-2/p}\}.
\end{align*}
Moreover, estimate \eqref{eq:size-Ir(k)} yields $\sqrt{\Log |\mathcal{I}_4(r)|} \lesssim \sqrt{\Log m}+\sqrt{r\Log d_A}$,
so  part ii) follows from part i) and estimate~\eqref{eq:supgauss}.

Part iii) easily follows from ii) and another application of \eqref{eq:supgauss}.
\end{proof}

\begin{cor}
\label{cor:Ykl}
If $1\le k\le m$, $1\le l \le n$, $p^*,q\in[2,4)$, and $\ve\in(0,1/2]$, then
\begin{align*}
\Ex Y_{k,l}(d_A^\ve) 
&\lesssim
\ve^{-1}\Bigl[ D_1+D_2
\\
&+\Bigl(\sqrt{\Log(mn)}
+ \Bigl(\ve^{-1/2}+\Bigl(\frac{  40}{4-p^*\vee q}\Bigr)^{2/(4-p^*\vee q)}\Bigr)
d_A^{1/4+4\ve}\Bigr) 
\max_{i,j}|a_{ij}| \Bigr].
\end{align*}
\end{cor}

\begin{proof}
By symmetry we may assume that $l\geq k$.

First assume that $k\geq d_A^{1/2}$.
Lemmas \ref{lem:decomp}--\ref{lem:I_0r} yield that for every $1\leq r\leq k$, 
\begin{align*}
\Ex Y_{k,l}(d_A^\ve)
& 
\lesssim \ve^{-1} \Bigr[ D_1+D_2
\\
& \quad+\Bigl(( 5\Log(d_A))^{2/(4-p^*\vee q)}+\sqrt{\Log(mn)}\Bigl)\max_{i,j}|a_{ij}|
\\ 
&\quad
 +\Bigl(\sqrt{\Log d_A}d_A^{1/2+3\ve}\Bigl(\frac{r}{l}\Bigr)^{1/p}
+\sqrt{k\Log d_A}\, d_A^{\ve}l^{1/2-1/p}r^{-1/2}\Bigr)\max_{i,j}|a_{ij}| \Bigr].
\end{align*}
Moreover, estimate~\eqref{eq:sup-eps} implies
\begin{gather} 
\notag
( 5\Log(d_A))^{2/(4-p^*\vee q)} 
 \leq   \Bigl(\frac{  40}{4-p^*\vee q}\Bigr)^{2/(4-p^*\vee q)}d_A^{1/4},
\\
\label{eq:log-vs-power} 
\sqrt{\Log d_A} \lesssim \ve^{-1/2} d_A^{\ve},
\end{gather}
and
\begin{align*}
\inf_{1\leq r\leq k} \Bigl(d_A^{1/2+3\ve}\Bigl(\frac{r}{l}\Bigr)^{1/p}+\sqrt{k}d_A^{\ve}l^{1/2-1/p}r^{-1/2}\Bigr)
&\leq 
\inf_{1\leq r\leq k} \Bigl( d_A^{1/2+3\ve}\Bigl(\frac{r}{k}\Bigr)^{1/2}+\sqrt{k}d_A^\ve r^{-1/2}\Bigr)
\\ &\lesssim d_A^{1/4+ 3\ve},
\end{align*}
where the last estimate follows by taking $r=\lfloor kd_A^{-1/2}\rfloor \in [1,k]$.

Now assume that $k<d_A^{1/2}$. Recall that for a fixed nonempty set $I$,
\[
X_I=X_I(A,p,q)\coloneqq \|(a_{ij}g_{ij})_{i\in I,j \in [n]}\|_{p\to q}.
\]
Note that
\[
Y_{k,l}(d_A^\ve)\leq \max_{I\in \mathcal{I}_4(k)}X_I.
\]
Therefore,  estimates \eqref{eq:supgauss}, \eqref{eq:supoverI}, and $|\mathcal{I}_4(k)|\leq m4^kd_{A}^{4k}$ (see \eqref{eq:size-Ir(k)}) yield
\begin{align}
\notag
\Ex Y_{k,l}(d_A^\ve)
&\lesssim
\max_{I\in \mathcal{I}_4(k)}\Ex X_I
+\sqrt{\Log|\mathcal{I}_4(k)|}\max_{I\in \mathcal{I}_4(k)}
\sup_{\|s\|_{q^*}\leq 1}\sup_{\|t\|_p\leq 1}\Bigl(\sum_{i,j}a_{ij}^2s_i^2t_j^2\Bigr)^{1/2}.
\\
\label{eq:estyklsmallk}
&\lesssim D_1+
( \sqrt{k} +\sqrt{\Log m}+\sqrt{k\Log d_A})\max_{i,j}|a_{ij}|.
\end{align}
The assertion follows easily by the assumption that $k<d_A^{1/2}$ and  
estimate \eqref{eq:log-vs-power}.
\end{proof}

\begin{proof}[Proof of Proposition~\ref{prop:pqdA}]
We apply Proposition \ref{prop:redtosupflat} and Corollary \ref{cor:Ykl} (with $\ve/4$ instead of $\ve$).
\end{proof}

\textbf{Acknowledgements.} Part of this work was carried out while the first-named author was visiting the
Hausdorff Research Institute for Mathematics, Univerity of Bonn. The hospitality of HIM and
of the organizers of the program \textit{Boolean Analysis in Computer Science} is gratefully acknowledged.

\vspace{0.4cm}

  \bibliographystyle{amsplain}
  \bibliography{matrices_pqGauss}

\end{document}